\documentclass[11pt,reqno]{amsart}
\usepackage[letterpaper,margin=1in,footskip=0.25in]{geometry}
\usepackage{microtype}
\usepackage{amssymb}
\usepackage{bm}
\usepackage{mathrsfs}
\usepackage[mathic=true]{mathtools}
\usepackage{enumitem}
\usepackage[noadjust]{cite}
\renewcommand{\citepunct}{;\penalty\citemidpenalty\ }

\usepackage{tikz-cd}

\PassOptionsToPackage{pdfusetitle,colorlinks,pagebackref,bookmarksdepth=3,linktoc=all,pdfencoding=auto}{hyperref}
\usepackage{bookmark}
\hypersetup{citecolor=[HTML]{2E7E2A},linkcolor=[HTML]{800006},urlcolor=[HTML]{8A0087}}
\usepackage[hyphenbreaks]{breakurl}

\numberwithin{equation}{subsection}

\newtheorem{theorem}[equation]{Theorem}
\newtheorem{corollary}[equation]{Corollary}
\newtheorem{lemma}[equation]{Lemma}
\newtheorem{proposition}[equation]{Proposition}

\newtheorem{innercustomthm}{Theorem}
\newenvironment{customthm}[1]
  {\renewcommand\theinnercustomthm{#1}\innercustomthm}
  {\endinnercustomthm}

\theoremstyle{definition}
\newtheorem{definition}[equation]{Definition}
\newtheorem{examples}[equation]{Examples}

\newtheoremstyle{citeddef}{.5\baselineskip\@plus.2\baselineskip\@minus.2\baselineskip}{.5\baselineskip\@plus.2\baselineskip\@minus.2\baselineskip}{}{}{\bfseries}{\bfseries .}{5pt plus 1pt minus 1pt}{\thmname{#1}\thmnumber{ #2}\thmnote{ \normalfont#3}}
\theoremstyle{citeddef}
\newtheorem{citeddef}[equation]{Definition}
\newtheorem{citedex}[equation]{Example}

\theoremstyle{remark}
\newtheorem{remark}[equation]{Remark}

\newtheoremstyle{step}{.25\baselineskip\@plus.1\baselineskip\@minus.1\baselineskip}{.25\baselineskip\@plus.1\baselineskip\@minus.1\baselineskip}{\itshape}{}{\bfseries}{\bfseries .}{5pt plus 1pt minus 1pt}{\thmname{#1}\thmnumber{ #2}\thmnote{ \normalfont(#3)}}
\theoremstyle{step}
\newtheorem{step}{Step}[subsection]

\newtheoremstyle{substep}{.25\baselineskip\@plus.1\baselineskip\@minus.1\baselineskip}{.25\baselineskip\@plus.1\baselineskip\@minus.1\baselineskip}{\itshape}{}{\itshape}{\itshape .}{5pt plus 1pt minus 1pt}{\thmname{#1}\thmnumber{ \normalfont#2}\thmnote{ \normalfont(#3)}}
\theoremstyle{substep}
\newtheorem{claim}[equation]{Claim}

\DeclareMathOperator{\Ass}{Ass}
\DeclareMathOperator{\C}{\mathit{C}}
\DeclareMathOperator{\CF}{\mathit{CF}}
\DeclareMathOperator{\D}{\mathit{D}}
\DeclareMathOperator{\Diag}{Diag}
\DeclareMathOperator{\Codiag}{Codiag}
\DeclareMathOperator{\Coreal}{Coreal}
\DeclareMathOperator{\Exc}{Exc}
\DeclareMathOperator{\Gr}{Gr}
\DeclareMathOperator{\Ho}{Ho}
\DeclareMathOperator{\Hom}{Hom}
\DeclareMathOperator{\Hrc}{\mathbf{Hrc}}
\DeclareMathOperator{\Mor}{Mor}
\DeclareMathOperator{\Ob}{Ob}
\DeclareMathOperator{\PProj}{\mathbf{Proj}}
\DeclareMathOperator{\Real}{Real}

\DeclareMathOperator{\Res}{Res}
\DeclareMathOperator{\Sch}{Sch}
\DeclareMathOperator{\Sh}{Sh}
\DeclareMathOperator{\Sp}{\mathbf{Sp}}
\DeclareMathOperator{\Spa}{Spa}
\DeclareMathOperator{\Spec}{Spec}
\DeclareMathOperator{\Supp}{Supp}
\DeclareMathOperator{\Top}{Top}
\DeclareMathOperator{\Tor}{Tor}
\DeclareMathOperator{\ZR}{ZR}
\DeclareMathOperator{\rd}{rd}
\DeclareMathOperator{\tot}{tot}
\DeclareMathOperator*{\RRvarprojlim}{\underleftarrow{\mathbf{R}\kern0.5pt\mathrm{lim}}}

\newcommand{\TwoCat}{\mathbf{2Cat}}
\newcommand{\Ab}{\mathbf{Ab}}
\newcommand{\AU}{\mathrm{AU}}
\newcommand{\Cat}{\mathbf{Cat}}
\newcommand{\DGA}{\mathbf{DGA}}
\newcommand{\FF}{\mathrm{F}}
\newcommand{\MHC}{\mathbf{MHC}}
\newcommand{\MHD}{\mathbf{MHD}}
\newcommand{\Mod}{\mathrm{Mod}}
\newcommand{\Sets}{\mathbf{Sets}}
\newcommand{\an}{\mathrm{an}}
\newcommand{\coh}{\mathrm{coh}}
\newcommand{\comp}{\mathrm{comp}}
\newcommand{\diff}{\mathrm{diff}}
\newcommand{\dR}{\mathrm{dR}}
\newcommand{\et}{\mathrm{\acute{e}t}}
\newcommand{\id}{\mathrm{id}}
\newcommand{\lisse}{\mathrm{lis}}
\newcommand{\locproj}{\mathrm{locproj}}
\newcommand{\op}{\mathrm{op}}
\newcommand{\proet}{\mathrm{pro\acute{e}t}}
\newcommand{\proj}{\mathrm{proj}}
\newcommand{\proper}{\mathrm{proper}}

\newcommand{\red}{\mathrm{red}}
\newcommand{\reg}{\mathrm{reg}}
\newcommand{\typ}{\mathrm{typ}}

\newcommand{\bA}{\mathbf{A}}
\newcommand{\BB}{\mathbf{B}}
\newcommand{\CC}{\mathbf{C}}

\newcommand{\bP}{\mathbf{P}}
\newcommand{\QQ}{\mathbf{Q}}
\newcommand{\RR}{\mathbf{R}}
\newcommand{\ZZ}{\mathbf{Z}}

\newcommand{\bfs}{\mathbf{s}}
\newcommand{\cE}{\mathcal{E}}
\newcommand{\cI}{\mathcal{I}}
\newcommand{\cJ}{\mathcal{J}}
\newcommand{\cM}{\mathcal{M}}
\newcommand{\cO}{\mathcal{O}}
\newcommand{\cW}{\mathcal{W}}
\newcommand{\cX}{\mathcal{X}}
\newcommand{\fI}{\mathfrak{I}}
\newcommand{\fm}{\mathfrak{m}}
\newcommand{\fp}{\mathfrak{p}}
\newcommand{\sA}{\mathscr{A}}
\newcommand{\sC}{\mathscr{C}}
\newcommand{\sD}{\mathscr{D}}
\newcommand{\sF}{\mathscr{F}}
\newcommand{\sG}{\mathscr{G}}
\newcommand{\sH}{\mathscr{H}}
\newcommand{\sK}{\mathscr{K}}
\newcommand{\sL}{\mathscr{L}}
\newcommand{\sM}{\mathscr{M}}

\mathchardef\mhyphen="2D
\def\hyph{-\penalty0\hskip0pt\relax}

\DeclarePairedDelimiter\abs{\lvert}{\rvert}

\providecommand\given{}
\newcommand\SetSymbol[1][]{\nonscript\:#1\vert\allowbreak\nonscript\:\mathopen{}}
\DeclarePairedDelimiterX\Set[1]\{\}{\renewcommand\given{\SetSymbol[\delimsize]}#1}

\newcommand{\hooklongrightarrow}{\lhook\joinrel\longrightarrow}

\begin{document}
\title[Injectivity theorems and cubical descent]{Injectivity theorems
and cubical descent\\for schemes, stacks, and analytic
spaces}
\author{Takumi Murayama}
\address{Department of Mathematics\\Purdue University\\150 N. University
Street\\West Lafayette, IN 47907-2067\\USA}
\email{\href{mailto:murayama@purdue.edu}{murayama@purdue.edu}}
\urladdr{\url{https://www.math.purdue.edu/~murayama/}}

\thanks{This material is based upon work supported by the National Science
Foundation under Grant No.\ DMS-2201251.}


\makeatletter
  \hypersetup{
    pdftitle={Injectivity theorems and cubical descent for schemes, stacks, and analytic spaces},
    pdfsubject=\@subjclass,pdfkeywords=\@keywords
  }
\makeatother

\begin{abstract}
  We prove relative injectivity, torsion-freeness, and vanishing theorems for
  generalized normal crossing pairs on schemes, algebraic stacks, formal schemes,
  semianalytic germs of complex analytic spaces, rigid analytic
  spaces, Berkovich spaces, and adic spaces locally of weakly finite type over a
  field, all in equal characteristic zero.
  We give a uniform proof for all these theorems
  in all the categories of spaces mentioned above, which were previously only known
  for varieties and complex analytic spaces due to work of Ambro and
  Fujino.
  Ambro and Fujino's results are integral in the proofs of
  the fundamental theorems of the
  minimal model program for (semi-)log canonical pairs and the theory of
  quasi-log structures.
  Our results resolve a significant barrier to extending these results on
  (semi-)log canonical pairs and quasi-log structures beyond the setting of
  varieties and complex analytic spaces.
  \par In order to prove our most general injectivity theorems, we
  generalize to all these categories of spaces
  a criterion due to Guill\'en and Navarro Aznar characterizing when
  functors defined on smooth varieties extend to all varieties.
  This extension result uses cubical hyperresolutions, which
  we construct in all categories of spaces mentioned above.
  Our extension result is very general and is of independent interest.
  We use this extension result to prove our injectivity theorems for
  generalized normal crossing pairs.
  We also apply our extension result to
  develop the theoretical foundations for the Deligne--Du
  Bois complex in these categories of spaces and to construct a
  weight filtration on the (pro-)\'etale cohomology of schemes and rigid
  analytic spaces.
  These results establish some aspects of Deligne--Hodge theory in all
  categories of spaces mentioned above.
\end{abstract}

\maketitle
\setcounter{tocdepth}{1}
{\hypersetup{hidelinks}\tableofcontents}
\newpage

\section{Introduction}\label{sect:intro}
\subsection{Background}
Let \(f\colon X \to Y\) be a locally projective morphism of spaces of equal
characteristic zero in one of the following
categories:
\begin{enumerate}[label=\((\textup{\Roman*})\),ref=\textup{\Roman*}]
  \item[\((0)\)]
  \makeatletter
  \protected@edef\@currentlabel{0}
  \phantomsection
  \label{setup:introalgebraicspaces}
  \makeatother
    Decent locally Noetherian algebraic stacks over a scheme \(S\).
  \item\label{setup:introformalqschemes}
    Locally Noetherian formal schemes.
  \item\label{setup:introcomplexanalyticgerms}
    Semianalytic germs of complex analytic spaces.
  \item\label{setup:introberkovichspaces}
    Berkovich \(k\)-analytic spaces
    over a complete non-Archimedean field \(k\).
  \makeatletter
  \item[{\((\ref*{setup:introberkovichspaces}')\)}]
  \protected@edef\@currentlabel{\ref*{setup:introberkovichspaces}'}
  \phantomsection
  \label{setup:introrigidanalyticspaces}
  \makeatother
    Rigid \(k\)-analytic spaces over a complete non-trivially valued
    non-Archimedean field \(k\).
  \item\label{setup:introadicspaces}
    Adic spaces locally of weakly finite type over a complete
    non-trivially valued non\hyph{}Archimedean field \(k\).
\end{enumerate}
Even if one's primary interest is smooth complex projective varieties, there are
compelling reasons to consider morphisms \(f\) between spaces of the form
\((\ref{setup:introalgebraicspaces})\textnormal{--}(\ref{setup:introadicspaces})\).
For example, moduli spaces are often not schemes or even algebraic spaces.
Studying deformations of varieties sometimes requires working with formal
schemes and their birational geometry (see, e.g., \cite{AAB}).
Working over formal power series rings has been useful in studying
Shokurov's conjectures on the ascending chain condition for invariants of
singularities like the log canonical threshold \cite{dFEM10,dFEM11} and the minimal
log discrepancy \cite{Kaw15}.
Even Hironaka's proof that complex varieties have resolutions of singularities
requires leaving the world of complex varieties \cite{Hir64}.
See \citeleft\citen{Mur}\citemid p.\ 2\citepunct \citen{LM}\citemid p.\
4\citeright\ for more discussion.\medskip
\par Until recently, two extremely useful techniques originally developed for
complex varieties
were unavailable for morphisms \(f\) as above:
\begin{enumerate}[label=\((\arabic*)\),ref=\arabic*]
  \item\label{intro:kvquestion} Relative 
    Kawamata--Viehweg-type vanishing theorems \cite{Kaw82,Vie82,KMM87}.
  \item\label{intro:mmpquestion} The relative minimal model program with scaling
    for klt pairs \cite{BCHM10,HM10}.
\end{enumerate}
For \((\ref{intro:kvquestion})\), the case of varieties in case
\((\ref{setup:introalgebraicspaces})\) was shown by Kawamata \cite{Kaw82},
Viehweg \cite{Vie82}, and Kawamata, Matsuda, and Matsuki \cite{KMM87}.
Case \((\ref{setup:introcomplexanalyticgerms})\) was shown by Nakayama \cite{Nak87}.
For \((\ref{intro:mmpquestion})\), the case of varieties in case
\((\ref{setup:introalgebraicspaces})\) was shown by
Birkar, Cascini, Hacon, and M\textsuperscript{c}Kernan \cite{BCHM10,HM10}.
The case of algebraic spaces of finite type
over a field of
characteristic zero in case
\((\ref{setup:introalgebraicspaces})\) was shown by Vilallobos-Paz \cite{VP} and 
case \((\ref{setup:introcomplexanalyticgerms})\) was shown by
Fujino \cite{Fuj} (see also \cite{DHP,LM}).
For the rest of the categories
\((\ref{setup:introalgebraicspaces})\textnormal{--}(\ref{setup:introadicspaces})\),
the author of the present paper resolved \((\ref{intro:kvquestion})\) in
\cite{Mur} and resolved \((\ref{intro:mmpquestion})\) in joint work with Lyu
\cite{LM}, where in
case \((\ref{setup:introalgebraicspaces})\), we restricted to
the subcategory of algebraic spaces.\medskip
\par For complex varieties,
advances in the minimal model program for log canonical and semi-log canonical pairs
require injectivity theorems due to Ambro \cite{Amb03,Amb14,Amb20},
Fujino \cite{Fuj04,Fuj11,Fuj16,Fuj17}, and Fujino and Fujisawa \cite{FF14},
which are strong generalizations of
Kawamata--Viehweg-type vanishing theorems.
Working with log canonical and semi-log canonical pairs is
critical in algebraic geometry, especially in moduli theory, where introducing
semi-log canonical degenerations is necessary to obtain compact moduli
spaces \cite{KSB88,Ale96,Kol23}.
Existing proofs of these injectivity theorems rely on Hodge theory or
complex analysis.
None of these proofs can be adapted to arbitrary spaces in the categories
\((\ref{setup:introalgebraicspaces})\textnormal{--}(\ref{setup:introadicspaces})\).
Thus, the lack of injectivity theorems has been a significant barrier to
extending results on (semi-)log canonical pairs and quasi-log structures beyond
the setting of varieties and complex analytic spaces.
\subsection{Overview of this paper}
\par In this paper, our goal is to prove the injectivity theorems
mentioned above in the categories
\((\ref{setup:introalgebraicspaces})\textnormal{--}(\ref{setup:introadicspaces})\).
For varieties and complex analytic spaces, these injectivity theorems have proved
critical for advancing the minimal model program for log canonical and
semi-log canonical pairs.
Proving these injectivity theorems in full generality requires difficult
techniques that are of independent interest and have other
applications.\medskip
\par Our main results are as follows (all in equal characteristic zero):
\begin{enumerate}[label=\((\arabic*)\)]
  \item The injectivity, torsion-freeness, and vanishing theorems for
    generalized normal crossing pairs in the categories
    \((\ref{setup:introalgebraicspaces})\textnormal{--}(\ref{setup:introadicspaces})\)
    (Theorem \ref{thm:intromaininj}).
  \item
    The construction of cubical hyperresolutions in the categories
    \((\ref{setup:introalgebraicspaces})\textnormal{--}(\ref{setup:introadicspaces})\)
    (Theorem \ref{thm:introcubicalhyperres}).
  \item
    An extension criterion for functors defined on regular spaces, generalizing
    a result of Guill\'en and Navarro Aznar proved for varieties to all
    categories 
    \((\ref{setup:introalgebraicspaces})\textnormal{--}(\ref{setup:introadicspaces})\).
  \item 
    The construction of the Deligne--Du Bois complex and the basic theory of Du
    Bois singularities in the categories
    \((\ref{setup:introalgebraicspaces})\textnormal{--}(\ref{setup:introadicspaces})\)
    (\S\ref{sect:delignedubois}).
  \item
    The construction of a weight filtration on the (pro-)\'etale cohomology of
    schemes and compactifiable rigid analytic spaces (Theorem
    \ref{thm:weightfiltration}).
\end{enumerate}
A common theme throughout the paper is that we must build suitable replacements
for Hodge theory.
We also do not have a direct analogue of the degeneration of the Hodge-to-de Rham
spectral sequence available because our spaces may not be proper over a
field or a point.
As a consequence of Theorems \ref{thm:introcubicalhyperres} and
\ref{thm:introgna}, we are able to prove
some aspects of Hodge theory in all the categories
\((\ref{setup:introalgebraicspaces})\textnormal{--}(\ref{setup:introadicspaces})\)
listed above.
See Remark \ref{rem:somehodge} and \cite[(4.2)]{GNA02}.\medskip
\par We will discuss the results listed above in the remainder of this
introduction.
In the remainder of this introduction, we always work in equal characteristic
zero.
\subsection{Injectivity, torsion-freeness, and vanishing
theorems}\label{sect:intromaininj}
\par Our first main result is that injectivity theorems and related
torsion-freeness and vanishing theorems hold in all the categories
\((\ref{setup:introalgebraicspaces})\textnormal{--}(\ref{setup:introadicspaces})\)
mentioned above, generalizing the results for complex varieties and complex
analytic spaces that have proved to be indispensable in the
theory of (semi-)log canonical pairs and quasi-log structures for complex
varieties \cite{Amb03,Fuj11,Bir12,Fuj14,Has20} and complex analytic spaces
\cite{Fujconecontr,Fujql,EH}.
Our theorems therefore resolve a substantial roadblock for studying the birational
geometry of compact moduli spaces beyond the case of varieties or complex
analytic spaces.
For complex varieties (beyond the smooth or klt settings),
these injectivity, torsion-freeness, and vanishing theorems
are due to Kawamata \cite{Kaw85} (when \(B\) below has coefficients in \((0,1)\)),
Ambro \cite{Amb03,Amb14,Amb20}, Fujino
\cite{Fuj04,Fuj11,Fuj16,Fuj17}, and Fujino and Fujisawa \cite{FF14}.
In the complex analytic setting, the special case of this 
result for simple normal crossing pairs is due to Fujino
\cite{Fujvan}.
\par See \S\S\ref{sect:wlc}--\ref{sect:gnc} for the definitions of
\textsl{generalized normal crossings log pairs} and of \(\omega^{[r]}_{(X,B)}\)
needed for the most general statements below.
For simplicity, one can consider the special case of the statement below where
\((X,B)\) is a pair consisting of a regular space \(X\) and a \(\QQ\)-divisor
\(B\) with normal crossing support and coefficients in \([0,1]\), in which case
\[
  \omega^{[r]}_{(X,B)} = \cO_X\bigl(r(K_X+B)\bigr)
\]
for all sufficiently divisible \(r\).
The nomenclature for the statements
\((\ref{thm:intromaininjesvi})\textnormal{--}(\ref{thm:intromaininjohskol})\)
below follows \cite{Amb20}.
\begin{customthm}{\ref{thm:maininj}}\label{thm:intromaininj}
  Let \(f\colon (X,B) \to Y\) be a locally projective morphism of spaces of one
  of the types
  \((\ref{setup:introalgebraicspaces})\textnormal{--}(\ref{setup:introadicspaces})\)
  such that \((X,B)\) is a generalized normal crossing pair.
  In case \((\ref{setup:introalgebraicspaces})\) (resp.\
  \((\ref{setup:introformalqschemes})\)), suppose that \(Y\) is locally
  quasi-excellent of equal characteristic zero and has a dualizing complex
  (resp.\ \(c\)-dualizing complex) \(\omega_Y^\bullet\).
  In cases \((\ref{setup:introberkovichspaces})\),
  \((\ref{setup:introrigidanalyticspaces})\), and
  \((\ref{setup:introadicspaces})\), suppose that \(k\) is of characteristic
  zero.
  Let \(\sL\) be an invertible \(\cO_X\)-module.
  \begin{enumerate}[label=\((\roman*)\),ref=\roman*]
    \item\label{thm:intromaininjesvi}
      \textnormal{(The Esnault--Viehweg injectivity theorem)}
      Suppose that 
      \[
        \sL^{\otimes r} \simeq \omega^{[r]}_{(X,B)}
      \]
      for an integer \(r \ge 1\) such that \(rB\) has
      integer coefficients.
      Let \(D\) be an effective Cartier divisor with support in \(\Supp(B)\).
      Then, the natural map
      \[
        R^if_*\sL \longrightarrow R^if_*\bigl(\sL(D)\bigr)
      \]
      is injective for every \(i\).
    \item\label{thm:intromaininjtankol}
      \textnormal{(The Tankeev--Koll\'ar injectivity theorem)}
      Suppose that
      \[
        \sL^{\otimes r} \simeq \omega^{[r]}_{(X,B)} \otimes_{\cO_X} \sH
      \]
      for an integer \(r \ge 1\) such that
      \(rB\) has integer coefficients and for an invertible \(\cO_X\)-module \(\sH\)
      such that \(f^*f_*\sH \to \sH\) is surjective.
      Let \(s \in \Gamma(X,\sH)\) be a global section that is invertible at the
      generic point of every log canonical center of \((X,B)\) and let \(D\) be the
      effective Cartier divisor defined by \(s\).
      Then, the natural maps
      \[
        R^i f_*\sL \longrightarrow R^if_*\bigl(\sL(D)\bigr)
      \]
      are injective for every \(i\).
    \item\label{thm:intromaininjtf} \textnormal{(The Koll\'ar torsion-freeness theorem)}
      Suppose that
      \[
        \sL^{\otimes r} \simeq \omega^{[r]}_{(X,B)}
        \otimes_{\cO_X} \sH
      \]
      for an integer \(r \ge 1\) such that
      \(rB\) has integer coefficients
      and for an invertible \(\cO_X\)-module
      \(\sH\) such that \(f^*f_*\sH \to \sH\) is surjective.
      In case \((\ref{setup:introalgebraicspaces})\), suppose that \(Y\) is an
      algebraic space.
      Then, every associated subspace of \(R^i f_*\sL\) is the \(f\)-image of an
      irreducible component of \(X\) or a log
      canonical center of \((X,B)\).
    \item\label{thm:intromaininjohskol}
      \textnormal{(The Ohsawa--Koll\'ar vanishing theorem)}
      Let \(g\colon Y \to Z\) be another locally projective morphism, where
      \(Z\) satisfies the same hypotheses as \(Y\) above.
      Suppose that
      \[
        \sL^{\otimes r} \simeq \omega^{[r]}_{(X,B)} \otimes_{\cO_X} f^*\sA
      \]
      for an integer \(r \ge 1\) such that
      \(rB\) has integer coefficients and \(\sA\) is a \(g\)-ample invertible
      \(\cO_Y\)-module.
      Then, we have
      \[
        R^pg_*R^qf_*\sL = 0
      \]
      for all \(q\) and for all \(p \ne 0\).
  \end{enumerate}
\end{customthm}
See \cite[Theorem 5.1\((a)\)]{EV92},
\citeleft\citen{Tan71}\citemid Proposition
1\citepunct \citen{Kol86}\citemid Theorem 2.2\citeright,
\cite[Theorem 2.1\((i)\)]{Kol86},
and \citeleft\citen{Ohs84}\citemid Theorem 3.1\citepunct \citen{Kol86}\citemid
Theorem 2.1\((iii)\)\citeright, respectively, for the original versions of
these statements that inspired the nomenclature for these results used above
that we have adopted from \cite{Amb20}.
In \cite{Fuj17,Fujvan}, Fujino refers to these statements as the Hodge-theoretic
injectivity theorem, the injectivity theorem, the torsion-freeness theorem or
the strict support condition, and the vanishing theorem, respectively.
\par As far as we are aware, the only known cases of Theorem \ref{thm:intromaininj}
outside of the smooth or klt case
are the case for complex varieties due to Ambro \cite{Amb20} and the case for
simple normal crossing pairs on complex analytic spaces due to Fujino
\cite{Fujvan}.
An earlier result of Kawamata \cite{Kaw85} proves the case when \(B\) above has
coefficients in \((0,1)\) using a slighly stronger definition for generalized
normal crossings.
Thus, the results for schemes and algebraic stacks that are not of finite type over a field
\((\ref{setup:introalgebraicspaces})\),
formal schemes \((\ref{setup:introformalqschemes})\),
and non-Archimedean analytic spaces
\((\ref{setup:introberkovichspaces})\),\((\ref{setup:introrigidanalyticspaces})\),\((\ref{setup:introadicspaces})\)
are completely new, and the results for
generalized normal crossing pairs on semianalytic germs of complex analytic
spaces \((\ref{setup:introcomplexanalyticgerms})\) generalize \cite{Fujvan}.
Moreover, our proof method gives a uniform proof of the injectivity,
torsion-freeness, and vanishing theorems in all the categories
\((\ref{setup:introalgebraicspaces})\textnormal{--}(\ref{setup:introadicspaces})\)
simultaneously that recovers the results from \cite{Amb20,Fujvan}.\medskip
\par As in our previous work \cite{Mur}, we first prove Theorem
\ref{thm:maininj} for schemes and then extend these results to the categories
\((\ref{setup:introalgebraicspaces})\textnormal{--}(\ref{setup:introadicspaces})\)
using the GAGA theorems from \cite{EGAIII1,Kop74,Ber93,Hub07,Poi10,AT19,LM}.
The proof of Theorem \ref{thm:intromaininj} for schemes
is a subtle application of relative Noetherian approximation \cite[\S8]{EGAIV3}
and our version of Grothendieck's limit theorem \cite[Expos\'e VI, Th\'eor\`eme
8.7.3]{SGA42} for local cohomology \cite[Theorem 3.13]{Mur}.
In \cite{Mur}, a surprising aspect of our proof was the appearance of the
\textsl{Zariski--Riemann space} \(\ZR(X)\), which is a locally ringed space but
is not a scheme when \(\dim(X) \ge 2\).
Because of the boundary divisors \(B\), we do not know of an
appropriate analogue for the Zariski--Riemann space, even if \(B\) is a reduced
simple normal crossing divisor.\medskip
\par In this paper, we instead work one element of local cohomology
at a time and construct boundary divisors that may depend on the element of
local cohomology chosen.
This strategy is enough to prove the case when \(X\) is regular and \(B\) has
normal crossing support in Theorem
\ref{thm:intromaininj}.
However, extending our results to the case when \(X\) is itself a normal
crossing divisor, or more generally when \((X,B)\) is a 
generalized normal crossing pair,
introduces a host of new difficulties.
We use simplicial resolutions to reduce to the normal crossing case, following
the approach in \cite{Kaw85,Amb20}.
Since simplicial resolutions are infinite and the theoretical foundations on the
Deligne--Du Bois complex \(\underline{\Omega}_X^\bullet\) obtained from such a
simplicial resolution in the categories
\((\ref{setup:introalgebraicspaces})\textnormal{--}(\ref{setup:introadicspaces})\)
do not exist in the literature, we prove the necessary foundational results in
Part \ref{part:cubical} of this paper.\medskip
\par As an application of Theorem \ref{thm:maininj}, in Theorem \ref{thm:ma18},
we extend Ma's
characterization \cite[Theorem 5.5]{Ma18}
of derived splinters of equal characteristic zero as the class
of \(\QQ\)-algebras satisfying Hochster and Huneke's vanishing conjecture for
maps of Tor \cite[Theorem 4.1 and (4.4)]{HH95} to the complete local case.
Ma's result applies to local rings essentially of finite type
over fields of characteristic zero.
Our proof also recovers Hochster and Huneke's vanishing theorem for maps of Tor
for all regular domains of equal characteristic zero \cite[Theorem 4.1]{HH95}.
In contrast with \cite{HH95}, this proof does not use reduction modulo \(p\).
See \cite[Remark 5.6(2)]{Ma18}.
\subsection{Cubical hyperresolutions and an extension criterion for functors
defined on regular spaces}
The second group of results in this paper concern cubical
hyperresolutions.
Cubical hyperresolutions are the cubical analogue of simplicial resolutions.
Deligne \cite{Del74} used simplicial resolutions
to construct mixed Hodge structures on the singular
cohomology of singular complex varieties.
\par For varieties, the existence of cubical hyperresolutions is due to
Navarro Aznar (see \cite[\S2]{Gui87}) and was published by Guill\'en
\cite[Expos\'e I]{GNAPGP88}.
The complex analytic case appears in \cite[Expos\'e I]{GNAPGP88} as well.
The theory of cubical hyperresolutions and various applications to Hodge theory,
vanishing theorems, and more were shown in \cite{GNAPGP88}.
Our existence result for cubical hyperresolutions
shows that cubical hyperresolutions exist in all categories of spaces
of the form
\((\ref{setup:introalgebraicspaces})\textnormal{--}(\ref{setup:introadicspaces})\).
\begin{customthm}{\hyperref[thm:gnapgp215]{B}}[Cubical
  hyperresolutions exist; special case of Theorem \ref{thm:gnapgp215}]
  \label{thm:introcubicalhyperres}
  Let \(\Sp\) be a category of reduced spaces of one of the types
  \((\ref{setup:introalgebraicspaces})\textnormal{--}(\ref{setup:introadicspaces})\)
  that is essentially stable under fiber products, immersions, and proper morphisms.
  Let \(I\) be a finite orderable category and let
  \[
    S\colon I^\op \longrightarrow \Sp
  \]
  be a functor such that
  \[
    \dim(S) \coloneqq \sup\Set[\big]{\dim(S_i) \given i \in \Ob(I)}
  \]
  is finite.
  In cases \((\ref{setup:introalgebraicspaces})\) and
  \((\ref{setup:introformalqschemes})\), suppose that
  \(S_i\) is quasi-excellent of equal characteristic zero for every \(i \in
  \Ob(I)\).
  In cases \((\ref{setup:introberkovichspaces})\),
  \((\ref{setup:introrigidanalyticspaces})\), and
  \((\ref{setup:introadicspaces})\), suppose that \(k\) is of characteristic
  zero.
  Then, there exists an augmented cubical hyperresolution
  \[
    Z_\bullet\colon (\square_r^+ \times I)^\op \longrightarrow \Sp
  \]
  of \(S\) such that
  \[
    \dim(Z_\alpha) \le \dim(S) - \abs{\alpha} + 1
  \]
  for every \(\alpha \in \square_r\).
\end{customthm}
\par Our proof of Theorem \ref{thm:introcubicalhyperres} largely follows that in
\cite[Expos\'e I]{GNAPGP88}, except for some corrections due to Steenbrink
\cite{Ste17} and our use of Temkin's resolutions of singularities
\cite{Tem12,Tem18} instead of Hironaka's \cite{Hir64}.
We work through the proof of Theorem \hyperref[thm:gnapgp215]{B} and
other material from \cite[Expos\'e I]{GNAPGP88} carefully in this paper
for two reasons.
\begin{enumerate}[label=\((\arabic*)\),ref=\arabic*]
  \item\label{intro:needprojhyperres}
    For our applications, we need to construct cubical hyperresolutions
    using only (locally) blowup morphisms.
    This extra condition is necessary for
    our proof of Theorem \ref{thm:introgna} below.
  \item Even for complex varieties, there are some subtleties concerning the
    definitions of resolutions, 2-resolutions, and cubical hyperresolutions in
    \cite{GNAPGP88,GNA02}.
    Specifically, Steenbrink \cite{Ste17} constructed a counterexample to the
    existence results in \cite[Expos\'e I]{GNAPGP88} and gave corrections
    that we incorporate into this paper.
    See Remark \ref{rem:ste17proofcomment}.
\end{enumerate}
\par Compared to simplicial hyperresolutions, cubical hyperresolutions have the
advantage of only having finitely many irreducible components.
Cubical objects are also convenient because products of cubes are cubes,
and hence cubical objects are stable under products.
\par However, the most significant advantage of cubical hyperresolutions in the context of
this paper is the following result, which provides a powerful and precise categorical
framework for extending functors from the subcategory of regular spaces to an
entire category of spaces.
This result was first proved for complex varieties by Guill\'en and Navarro
Aznar \cite{GNA02} and for compactifiable complex analytic spaces by Cirici and
Guill\'en \cite{CG14}.
A special case of their extension criterion says that given a functor \(G\)
defined on
regular spaces with values in a category \(\C^b(\sA)\) of bounded complexes over
an Abelian category \(\sA\), if \(G\)
is additive and compatible with blowups along regular subspaces in a suitable
sense, then \(G\) can be extended to a functor on \emph{all} spaces.
The extension \(G'\) in Theorem \ref{thm:introgna} below is
computed by taking a cubical hyperresolution \(X_\bullet\) of a space \(X\),
evaluating \(G\) on each component of \(X_\bullet\), and then taking the
associated simple (or total) complex.
\par Because the full statement requires substantial 2-categorical preliminaries,
we state a special case of our result below.
See Theorem \ref{thm:gna215} for the full statement, which works for
\(\Phi\)-rectified functors \(G\) (see Definition \ref{def:rectified})
with values in a category of cohomological descent (see Definition
\ref{def:catcohdescent}).
Below, a space \(U\) is \textsl{compactifiable} if there exists a quasi-compact
space \(X\) and a closed subspace \(Z \subseteq X\) such that \(U \simeq X -
Z\).
\begin{customthm}{\hyperref[thm:gna215]{C}}[An extension criterion for functors
  defined on regular compactifiable spaces;
  special case of Theorem \ref{thm:gna215}]
  \label{thm:introgna}
  Let \(\Sp\) be a small category of compactifiable
  finite-dimensional reduced spaces of one of the types
  \((\ref{setup:introalgebraicspaces})\textnormal{--}(\ref{setup:introadicspaces})\)
  that is essentially stable under fiber products, immersions, and proper morphisms.
  In cases \((\ref{setup:introalgebraicspaces})\) and
  \((\ref{setup:introformalqschemes})\), suppose that all objects in \(\Sp\) are
  Noetherian and quasi-excellent of equal characteristic zero.
  In cases \((\ref{setup:introberkovichspaces})\),
  \((\ref{setup:introrigidanalyticspaces})\), and
  \((\ref{setup:introadicspaces})\), suppose that
  \(k\) is of characteristic zero.
  \par Let \(\Sp_\reg\) be the full subcategory of \(\Sp\) consisting of regular
  spaces.
  Let \(\sA\) be an Abelian category and consider a functor
  \[
    G\colon \bigl(\Sp_\reg\bigr)^{\op} \longrightarrow \C^b(\sA)
  \]
  satisfying the following conditions:
  \begin{enumerate}[label=\((\mathrm{F\arabic*})\),ref=\mathrm{F\arabic*}]
    \item
      \label{thm:introgna251f1}
      \(G(\emptyset) = 0\) and the canonical morphism
      \(G(X \sqcup Y) \to G(X) \times G(Y)\)
      is a quasi-isomorphism.
    \item
      \label{thm:introgna251f2}
      For every Cartesian square
      \[
        \begin{tikzcd}
          \tilde{Y} \dar[swap]{g} \rar[hook]{j} & \tilde{X} \dar{f}\\
          Y \rar[hook]{i} & X
        \end{tikzcd}
      \]
      in \(\Sp_\reg\) where \(i,j\) are closed immersions and \(f\) is the
      blowup of \(X\) along \(Y\), the morphism
      \[
        G(X) \xrightarrow{f^*+i^*} \bfs\Bigl( G(\tilde{X}) \oplus G(Y)
        \xrightarrow{j^*-g^*} G(\tilde{Y})\Bigr)
      \]
      is a quasi-isomorphism, where \(\bfs(\,\cdot\,)\)
      denotes the simple (or total) complex.
  \end{enumerate}
  Then, there exists an essentially unique extension of \(G\) to a functor
  \[
    G'\colon \Sp^{\op} \longrightarrow \D^b(\sA)
  \]
  satisfying the analogue of \((\ref{thm:introgna251f2})\) for proper modifications
  \(f\) such that \(f\rvert_{X-Y}\) is an isomorphism.
\end{customthm}
\par The proof of Theorem \ref{thm:introgna} and the original version in
\cite{GNA02} is difficult.
The proof in \cite{GNA02} relies on two main ingredients:
\begin{enumerate}
  \item\label{ingredient:intrognares} Hironaka's resolutions of singularities \cite[Chapter 0, \S3, Main
    theorem I]{Hir64}.
  \item\label{ingredient:intrognachow} The Chow--Hironaka lemma
    \citeleft\citen{Hir64}\citemid pp.\ 144--145\citepunct \citen{Hir75}\citemid p.\
    505\citeright,
    which says that
    any proper birational morphism between smooth varieties can be dominated by a
    finite sequence of blowups along smooth centers.
\end{enumerate}
The main difference between our proof and that in \cite{GNA02} is that we do not
have a version of the Chow--Hironaka lemma available for arbitrary spaces in the
categories 
\((\ref{setup:introalgebraicspaces})\textnormal{--}(\ref{setup:introadicspaces})\).
Instead, our new insight is that the proof in \cite{GNA02} works just as well if
we replace compositions of blowups in one direction with compositions of blowups
and their inverses.
Thus, we are able to replace the Chow--Hironaka lemma with
the weak factorization theorem \cite{AKMW02,Wlo03}, which holds in all categories
\((\ref{setup:introalgebraicspaces})\textnormal{--}(\ref{setup:introadicspaces})\)
by work of Abramovich and Temkin \cite{AT19}.
However, to apply this weak factorization theorem, we must construct cubical
hyperresolutions that only use projective morphisms (see
\((\ref{intro:needprojhyperres})\) on p.\ \pageref{intro:needprojhyperres}).
\subsection{Additional applications}
We discuss two additional applications of Theorem \ref{thm:introgna} and
its proof.
\subsubsection{The Deligne--Du Bois complex and Du Bois singularities}
\par In Theorem \ref{thm:gna41},
we use Theorem \ref{thm:introgna} and its proof to construct the Deligne--Du Bois
complex \(\underline{\Omega}_X^\bullet\) in all categories
\((\ref{setup:introalgebraicspaces})\textnormal{--}(\ref{setup:introadicspaces})\).
While the construction only makes sense when sheaves of differential forms
\(\Omega_X\) make sense, we can always define the sheaf
\(\underline{\Omega}_X^0\).
We use the sheaf \(\underline{\Omega}_X^0\) to define Du Bois singularities for
all spaces in the categories
\((\ref{setup:introalgebraicspaces})\textnormal{--}(\ref{setup:introadicspaces})\).
See Definition \ref{def:dubois}.
The construction of \(\underline{\Omega}_X^\bullet\) for varieties is originally
due to Du Bois \cite{DB81}, who used simplicial resolutions.
A cubical construction is given in \cite[Expos\'es III and V]{GNAPGP88}.
For complex analytic spaces, Guill\'en and Navarro Aznar constructed
\(\underline{\Omega}_X^\bullet\) using their version of
Theorem \ref{thm:introgna}.
Guo recently gave a construction for \(\underline{\Omega}_X^\bullet\) for rigid
analytic spaces over \(p\)-adic fields \cite{Guo23}.
\par By Theorem \ref{thm:introgna}, our constructions of
\(\underline{\Omega}_X^\bullet\) and \(\underline{\Omega}_X^0\)
are automatically functorial and these objects are essentially unique.
As mentioned in \S\ref{sect:intromaininj}, these results on
\(\underline{\Omega}_X^0\) are critical in our proof of our injectivity theorems
(Theorem \ref{thm:maininj})
because they allow us to use specific cubical hyperresolutions to prove that
generalized normal crossings log pairs have at most Du Bois
singularities, and the finiteness of the hyperresolutions allows us to prove
they are compatible with localization and completion (Corollary
\ref{cor:delignedblocal}).
As far as we are aware, existing proofs of these formal properties of
\(\underline{\Omega}_X^\bullet\) in the literature (see
\cite{DB81,GNAPGP88}) do not work for arbitrary spaces of the form
\((\ref{setup:introalgebraicspaces})\textnormal{--}(\ref{setup:introadicspaces})\), which led
us to generalize \cite{GNA02} to our setting.
\subsubsection{A weight filtration on the (pro-)\'etale cohomology of schemes
and compactifiable rigid analytic spaces}
\par Finally, we mention one last application of Theorem \ref{thm:introgna}.
Deligne's original motivation for working with simplicial resolutions comes from
Hodge theory, where he constructed a mixed Hodge structure on the singular
cohomology of arbitrary complex varieties using simplicial resolutions
\cite{Del74}.
Gillet and Soul\'e later constructed the weight filtration on singular
cohomology using resolutions of singularities and algebraic \(K\)-theory
\cite{GS95}.
\par The work of Guill\'en and Navarro Aznar \cite{GNA02} and our generalization
of their work in Theorem \ref{thm:introgna} gives a new approach to constructing the
weight filtration using only resolutions of singularities and either the
Chow--Hironaka lemma (in \cite{GNA02}) or the weak factorization theorem (in
this paper).
Using \cite{GNA02}, Totaro constructed the weight filtration on singular
cohomology with compact support on complex or real analytic spaces \cite{Tot02}.
McCrory and Parusi\'nski \cite{McP11,McP14} also used \cite{GNA02} to construct
analogous weight filtrations for Borel--Moore homology and compactly supported
singular homology on real algebraic varieties with \(\ZZ/2\ZZ\) coefficients.
\par Using \cite{GNA02}, Cirici and Guill\'en \cite{CG14} showed
that Deligne's construction can be used to construct a weight filtration on the
singular cohomology of \emph{compactifiable} complex analytic spaces.
In this paper, we adapt their construction to construct a weight filtration on
the (pro-)\'etale cohomology of schemes and compactifiable rigid analytic
spaces (Theorem \ref{thm:weightfiltration}).
\subsection{Notation}
By a \textsl{ring} we mean a commutative ring with identity and by a
\textsl{ring map} we mean unital ring homomorphisms.
If \(f\colon X \to Y\) is a separated morphism of locally Noetherian
schemes and \(Y\) is assumed to have a dualizing complex \(\omega_Y^\bullet\),
then we will always use the induced dualizing complex \(\omega_X^\bullet
\coloneqq f^!\omega_Y^\bullet\).
We will follow the analogous convention for other kinds of spaces when the
exceptional pullback \(f^!\) is defined.
See \cite[\S23]{LM}.
\par We denote by \(\Cat\) the (large) category of small categories.
We will sometimes work with the structure of \(\Cat\) as a 2-category.
If \(\sC\) is a small category, we denote by \(\Ob(\sC)\) and
\(\Mor(\sC)\) the sets of objects and morphisms of \(\sC\),
respectively.
For every natural number \(n \ge 0\), we denote by \(\underline{n}\) the
well-ordered set \(\{0,1,\ldots,n-1\}\).
Sites, topoi, and their morphisms are defined as in
\cite[\href{https://stacks.math.columbia.edu/tag/00VH}{Tag
00VH}, \href{https://stacks.math.columbia.edu/tag/00X1}{Tag 00X1}, and
\href{https://stacks.math.columbia.edu/tag/00XA}{Tag 00XA}]{stacks-project}.

\subsection{Set-theoretic conventions}\label{sect:settheory}
We prefer to work within Zermelo--Fraenkel set theory with
the Axiom of Choice (ZFC).
We will also sometimes use the language of proper classes.
This is allowed since introducing classes, for example by using von
Neumann--Bernays--G\"odel set theory (NBG) with the Axiom of
Choice, yields a conservative extension of
ZFC \cite[\S4.1]{McL20}.
\begin{remark}[Choosing small categories of spaces]\label{rem:smallspaces}
  When we need to work with a \emph{small} subcategory \(\Sp\) of one of the
  categories of spaces of the form
  \((\ref{setup:introalgebraicspaces})\textnormal{--}(\ref{setup:introadicspaces})\),
  we choose \(\Sp\) so that \(\Sp\) is essentially stable under
  fiber products, immersions, and proper morphisms, and so that these
  constructions coincide with what they would be in the (large) category of
  spaces.
  For schemes, such a choice for \(\Sp\) can be made by taking the category
  \(\Sch_\alpha\) constructed in
  \cite[\href{https://stacks.math.columbia.edu/tag/000J}{Tag
  000J}]{stacks-project}, which is essentially stable under more constructions
  than those listed above \cite[\href{https://stacks.math.columbia.edu/tag/000R}{Tag
  000R}]{stacks-project}, and replacing \(\Sch_\alpha\) with a smaller subcategory.
  For example, we could take all objects in \(\Sch_\alpha\) that are
  quasi-excellent or of finite type over a (varying or fixed) field \(k\).
  The 2-category of algebraic stacks over a scheme \(S\) is already small by
  definition in \cite[\href{https://stacks.math.columbia.edu/tag/03YP}{Tag
  03YP}]{stacks-project}.
  In cases
  \((\ref{setup:introformalqschemes})\textnormal{--}(\ref{setup:introadicspaces})\),
  one adapts the proofs in
  \cite[\href{https://stacks.math.columbia.edu/tag/000J}{Tag
  000J} and \href{https://stacks.math.columbia.edu/tag/000R}{Tag
  000R}]{stacks-project}.
\end{remark}
  Many results in the literature use the Axiom of Universes \((\AU)\) from
  \cite[Expos\'e I, Appendice, \S4, (A.6)]{SGA41}.
  When we can, we cite references that do not assume \((\AU)\), for
  example \cite{stacks-project}.
  The results we need from
  \cite{EGAII,EGAIII1,EGAIV2,EGAIV3,EGAIV4,SGA1new,SGA41,SGA42,SGA43,SGA5} do not need
  \((\AU)\) \cite{McL10,McL20}.\medskip
  \par We prefer to work in ZFC to avoid introducing additional axioms.
  The Axiom of Universes \((\AU)\)
  implies the existence of a strongly inaccessible cardinal \cite[p.\
  360]{McL10}.
  This in turn implies the consistency of ZFC and therefore
  cannot be proved in ZFC by G\"odel's second incompleteness theorem
  \cite[pp.\ 18--19]{Kan09}.

\section{Preliminaries on spaces}
In this section, we fix some terminology that will be convenient to
unify our presentation for all spaces of the form
\((\ref{setup:introalgebraicspaces})\textnormal{--}(\ref{setup:introadicspaces})\).
\subsection{Affinoids}\label{sect:affinoids}
We adopt the terminology \textsl{affinoid subdomain} from non-Archimedean
geometry.
See \cite[\S6.2.3(2)]{AT19} for the definition in cases
\((\ref{setup:introformalqschemes})\textnormal{--}(\ref{setup:introrigidanalyticspaces})\).
For case \((\ref{setup:introalgebraicspaces})\), we will mean smooth morphisms
from an affine scheme.
For case \((\ref{setup:introadicspaces})\), see
\cite[Definition on p.\ 521]{Hub94}.
\subsection{Associated subspaces}\label{sect:associated}
For a space \(X\), we define the set of associated subspaces of a coherent
\(\cO_X\)-module following
\cite[\href{https://stacks.math.columbia.edu/tag/0CTX}{Tag
0CTX}]{stacks-project} in case \((\ref{setup:introalgebraicspaces})\) when \(X\)
is an algebraic space,
\cite[p. 379]{Siu69} in case \((\ref{setup:introcomplexanalyticgerms})\),
and \cite[D\'efinition 2.7]{Duc21} in cases
\((\ref{setup:introberkovichspaces})\) and
\((\ref{setup:introrigidanalyticspaces})\).
In cases \((\ref{setup:introformalqschemes})\) and
\((\ref{setup:introadicspaces})\), we define associated subspaces
affinoid-locally, in which case the excellence results
necessary to adapt the argument in \cite[\S2]{Duc21} hold by assumption and by
\cite[Lemma 24.9]{LM}, respectively.
\subsection{Closures}\label{sect:closures}
In each of the categories
\((\ref{setup:introalgebraicspaces})\textnormal{--}(\ref{setup:introadicspaces})\),
there is a good notion of scheme-theoretic closures (for open subsets)
or analytic closures (for analytic subdomains).
See
\cite[\href{https://stacks.math.columbia.edu/tag/0834}{Tag 0834}]{stacks-project}
for case \((\ref{setup:introalgebraicspaces})\) and
\cite[Proposition 3.29]{Yas09} for case \((\ref{setup:introformalqschemes})\).
\par For the analytic cases, we proceed as follows.
Case \((\ref{setup:introberkovichspaces})\) is defined in
\cite[Lemme-D\'efinition 2.17]{Duc21}.
Case \((\ref{setup:introcomplexanalyticgerms})\) can be defined similarly by
working on affinoid subdomains.
Note
that the excellence results necessary
to adapt the argument in \cite[\S2]{Duc21} hold for affinoid semianalytic germs
of complex analytic spaces by \citeleft\citen{Fri67}\citemid Th\'eor\`eme
(I,\,9)\citepunct \citen{Mat73}\citemid Theorem 2.7\citepunct \citen{AT19}\citemid
Lemma B.6.1\((i)\)\citeright.
By \cite[Theorem 1.6.1]{Ber93}, \((\ref{setup:introrigidanalyticspaces})\) can
be reduced to case \((\ref{setup:introberkovichspaces})\).
Finally, for case \((\ref{setup:introadicspaces})\), we work locally to reduce
to the case of taking closures in \(\Spa(A,A^+)\) where \(A\) is topologically
of finite type over \(k\).
We then think of \(\Spa(A,A^\circ)\) as an open subspace of \(\Spa(A,A^+)\) (see
\cite[Proof of Proposition 3.2]{Man23}) and
reduce to case \((\ref{setup:introrigidanalyticspaces})\) using
\citeleft\citen{Hub94}\citemid Proposition 4.5\((iv)\)\citepunct
\citen{Man23}\citemid Lemma 3.1\citeright.
\subsection{Compactifiable spaces}
We will use the following terminology.
\begin{definition}
  Let \(U\) be a space of one of the types
  \((\ref{setup:introalgebraicspaces})\textnormal{--}(\ref{setup:introadicspaces})\).
  We say that \(U\) is \textsl{compactifiable} if there exists a quasi-compact
  space \(X\) and closed subspace \(Z \subseteq X\) such that \(U \simeq X -
  Z\).
\end{definition}
\subsection{Dimension}\label{sect:dimension}
We define dimension as in
\cite[\href{https://stacks.math.columbia.edu/tag/0AFP}{Tag
0AFP}]{stacks-project} in case \((\ref{setup:introalgebraicspaces})\),
\cite[Chapitre 0, D\'efinition 14.1.2]{EGAIV1} in case
\((\ref{setup:introformalqschemes})\),
\cite[5.1.1]{GR84} in case \((\ref{setup:introcomplexanalyticgerms})\),
\cite[p.\ 23]{Ber93} in case \((\ref{setup:introberkovichspaces})\),
\cite[p.\ 496]{Con99} in case \((\ref{setup:introrigidanalyticspaces})\),
and
\cite[Definition 1.8.1]{Hub96} in case \((\ref{setup:introadicspaces})\).
\subsection{The \'etale and pro-\'etale topologies}\label{sect:etaletop}
When discussing sheaves on diagrams of spaces, we will mostly work with sheaves on the
\'etale topology.
In case \((\ref{setup:introalgebraicspaces})\), we instead work with the
lisse-\'etale topology
\cite[\href{https://stacks.math.columbia.edu/tag/0787}{Tag
0787}]{stacks-project}
using the six-functor formalism established in
\cite{LO08,LZ}.
\par In order to treat the other cases
\((\ref{setup:introformalqschemes})\textnormal{--}(\ref{setup:introadicspaces})\)
at once, we will refer to the Euclidean topology as the \'etale topology in the
complex analytic setting (case \((\ref{setup:introcomplexanalyticgerms})\)) and
use the \'etale topology in the other cases.
In case \((\ref{setup:introberkovichspaces})\), see \cite[\S4.1]{Ber93}.
In case \((\ref{setup:introrigidanalyticspaces})\), see \cite[p.\ 58]{SS91}.
In case \((\ref{setup:introadicspaces})\), see \cite[\S2.1]{Hub96}.
In case \((\ref{setup:introformalqschemes})\), note the \'etale site on a
formal scheme and the scheme defined (locally) by an ideal of definition are
equivalent \cite[Expos\'e I, Corollaire 8.4]{SGA1new}.
\par We will also use the pro-\'etale topology in cases
\((\ref{setup:introalgebraicspaces})\) and \((\ref{setup:introadicspaces})\).
See \cite{BS15,Cho20} for \((\ref{setup:introalgebraicspaces})\) (for arbitrary
schemes and for algebraic stacks of finite type over a Noetherian scheme \(S\)).
See \cite[\S3]{Sch13} for case \((\ref{setup:introadicspaces})\).

\subsection{Irreducibility and irreducible components}\label{sect:irred}
In case \((\ref{setup:introalgebraicspaces})\), we say that \(\cX\) is
irreducible if the topological space \(\abs{\cX}\) is irreducible and define
irreducible components by taking closures of generic points in \(\abs{\cX}\).
The assumption that \(\cX\) is decent (see
\cite[\href{https://stacks.math.columbia.edu/tag/0GW1}{Tag
0GW1}]{stacks-project} for the definition)
is used to guarantee that generic points for
irreducible components are unique
\cite[\href{https://stacks.math.columbia.edu/tag/0GW8}{Tag
0GW8}]{stacks-project}.
In case \((\ref{setup:introformalqschemes})\), we use the topological definition
from \cite[\S2.1]{EGAInew}.
In case \((\ref{setup:introcomplexanalyticgerms})\), we define irreducibility as
in \cite[9.1.2]{GR84} and apply the Global Decomposition Theorem
\cite[9.2.2]{GR84} to a representative of a semianalytic germ of a complex
analytic space to obtain a decomposition of the germ
into irreducible components.
In case \((\ref{setup:introberkovichspaces})\), we follow \cite[p.\
1455]{Duc09}.
In case \((\ref{setup:introrigidanalyticspaces})\), we follow \cite[Definition
2.2.2]{Con99}.
In case \((\ref{setup:introadicspaces})\), we follow \cite[Definitions 2.11 and
2.14]{Man23}.
\subsection{Relative Proj and blowing up}\label{sect:blowingup}
We define relative \(\PProj\) and blowups as in
\cite[\S5.1.2]{Tem12}
in case \((\ref{setup:introalgebraicspaces})\),
\cite[Definition 6.3]{FK06} in case
\((\ref{setup:introformalqschemes})\),
\cite[Chapter II, \S1.b]{Nak04} in case \((\ref{setup:introcomplexanalyticgerms})\),
\cite[(5.4)]{Duc21} in case \((\ref{setup:introberkovichspaces})\),
\cite[Definition 2.3.3]{Con06} in case
\((\ref{setup:introrigidanalyticspaces})\),
and
\citeleft\citen{Guo23}\citemid \S2.2\citepunct \citen{Zav}\citemid Definition
6.7\citeright\ in case \((\ref{setup:introadicspaces})\).

\section{Diagrams and codiagrams}
We collect the definitions we
need on diagrams and codiagrams
that are necessary for the theory of cubical hyperresolutions.
This material combines some of the material in \cite[Expos\'e I, \S1]{GNAPGP88}
and in \cite[\S1]{GNA02}.
We have tried to use the notation in \cite{GNAPGP88,GNA02} as much as possible.
We recommend the reader skip this section on first reading and return back to it
as needed.
\subsection{Conventions}
\par In this section, \(\sC\) denotes a (not necessarily small)
category and \(I,J\) denote small categories.
\subsection{1-diagrams and 1-codiagrams}
We start with the notion of a 1-diagram and a 1-codiagram.
We will not use 1-codiagrams until \S\ref{sect:gna02}.
\begin{definition}[1-diagrams and 1-codiagrams
  {\citeleft\citen{GNAPGP88}\citemid Expos\'e I, (1.2)\citepunct
  \citen{GNA02}\citemid (1.2.1) and (1.6.1)\citeright}]
  \label{def:gnapgp12}
  A \textsl{1-diagram of \(\sC\) of type \(I\)}, which we also call an
  \textsl{\(I\)-object of \(\sC\)}, is a functor \(I^\op \to \sC\).
  A \textsl{1-codiagram of \(\sC\) of type \(I\)} is a functor \(I \to \sC\).
  \par Let \(X\) and \(Y\) be two 1-diagrams (resp.\ 1-codiagrams)
  of \(\sC\) of types \(I\) and
  \(J\), respectively.
  A \textsl{morphism of 1-diagrams of \(\sC\) from \(X\) to \(Y\)} (resp.\
  a \textsl{morphism of 1-codiagrams of \(\sC\) from \(X\) to \(Y\)}) is a pair
  consisting of a functor \(\varphi\colon I \to J\) and a natural transformation
  \(f\colon X \rightarrow \varphi^*(Y)\).
  Such a morphism \(f\) is called a \textsl{\(\varphi\)-morphism} and can be
  visualized using the diagrams
  \begin{equation}\label{eq:diagrams2cat}
    \begin{tikzcd}[column sep=tiny]
      I^\op \arrow[rr,"\varphi^\op"]\arrow[dr,"X"'near start,""{name=X,near end}] & & J^\op
      \arrow[dl,"Y"near start,""'{name=Y,near end}] &[5em]
      I \arrow[rr,"\varphi"]\arrow[dr,"X"'near start,""{name=coX,near end}]
      &[0.33em] &[0.33em] J
      \arrow[dl,"Y"near start,""'{name=coY,near end}]\\
      & \sC & & & \sC\mathrlap{.}
      \arrow[Rightarrow,from=X,to=Y,"f"]
      \arrow[phantom,from=Y,to=coX,"\text{and}"{yshift=2.75pt}]
      \arrow[Rightarrow,from=coX,to=coY,"f"]
    \end{tikzcd}
  \end{equation}
  If \(\varphi\) is the identity functor on \(I\), we simply say that \(f\) is
  an \textsl{\(I\)-morphism} or (for 1-diagrams only)
  a \textsl{morphism of \(I\)-objects}.
  \par The 1-diagrams (resp.\ 1-codiagrams) of \(\sC\) with
  \(\varphi\)-morphisms as defined above form a category \(\Diag_1(\sC)\)
  (resp.\ \(\Codiag_1(\sC)\)).
  If \(\fI\) is a subcategory of \(\Cat\), we denote by \(\Diag_\fI(\sC)\)
  (resp.\ \(\Codiag_\fI(\sC)\)) the subcategory of \(\Diag_1(\sC)\) (resp.\
  \(\Codiag_1(\sC)\)) whose types lie in \(\fI\).
  When \(\fI\) consists of one object \(I\) and the identity on \(I\), we denote
  this subcategory by \(\Diag_I(\sC)\) (resp.\ \(\Codiag_I(\sC)\)).
  \par We have functors
  \[
    \typ\colon \Diag_1(\sC) \longrightarrow \Cat\qquad\text{and}\qquad
    \typ\colon \Codiag_1(\sC) \longrightarrow \Cat
  \]
  defined by setting \(\typ(X) = I\) for every 1-diagram or 1-codiagram
  \(X\) of \(\sC\) and \(\typ(f) = \varphi\) for every \(\varphi\)-morphism
  \(f\) of 1-diagrams or 1-codiagrams of \(\sC\).
  \par If \(X\) is a 1-diagram of \(\sC\) of type \(I\),
  \(i \in \Ob(I)\), and \(u \in \Mor(I)\), we
  denote by \(X_i\), \(X_u\), and \(f_i\) the images of \(i\), \(u\), and \(i\)
  under \(X\), \(X\), and \(f\), respectively.
  If \(X\) is a 1-codiagram of \(\sC\) of type \(I\), we use the
  notation \(X^i\), \(X^u\), and \(f^i\) instead of \(X_i\), \(X_u\), and
  \(f_i\).
\end{definition}
\begin{definition}[Augmented 1-diagrams and 1-codiagrams
  {\cite[Expos\'e I, (1.3)]{GNAPGP88}}]\label{def:gnapgp13}
  Let \(S\) be an object of \(\sC\).
  A \textsl{1-diagram} (resp.\ \textsl{1-codiagram}) \textsl{of type \(I\)
  augmented by \(S\)} is a 1-diagram (resp.\ 1-codiagram) \(X^+\) of type \(I\) of
  the slice category \(\sC/S\) (resp.\ the coslice category \(S\backslash \sC\)).
  We identify \(X^+\) with the 1-diagram (resp.\ 1-codiagram)
  \(X\) of type \(I\) of \(\sC\) together with a
  morphism \(a\colon X \to S\) (resp.\ \(a\colon S \to X\))
  of 1-diagrams (resp.\ 1-codiagrams) of \(\sC\) where \(\typ(a)\) is
  the canonical functor \(I \to \underline{1}\) (resp.\ the canonical functor
  \(\underline{0} \to I\)).
  \par Let \(\varphi\colon I \to J\) be a functor between small
  categories and let \(u\colon S \to T\) be a morphism in \(\sC\).
  Let \(a \colon X \to S\) (resp.\ \(a\colon S \to X\))
  and \(b\colon Y \to T\) (resp.\ \(a\colon T \to Y\)) be 1-diagrams (resp.\
  1-codiagrams) of type \(I\)
  and \(J\) of \(\sC\) augmented by \(S\) and \(T\), respectively.
  A \textsl{\(\varphi\)-morphism augmented by \(u\)} is a \(\varphi\)-morphism
  \(f\colon X \to Y\) such that the diagram
  \begin{equation}\label{eq:moraugmented}
    \begin{tikzcd}
      X \rar{f}\dar[swap]{a} & Y\arrow[d,"b"{name=b}]
      &[5em] S \rar{u}\arrow[d,"a"'{name=a}] & T\dar{b}\\
      S \rar{u} & T & X \rar{f} & Y
      \arrow[from=b,to=a,phantom,"\text{resp.}"{yshift=-1pt}]
    \end{tikzcd}
  \end{equation}
  of 1-diagrams of \(\sC\) commutes.
  The augmented 1-diagrams (resp.\ 1-codiagrams) of \(\sC\) together with augmented
  \(\varphi\)-morphisms form a category \(\Diag_1^+(\sC)\) (resp.\
  \(\Codiag_1^+(\sC)\)).
  We have two projection functors
  \[
    \begin{tikzcd}[cramped,row sep=0,column sep=scriptsize]
      \sC & \lar[swap]{\pi_0} \Diag_1^+(\sC) \rar{\pi_1} & \Diag_1(\sC)\\
      \sC & \lar[swap]{\pi_0} \Codiag_1^+(\sC) \rar{\pi_1} & \Codiag_1(\sC)
    \end{tikzcd}
  \]
  where \(\pi_0(X \to S) = S\) (resp.\ \(\pi_0(S \to X) = S\))
  and \(\pi_1(X \to S) = X\) (resp.\ \(\pi_1(S \to X) = X\)) on objects and
  \(\pi_0(f) = u\) and \(\pi_1(f) = f\) for a morphism \(f\) as in
  \eqref{eq:moraugmented}. 
\end{definition}
\subsection{The total 1-diagram of a 2-diagram}
Next, we define 2-diagrams and total 1-diagrams.
\begin{definition}[2-diagrams
  {\citeleft\citen{GNAPGP88}\citemid Expos\'e I, (1.4)\citeright}]
  The category of \textsl{2-diagrams of \(\sC\)} is the category
  \(\Diag_2(\sC) \coloneqq \Diag_1(\Diag_1(\sC))\).
  We have a functor
  \[
    \typ_1\colon \Diag_2(\sC) \longrightarrow \Diag_1(\sC)
  \]
  defined by \(\typ_1(X) = \typ \circ X\) for every 2-diagram \(X\) of \(\sC\).
\end{definition}
\begin{definition}[Total categories
  {\citeleft\citen{GNAPGP88}\citemid Expos\'e I, (1.5)\citeright}]
  Let \(K\) be an \(I\)-object of \(\Cat\).
  The \textsl{total category \(\tot(K)\) of the functor \(K\)} is the category
  obtained as follows:
  \begin{itemize}
    \item The objects of \(\tot(K)\) are pairs \((i,x)\) where \(i \in \Ob(I)\)
      and \(x \in \Ob(K_i)\).
    \item The morphisms \((i,x) \to (j,y)\) of \(\tot(K)\) are pairs \((u,a)\)
      consisting of a morphism
      \(u\colon i \to j\) of \(I\) and a morphism \(a\colon x \to K_u(y)\) of
      \(K_i\).
  \end{itemize}
  Let \((u,a)\colon (i,x) \to (j,y)\) and \((v,b)\colon (j,y) \to (k,z)\) be two
  morphisms of \(\tot(K)\).
  The composition is defined by 
  \((v,b) \circ (u,a) \coloneqq (v \circ u,K_u(b) \circ a)\).
  \par If \(f\colon K \to L\) is a morphism of 1-diagrams of \(\Cat\) such that
  \(\typ(f) = \varphi\), then \(f\) defines a functor
  \[
    \begin{tikzcd}[cramped,row sep=0,column sep=1.475em]
      \tot(f)\colon &[-2.2em] \tot(K) \rar & \tot(L)\\
      & (i,x) \rar[mapsto] & \bigl(\varphi(i),f_i(x)\bigr)\\
      & (u,a) \rar[mapsto] & \bigl(\varphi(u),f_i(a)\bigr)\mathrlap{.}
    \end{tikzcd}
  \]
  With these definitions, we obtain a functor
  \(\tot\colon \Diag_1(\Cat) \rightarrow \Cat\).
  \par If \(K\) is an \(I\)-object of \(\Cat\), we have the projection functor
  \(\pi\colon \tot(K) \rightarrow I\).
\end{definition}
\begin{definition}[Total 1-diagrams
  {\citeleft\citen{GNAPGP88}\citemid Expos\'e I, (1.6)\citeright}]
  \label{def:gnapgp16}
  Let \(K\) be an \(I\)-object of \(\Cat\).
  Let \(X\) be a 2-diagram of
  \(\sC\) such that \(\typ_1(X) = K\).
  The \textsl{total 1-diagram of \(X\)} is the 1-diagram
  \[
    \begin{tikzcd}[cramped,row sep=0,column sep=1.475em]
      \tot(X)\colon&[-2.2em] \tot(K) \rar & \sC\\
      & (i,x) \rar[mapsto] & X_i(x)\\
      & (u,a) \rar[mapsto] & X_i(a) \circ X_u(y)\mathrlap{.}
    \end{tikzcd}
  \]
  With these definitions, we obtain a functor
  \(\tot\colon \Diag_2(\sC) \rightarrow \Diag_1(\sC)\).
  By \cite[Expos\'e VI, Proposition 12.1]{SGA1new}, the commutative diagram of
  functors
  \[
    \begin{tikzcd}
      \Diag_2(\sC) \rar{\tot} \dar[swap]{\typ_1} & \Diag_1(\sC) \dar{\typ}\\
      \Diag_1(\Cat) \rar{\tot} & \Cat
    \end{tikzcd}
  \]
  is Cartesian.
  This diagram allows us to identify a 2-diagram \(X\) of \(\sC\) with a pair
  \((\tot(X),\typ_1(X))\).
\end{definition}
We can reinterpret augmentation using total diagrams.
\begin{definition}[Augmented categories
  {\cite[Expos\'e I, (1.7)]{GNAPGP88}}]\label{def:gnapgp17}
  We identify a functor \(X_\bullet\colon \underline{2}^\op \to \sC\) with
  the morphism \(f\colon X_1 \to X_0\) in \(\sC\).
  In particular, if \(\varphi\colon K_1 \to K_0\) is a functor between
  small categories, then \(\varphi\) is identified with the functor
  \(K_\bullet\colon \underline{2}^\op \to \Cat\).
  \par If \(\varphi\colon K_1 \to K_0\) is a functor between small
  categories, we denote by \(\tot(\varphi)\) the total category of the functor
  \(K_\bullet\) defined by \(\varphi\).
  If \(f\colon X_1 \to X_0\) is a \(\varphi\)-morphism of 1-diagrams of a
  category \(\sC\), we denote by \(\tot(f)\) the total 1-diagram of the
  2-diagram
  \[
    X_\bullet\colon \underline{2}^\op \longrightarrow \Diag_1(\sC)
  \]
  defined by \(f\).
  \par Let \(\varepsilon\colon I \to \underline{1}\) be the canonical functor
  from
  \(I\) to the singleton category \(\underline{1} = \{0\}\).
  The \textsl{augmented category of \(I\)} is the category \(I^+ \coloneqq
  \tot(\varepsilon)\).
  The unique object \(0 \in \Ob(\underline{1})\) is identified with the initial
  object of \(I^+\).
  Now let \(X\) be an \(I\)-object of \(\sC\) and let \(a\colon X \to S\) be an
  augmentation of \(X\) by an object \(S\) of \(\sC\).
  Then, \(X^+ \coloneqq \tot(a)\) is an \(I^+\)-object of \(\sC\).
  Conversely, every \(I^+\)-object of \(\sC\) is associated to an augmented
  \(I\)-object of \(\sC\).
\end{definition}
\begin{definition}[\(\pi\)-augmentations {\cite[Expos\'e I, (1.8)]{GNAPGP88}}]
  Let \(K\) be an \(I\)-object of \(\Cat\), let \(\pi\colon \tot(K) \to I\) be
  the projection functor, and let \(X\) be a 2-diagram of \(\sC\) such that
  \(\typ_1(X) = K\).
  Let \(S\) be an \(I\)-object of \(\sC\).
  A \textsl{\(\pi\)-augmentation of \(X\) by \(S\)} is a \(\pi\)-morphism \(a
  \colon \tot(X) \to S\) of 1-diagrams of \(\sC\).
  Giving such a \(\pi\)-morphism of 1-diagrams of \(\sC\) is equivalent to
  giving a functor
  \[
    X^+\colon I^\op \longrightarrow \Diag^+_1(\sC)
  \]
  such that \(\pi_0 \circ X^+ = S\) and \(\pi_1 \circ X^+ = X\).
  See Definition \ref{def:gnapgp13}.
\end{definition}
\subsection{Orderable categories}
An important notion for us will be the following notion of an orderable
category, which is a categorical generalization of a partially ordered set
(see Example \ref{ex:notorderedset}).
\begin{definition}[Orderable categories
  {\citeleft\citen{GNAPGP88}\citemid Expos\'e I, (1.9) and
  (1.11)\citepunct \citen{GNA02}\citemid (1.1.2)\citeright}]
  The set \(\Ob(I)\) can be given a preorder in the following manner:
  \(i \le j\) if and only if \(\Hom_I(i,j)\) is nonempty.
  We say that the category \(I\) is \textsl{orderable} if this preorder is a
  partial order and if, for every \(i \in \Ob(I)\), the set of endomorphisms of
  \(i\) consists only of the identity on \(i\).
\end{definition}
Note that a small category \(I\) is orderable if and only if \(I\) is rigid and
reduced, i.e., if every isomorphism and every endomorphism in \(I\) is an
identity morphism.
\begin{citedex}[{\cite[Expos\'e I, (1.10)]{GNAPGP88}}]\label{ex:notorderedset}
  Let \(E\) be a partially ordered set.
  Then, the category associated to \(E\) (see \citeleft\citen{GNA02}\citemid
  (1.1)\citeright) is an orderable category.
\end{citedex}
We now define finite orderable categories and the notion of induction on
such a category.
\begin{definition}[The category \(\Phi\) of finite orderable categories
  {\citeleft\citen{GNAPGP88}\citemid Expos\'e I, (1.13)\citepunct
  \citen{GNA02}\citemid (1.1.2)\citeright}]
  Suppose that \(I\) is orderable.
  We say that \(I\) is \textsl{finite} if the set \(\Mor(I)\) is
  finite.
  We say that \(I\) is \textsl{finite to the left} (resp.\ \textsl{finite to the
  right}) if, for every \(i \in \Ob(I)\), the slice category \(I/i\)
  (resp.\ the coslice category \(i\backslash I\)) is finite.
  If \(I\) is finite, then \(I\) is both finite to the left and finite to the
  right.
  \par We denote by \(\Phi\) the full subcategory of \(\Cat\) whose objects
  are finite orderable categories.
\end{definition}
\begin{definition}[Noetherian induction {\cite[Expos\'e I, (1.14)]{GNAPGP88}}]
  Suppose that \(I\) is orderable.
  Assume that \(I\) is finite to the right (resp.\ finite
  to the left).
  From the definitions, we see that the order associated to \(I\) (resp.\
  \(I^\op\)) is Noetherian.
  We call descending induction (resp.\ increasing induction) on \(I\)
  \textsl{Noetherian induction on \(I\)} (resp.\ \(I^\op\)).
\end{definition}
\subsection{Cubical categories}
Finally, we define cubical categories, which will be the main example of an
orderable category that we work with to construct hyperresolutions.
\begin{definition}[Cubical categories
  {\citeleft\citen{GNAPGP88}\citemid Expos\'e I, (1.15)\citepunct
  \citen{GNA02}\citemid (1.1.1), (1.4.1), and (1.4.2)\citeright}]
  Let \(S\) be a finite set (resp.\ a finite nonempty set).
  We denote by \(\square^+_S\) (resp.\ \(\square_S\)) the category associated to
  the set of subsets of
  \(S\) (resp.\ the set of nonempty subsets of \(S\)), partially ordered by
  inclusion.
  The objects of \(\square^+_S\) (resp.\ \(\square_S\)) can be identified with elements
  \((\alpha(s))_{s \in S} \in \{0,1\}^S\) (resp.\ \((\alpha(s))_{s \in S} \in
  \{0,1\}^S - (0,0,\ldots,0)\)).
  Note that
  \[
    \Ob\bigl(\square^+_S\bigr) - \{\emptyset\} = \Ob\bigl(\square_S\bigr)
  \]
  and that
  \(\square^+_S\) is the augmented category of \(\square_S\) (see Definition
  \ref{def:gnapgp17}).
  \par Now let \(n \ge -1\) be an integer.
  We set
  \[
    \square^+_n \coloneqq \square_{\underline{n}}^+ =
    \underline{2}^{n+1} \qquad\text{and}\qquad
    \square_n \coloneqq \square_{\underline{n}}.
  \]
  The objects of \(\square^+_n\) (resp.\ \(\square_n\)) can be identified with
  sequences (resp.\ nonzero sequences)
  \(\alpha = (\alpha_0,\alpha_1,\ldots,\alpha_n)\) where \(\alpha_i
  \in \{0,1\}\) for all \(0 \le i \le n\).
  For \(n = -1\) we have \(\square^+_{-1} = \underline{1}\) and for \(n = 0\) we
  have \(\square^+_0 = \underline{2}\).
  The category \(\square^+_n\) (resp.\ \(\square_n\)) is called the
  \textsl{standard augmented cubical} (resp.\ \textsl{standard cubical})
  \textsl{category of order \(n\)}.
  The 1-diagrams and 1-codiagrams of type \(\square^+_n\) (resp.\
  \(\square_n\)) of a category \(\sC\)
  are called \textsl{augmented cubical (resp.\ cubical) 1-diagrams and 1-codiagrams
  of \(\sC\)}, respectively.
  The 1-diagrams of \(\square^+_n\) (resp.\ \(\square_n\)) of a category \(\sC\)
  are also called 
  \textsl{augmented cubical objects} (resp.\ \textsl{cubical
  objects}) of \(\sC\).
  \par Since \(\square^+_S\) and \(\square_S\) arise as the orderable categories
  associated to finite partially ordered sets as in Example
  \ref{ex:notorderedset},
  we will omit the notation \(\Ob(-)\) when discussing objects in
  \(\square^+_S\) and \(\square_S\).
  If \(\alpha \in \square^+_S\), we set
  \[
    \abs{\alpha} \coloneqq \sum_{i \in S} \alpha(s).
  \]
\end{definition}
\begin{definition}[The categories \(\Pi^+\) and \(\Pi\)
  {\cite[(1.1.1) and (1.4.1)]{GNA02}}]
  We denote by \(\Pi^+\) the category whose objects are \(\square^+_S\) for finite
  sets \(S\) and whose morphisms are morphisms of the form \(\square^+_u\colon
  \square^+_S \to \square^+_T\) where \(\square^+_u(\alpha) = u(\alpha)\)
  for injective maps \(u \colon S \to T\).
  The category \(\Pi^+\) has a symmetric monoidal structure given by the
  disjoint union, where the unit is the empty set \(\emptyset\).
  \par We denote by \(\Pi\) the category whose objects are finite products
  \(\square_S \coloneqq \prod_{i \in I} \square_{S_i}\) for a finite family of finite
  nonempty sets \((S_i)_{i \in I}\), and whose morphisms are of the form
  \(\square_u\colon \square_S \to \square_T\) where \(\square_u(\alpha) = u(\alpha)\)
  for maps \(u \colon S \to T\) that are injective on each factor \(S_i\).
  The category \(\Pi\) has a symmetric monoidal structure: the product of
  \(\square_S\) where \(S = (S_i)_{i \in I}\) and \(\square_T\) where
  \(T = (T_j)_{j \in J}\) is \(\square_{S \sqcup T}\) where
  \(S \sqcup T = (U_k)_{k \in K}\), and where \(K = I \sqcup J\) and
  \[
    U_k = \begin{cases}
      S_k & \text{if}\ k \in I,\\
      T_k & \text{if}\ k \in J.
    \end{cases}
  \]
  The unit of this monoidal structure on \(\Pi\) is the family parametrized by the empty
  set.
  \par The maps \(\square^+_u\) and \(\square_u\) are strictly increasing, where
  \(\square_S = \prod_{i \in I} \square_{S_i}\) is given the product order.
\end{definition}
We now define the cubical analogue of face maps.
\begin{definition}[Face functors
  {\cite[Expos\'e I, (1.16) and (1.17)]{GNAPGP88}}]\label{def:gnapgp117}
  Let \(n \ge -1\) be an integer and let \(i\) be an integer such that \(0 \le i
  \le n+1\).
  We denote by \(\delta_i\colon \square^+_n \to \square^+_{n+1}\) the strictly
  increasing map defined by
  \begin{align*}
    \delta_i(\alpha) &=
    (\alpha_0,\alpha_1,\ldots,\alpha_{i-1},0,\alpha_i,\alpha_{i+1},\ldots,\alpha_n)
    \shortintertext{for \(\alpha \in \square^+_n\).
    For \(n = -1\), we set \(\delta_0(0) = 0\).
    A \textsl{face functor} is a composition of a finite number of maps of the
    form \(\delta_i\):}
    \delta &= \delta_{i_p} \circ \cdots \circ \delta_{i_2} \circ \delta_{i_1}
    \colon \square^+_n \longrightarrow \square^+_{n+p}.
  \end{align*}
  A face functor \(\delta\colon \square^+_n \to \square^+_{n+p}\) induces a
  functor \(\delta\colon \square_n \to \square_{n+p}\) by restriction, which we
  also call a \textsl{face functor}.
\end{definition}
\subsection{Inverse images and direct images}
We define inverse images and direct images for diagrams and
codiagrams.
We can define inverse images for diagrams and codiagrams of arbitrary type.
\begin{definition}[Inverse images of 1-diagrams and 1-codiagrams
  {\citeleft\citen{GNAPGP88}\citemid Expos\'e I, (1.2)\citepunct
  \citen{GNA02}\citemid (1.2.2)\citeright}]
  Let \(\varphi\colon I \to J\) be a functor between small categories.
  If \(Y\) is a 1-diagram (resp.\ 1\hyph{}codiagram) of \(\sC\) of type \(J\),
  the functor \(Y \circ \varphi^\op\) (resp.\ \(Y \circ \varphi\))
  is a 1-diagram (resp.\ 1-codiagram) of
  \(\sC\) of type \(I\), which we denote by \(\varphi^{\op*}(Y)\) (resp.\
  \(\varphi^*(Y)\)) or \(Y \times_I J\) depending on the context.
  \par If \(X\) is an object of \(\sC\), we have functors
  \[
    i_I \colon \sC \longrightarrow \Diag_I(\sC) \qquad \text{and} \qquad
    i_I \colon \sC \longrightarrow \Codiag_I(\sC)
  \]
  mapping \(X\) to the constant 1-diagram (resp.\ 1-codiagram)
  which sending every object of \(I^\op\) (resp.\ \(I\)) to \(X\).
  This is the same functor as the inverse image functor for \(I^\op \to
  \underline{1}\) (resp.\ \(I \to \underline{1}\)).
\end{definition}
We can use inverse images to define realization and corealization.
\begin{definition}[{\(\Real_\fI(\sC)\) and \(\Coreal_\fI(\sC)\)
  \cite[(1.2.1)]{GNA02}}]\label{def:gna121}
  Let \(\fI\) be a subcategory of \(\Cat\).
  Consider the following functors:
  \[
    \begin{tikzcd}[cramped,row sep=0,column sep=1.475em]
      \Diag_\fI(\sC)\colon &[-2.125em]
      \fI^\op \rar & \Cat & & \Codiag_\fI(\sC)\colon &[-2.125em]
      \fI^\op \rar & \Cat\\
      & I \rar[mapsto] & \Diag_I(\sC)
      & \text{and} & & I \rar[mapsto] & \Codiag_I(\sC)\\
      & \varphi \rar[mapsto] & \varphi^{\op}* & & & \varphi \rar[mapsto] &
      \varphi^*\mathrlap{.}
    \end{tikzcd}
  \]
  By the Grothendieck construction \cite[Theorem B1.3.5]{Joh021}, these functors
  correspond to categories \(\Real_\fI(\sC)\) and \(\Coreal_\fI(\sC)\) cofibered
  and fibered over \(\fI\), respectively.
\end{definition}
We only define direct images for (a slight generalization of)
\emph{cubical} diagrams and codiagrams.
\begin{definition}[Direct images of
  cubical 1-diagrams and 1-codiagrams {\cite[(1.2.2)]{GNA02}}]
  Let \(\sD\) be a category with an initial object \(0\).
  If \(\delta\colon \square \to \square'\) is a morphism of \(\Pi\), the
  \textsl{direct image of a 1-diagram} \(X\colon \square^\op \to \sD\) and
  the \textsl{direct image of a 1-codiagram} \(X\colon \square \to \sD\) are defined
  by
  \begin{align*}
    (\delta_*X)_\beta &\coloneqq \begin{cases}
      X_\alpha & \text{if}\ \beta = \delta(\alpha)\ \text{for}\ \alpha \in
      \square\\
      \hfil\emptyset & \text{if}\ \beta \in \square' - \delta(\square)
    \end{cases}\\
    (\delta_*X)^\beta &\coloneqq \begin{cases}
      X^\alpha & \text{if}\ \beta = \delta(\alpha)\ \text{for}\ \alpha \in
      \square\\
      \hfil\emptyset & \text{if}\ \beta \in \square' - \delta(\square)
    \end{cases}
  \end{align*}
  respectively.
  The direct image defines functors
  \[
    \begin{tikzcd}[cramped,row sep=0,column sep=1.475em]
      \delta_*\colon &[-2.125em] \Diag_\square(\sD) \rar & \Diag_{\square'}(\sD)\\
      \delta_*\colon &[-2.125em] \Codiag_\square(\sD) \rar & \Codiag_{\square'}(\sD)
    \end{tikzcd}
  \]
  where \(\delta_*u\) for a morphism \(u\) is defined using composition of
  functors.
\end{definition}

\begingroup
\makeatletter
\renewcommand{\@secnumfont}{\bfseries}
\part{Cubical hyperresolutions and cubical descent}
\label{part:cubical}
\makeatother
\endgroup
In this part, we construct cubical hyperresolutions for all reduced spaces of
the types
\((\ref{setup:introalgebraicspaces})\textnormal{--}(\ref{setup:introadicspaces})\).
We then prove that cubical hyperresolutions satisfy a form of
descent that applies to functors and to cohomology.\bigskip
\section{Very weak resolutions and augmented cubical hyperresolutions}
In this section, we construct very weak resolutions (Theorem \ref{thm:gnapgp26}) and 
augmented cubical hyperresolutions (Theorem \ref{thm:gnapgp215}) of \(I\)-spaces.
There are two main differences
in this section compared to \cite[Expos\'e I, \S2]{GNAPGP88}.
First, we work with spaces of the form
\((\ref{setup:introalgebraicspaces})\textnormal{--}(\ref{setup:introadicspaces})\)
and therefore generalize the results stated for complex varieties and complex
analytic spaces in
\cite{GNAPGP88} to these other categories of spaces.
Second, we adopt the definition of a \textsl{very weak resolution} (see
Definition \ref{def:veryweakres}) from
\cite[Definition 5]{Ste17} to avoid the counterexamples to \cite[Expos\'e I,
Th\'eor\`eme 2.6]{GNAPGP88} constructed by Steenbrink in \cite[Example 1]{Ste17}.
Note that Steenbrink in \cite{Ste17} explains how this new notion of a very weak resolution
suffices to construct cubical hyperresolutions of cubical varieties.
\subsection{Conventions}
\par Throughout this section, we denote by \(\Sp\) a subcategory of a
(not necessarily small)
category of \emph{reduced} spaces of one of the types
\((\ref{setup:introalgebraicspaces})\textnormal{--}(\ref{setup:introadicspaces})\)
that is essentially stable under fiber products, immersions, and proper
morphisms.
We will assume \(\Sp\) is chosen according to the strategy outlined in
Remark \ref{rem:smallspaces}.
We denote by \(I\) a small category.
While we will remind the reader of these assumptions in statements of results in
this section, we will not restate these assumptions in definitions or remarks.
\par We will also use the following terminology.
\begin{definition}\label{def:ispaces}
  An \textsl{\(I\)-space} is an object of the category
  \(\Diag_I(\Sp)\).
\end{definition}
We will also use the following notion of pairs of spaces.
\begin{definition}[cf.\ {\citeleft\citen{GNAPGP88}\citemid
  Expos\'e I, (3.11.1) and Expos\'e IV, \S1F\citeright}]
  \label{def:closedpairs}
  The \textsl{category of pairs of spaces} is the category whose objects are
  pairs \((X,X')\) where \(X\) is a space and \(X'\) is a closed subspace of
  \(X\), and whose morphisms \(f\colon (X,X') \to (Y,Y')\) are morphisms
  \(f\colon X \to Y\) of spaces such that \(f^{-1}(Y') \subseteq X'\).
  An \textsl{\(I\)-pair of spaces} is an \(I\)-object in the category of pairs
  of spaces.
  \par A pair \((X,X')\) is \textsl{regular} if \(X\) is regular
  and if \(X'\) is a union of connected components of \(X\) and normal crossing
  divisors on other connected components of \(X\).
  In the latter case, we mean that \(X'\)
  is \'etale-locally a simple normal crossing divisor on those connected
  components.
\end{definition}
Note that if \((X,X')\) is a pair of spaces, then \(X - X'\) is an open
\(I\)-subspace of \(X\) (see Definition \ref{def:propsforispaces}) but \(X'\) is
not an \(I\)-space in general.
\subsection{\emph{I}-spaces and discriminants}
We define the following terminology for \(I\)-spaces.
\begin{definition}[cf.\ {\cite[Expos\'e I, (2.1)]{GNAPGP88}}]\label{def:propsforispaces}
  Let \(X\) be an \(I\)-space and let \(f\) be a morphism of \(I\)-spaces.
  Let \(\bP\) be a property of spaces or morphisms of spaces.
  \begin{enumerate}[label=\((\roman*)\)]
    \item We say that \(X\) \textsl{is \(\bP\)} if \(X_i\) is \(\bP\) for every
      \(i \in \Ob(I)\).
    \item We say that \(X\) \textsl{is relatively \(\bP\)} if \(X_u\) is \(\bP\)
      for every \(u \in \Mor(I)\).
    \item We say that \(f\) \textsl{is \(\bP\)} if \(f_i\) is \(\bP\) for every
      \(i \in \Ob(I)\).
  \end{enumerate}
\end{definition}
We can now define proper modifications of \(I\)-spaces and the loci where these
modifications are not isomorphisms.
The notation for the discriminant is from \cite[Definition 5.14(2)]{PS08}.
\begin{definition}[Discriminants
  {\cite[Expos\'e I, (2.1) and D\'efinition 2.2]{GNAPGP88}}]
  Let \(f\colon X \to S\) be a proper morphism of spaces.
  We say that \(f\) is a \textsl{proper modification} if there
  exists a closed subspace \(Z\) of \(S\) such that \(f\) induces
  an isomorphism
  \begin{equation}
    f\bigr\rvert_{X - f^{-1}(Z)}\colon X - f^{-1}(Z)
    \overset{\sim}{\longrightarrow}
    S - Z\label{eq:propermod}
  \end{equation}
  where \(S-Z\) is dense in \(S\).
  For an arbitrary morphism \(f\colon X \to S\) of spaces,
  the \textsl{discriminant of \(f\)} is the smallest closed subspace
  \(\Delta(f)\) of \(S\) such that \(f\) induces an isomorphism \eqref{eq:propermod} for
  \(Z = \Delta(f)\).
  Note that if \(f\) is not a proper modification, then \(\Delta(f)\)
  can be the whole space \(S\).
  \par Now let \(f\colon X \to S\) be a morphism of \(I\)-spaces.
  The \textsl{discriminant of \(f\)} is the smallest closed sub-\(I\)-space \(\Delta(f)\) of
  \(S\) such that \(f\) induces an isomorphism
  \[
    f_i\bigr\rvert_{X_i - f_i^{-1}(\Delta(f)_i)}\colon X_i -
    f_i^{-1}\bigl(\Delta(f)\bigr)
    \overset{\sim}{\longrightarrow} S_i - \Delta(f)_i
  \]
  for every \(i \in \Ob(I)\).
\end{definition}
\subsection{Very weak resolutions}\label{sect:veryweakres}
To define an analogue of a resolution of singularities for \(I\)-spaces, we
first define the dimension of an \(I\)-space.
See \S\ref{sect:dimension} for the definition of dimension in 
\(\Sp\).
\begin{definition}[Dimension of an \(I\)-space
  {\cite[Expos\'e I, D\'efinition 2.9]{GNAPGP88}}]
  Let \(S\) be an \(I\)-space.
  We say that \(S\) is \textsl{finite-dimensional} if the set of integers
  \[
    \Set[\big]{\dim(S_i) \given i \in \Ob(I)}
  \]
  is bounded.
  If \(S\) is finite-dimensional, the \textsl{dimension of \(S\)} is
  \[
    \dim(S) \coloneqq \sup\Set[\big]{\dim(S_i) \given i \in \Ob(I)}.
  \]
\end{definition}
We define the analogue of a resolution of singularities for \(I\)-spaces as
follows, adjusted for \(I\)-pairs following \citeleft\citen{GNAPGP88}\citemid
Expos\'e IV, (1.24)\citeright.
\begin{definition}[Very weak resolutions;
{cf.\ \citeleft\citen{Ste17}\citemid Definition 5\citeright}]\label{def:veryweakres}
  Let \(f\colon X \to S\) be a proper morphism of \(I\)-spaces and let \(D
  \coloneqq \Delta(f)\) be the discriminant of \(f\).
  We say that \(f\) is a \textsl{very weak resolution of \(S\)} if \(X\) is
  regular and if
  \[
    \dim(D_i) < \dim(S)
    \qquad \text{and} \qquad
    \dim\bigl(f_i^{-1}(D_i)\bigr) < \dim(S)
  \]
  for all \(i \in \Ob(I)\).
  \par Now suppose \((S,S')\) is an \(I\)-pair of spaces.
  A very weak resolution \(f\colon X \to S\) is a \textsl{very weak resolution
  of the \(I\)-pair \((S,S')\)} if \((X,f^{-1}(S'))\) is a regular \(I\)-pair.
\end{definition}
\begin{remark}\label{rem:steenbrink}
  We cannot use the definition of a \textsl{resolution} in \cite[Expos\'e I, D\'efinition
  2.5]{GNAPGP88} (called a \textsl{weak resolution} in \cite[Definition 10.73]{Kol13})
  since even for complex varieties, a resolution of this form
  may not exist by an example of Steenbrink \cite[Example 1]{Ste17}.
  See Remark \ref{rem:ste17proofcomment}.
\end{remark}
We can now state one of the main results in this section.
Note that
\((\ref{thm:gnapgp26factorsthrough})\textnormal{--}(\ref{thm:gnapgp26compact})\)
are not stated in \cite{GNAPGP88},
although the construction in \cite{GNAPGP88}
yields a very weak resolution that satisfies
these two properties.
This flexibility in choosing \(X\) will be necessary in \S\ref{sect:gna02}
because we do not have versions of the Chow--Hironaka lemma from
\citeleft\citen{Hir64}\citemid pp.\ 144--145\citepunct \citen{Hir75}\citemid p.\
505\citeright.
See \S\ref{sect:blowingup} for the definition of blowups
used below.
\begin{theorem}[Very weak resolutions exist;
  cf.\ {\citeleft\citen{GNAPGP88}\citemid Expos\'e I, Th\'eor\`eme
  2.6 and (3.11.2) and Expos\'e IV, (1.24)\citepunct
  \citen{Ste17}\citemid Theorem 1\citeright}]\label{thm:gnapgp26}
  Suppose that \(\Sp\) is a subcategory of a category of reduced spaces of one of the
  types
  \((\ref{setup:introalgebraicspaces})\textnormal{--}(\ref{setup:introadicspaces})\)
  that is essentially stable under fiber products, immersions, and proper
  morphisms.
  Let \(I\) be a finite orderable category and let \((S,S')\) be an \(I\)-pair
  of spaces.
  In cases \((\ref{setup:introalgebraicspaces})\) and
  \((\ref{setup:introformalqschemes})\), suppose that
  \(S\) is quasi-excellent of equal characteristic zero.
  In cases \((\ref{setup:introberkovichspaces})\),
  \((\ref{setup:introrigidanalyticspaces})\), and
  \((\ref{setup:introadicspaces})\), suppose that \(k\) is of characteristic
  zero.
  Then, there exists a very weak resolution
  \[
    f\colon \bigl(X,f^{-1}(S')\bigr) \longrightarrow (S,S')
  \]
  of the \(I\)-pair \((S,S')\).
  Moreover:
  \begin{enumerate}[label=\((\roman*)\),ref=\roman*]
    \item\label{thm:gnapgp26factorsthrough}
      Let \(g\colon Y \to S\) be a proper morphism such that \(g_i\) is a proper
      modification for every \(i \in \Ob(I)\).
      We can construct \(f\) such that \(f\) factors through \(g\).
    \item\label{thm:gnapgp26locallyprojective}
      If the morphism \(g\) in \((\ref{thm:gnapgp26factorsthrough})\)
      is locally a birational/bimeromorphic blowup,
      then we can construct \(f\) such that the morphisms in
      \(X_i\) are locally birational/bimeromorphic blowups and the morphisms
      \(f_i\) are locally birational/bimeromorphic blowups for every \(i \in
      \Ob(I)\), where locality is in terms of affinoid subdomains of \(S\).
    \item\label{thm:gnapgp26compact}
      If \(S\) is compactifiable
      and the morphism \(g\) in
      \((\ref{thm:gnapgp26factorsthrough})\) is a birational/bimeromorphic blowup, then
      we can construct \(f\) such that the morphisms in
      \(X_i\) are all birational/bimeromorphic blowups and the morphisms
      \(f_i\) are birational/bimeromorphic blowups for every \(i \in \Ob(I)\).
  \end{enumerate}
\end{theorem}
To prove Theorem \ref{thm:gnapgp26},
we first prove two preliminary results.
See \S\ref{sect:irred} for a discussion on irreducibility and irreducible
components in \(\Sp\).
We use \(\abs{X}\) to denote the topological space underlying \(X\) (see
\cite[\href{https://stacks.math.columbia.edu/tag/04Y8}{Tag
04Y8}]{stacks-project} for the definition
in case \((\ref{setup:introalgebraicspaces})\)).
\begin{lemma}[Strict transforms; cf.\ {\cite[Expos\'e I, Lemme 2.6.1]{GNAPGP88}}]
  \label{lem:gnapgp261}
  Suppose that \(\Sp\) is a subcategory of a category of reduced spaces of one of the
  types
  \((\ref{setup:introalgebraicspaces})\textnormal{--}(\ref{setup:introadicspaces})\)
  that is essentially stable under fiber products and immersions.
  Let \(I\) be a small category.
  Let \(X\), \(Y\), and \(Y'\) be irreducible spaces in \(\Sp\),
  let \(a\colon X \to Y\) be a morphism, and
  let \(f\colon Y' \to Y\) be a proper modification.
  Suppose that \(a(\abs{X})\) is dense in \(\abs{Y}\).
  \begin{enumerate}[label=\((\roman*)\),ref=\roman*]
    \item\label{lem:gnapgp261i}
      Consider the category whose objects are triples \((Z,b,g)\) where \(Z\) is a
      space in \(\Sp\), \(g\colon Z \to X\) is a proper modification, and \(b
      \colon Z \to Y'\) is a morphism such that the diagram
      \begin{equation}\label{eq:gnapgp261i}
        \begin{tikzcd}
          Z \rar{b}\dar[swap]{g} & Y' \dar{f}\\
          X \rar{a} & Y
        \end{tikzcd}
      \end{equation}
      commutes.
      This category has a final object which we call the \textsl{strict
      transform} of \(a\) by \(f\).
    \item\label{lem:gnapgp261ii}
      Given a proper modification \(g\colon Z \to X\), there exists a
      morphism \(b\colon Z \to Y'\) such that the diagram \eqref{eq:gnapgp261i}
      commutes, in which case \(b(\abs{Z})\) is dense in \(\abs{Y'}\).
  \end{enumerate}
\end{lemma}
\begin{proof}
  We first prove \((\ref{lem:gnapgp261i})\).
  Let \(D \coloneqq \Delta(f)\) be the discriminant of \(f\).
  Set \(U = Y - D\), \(U' = f^{-1}(U)\), and \(V = a^{-1}(U)\).
  Let \(X'\) be the closure of \((U' \times_U V)_\red\) in \((Y' \times_Y
  X)_\red\) in the sense of \S\ref{sect:closures}.
  Then, the projection morphisms from \(Y' \times_Y X\) to \(Y'\) and \(X\)
  define morphisms \(a'\colon X' \to Y'\) and \(f'\colon X' \to X\),
  respectively.
  By the universal property of fiber products and reductions, we see that
  \((X',a',f')\) is the strict transform of \(a\) by \(f\).
  \par For \((\ref{lem:gnapgp261ii})\), it suffices to note that \(f\) is an
  isomorphism away from the discriminant \(D\) of \(f\), and hence
  \(\overline{b(\abs{Z})} \supseteq \overline{\abs{U'}} = \abs{Y'}\).
\end{proof}
\begin{lemma}[Upper envelopes;
  {cf.\ \citeleft\citen{GNAPGP88}\citemid Expos\'e I, Lemme
  2.6.2\citeright}]\label{lem:gnapgp262}
  Suppose that \(\Sp\) is a subcategory of a category of reduced spaces of one of the
  types
  \((\ref{setup:introalgebraicspaces})\textnormal{--}(\ref{setup:introadicspaces})\)
  that is essentially stable under fiber products and immersions.
  Let \(I\) be a small category.
  Let \(X_0\) be an irreducible space in \(\Sp\) and let
  \(\Set{f_r\colon X_r \to X_0}_{1 \le r \le n}\)
  be a finite family of proper
  modifications of \(X_0\).
  Consider the category whose objects are families \(\Set{h_r\colon Z \to
  X_r}_{1 \le r \le n}\) such that \(f_r \circ h_r = f_1 \circ h_1\) for all
  \(r\).
  This category has a final object
  \[
    \sup_{1 \le r \le n} \Set[\big]{f_r \colon X_r \longrightarrow X_0},
  \]
  which we call the \textsl{upper envelope} for
  the family \(\{f_r\}_{1 \le r \le n}\).
  Moreover, if the \(f_r\) are (locally) birational/bimeromorphic blowups,
  then their upper envelope is (locally) a birational/bimeromorphic blowup.
\end{lemma}
The notation for the upper envelope is from \cite[Lemma-Definition 5.24]{PS08}.
\begin{proof}
  Let \(D_r \coloneqq \Delta(f_r)\) be the discriminant of \(f_r\) for each \(r\).
  Set \(D = \bigcup_r D_r\) and \(U = X_0 - D\).
  Let
  \[
    X \coloneqq \overline{U} \subseteq \bigl(X_1 \times_{X_0} X_2 \times_{X_0}
    \cdots \times_{X_0} X_n\bigr)_\red
  \]
  and let \(g_r\colon X \to X_r\) be the canonical morphisms induced by the
  projection morphisms for each \(r\).
  By the universal property of fiber products and reductions, we see that
  \(\{g_r\}_{1 \le r \le n}\) is the upper envelope of \(\{f_r\}_{1 \le r \le
  n}\).
  The ``Moreover'' statement is true by definition.
\end{proof}
Next, we define an auxiliary \(I\)-set that will be used to construct the very weak
resolution in Theorem \ref{thm:gnapgp26}.
See \S\ref{sect:irred} for a discussion on irreducibility and irreducible
components in \(\Sp\).
\begin{citeddef}[{\cite[Expos\'e I, (2.6.3)]{GNAPGP88}}]
  Let \(S\) be an \(I\)-space.
  We associate to \(S\) an \(I\)-set \(\Sigma S\) in the following manner.
  If \(i \in \Ob(I)\), then \(\Sigma S_i\) is the set of irreducible, reduced
  subspaces \(S_{i,\alpha}\) of \(S_i\) for which there exists a morphism
  \(u\colon i \to j\) of \(I\) and an irreducible component \(T_j\) of \(S_j\)
  such that \(S_{i,\alpha} = \overline{S_u(T_j)}\).
  If \(u\colon i \to j\) is a morphism of \(I\) and \(\beta \in \Sigma S_j\),
  then we define \(\Sigma S_u(\beta)\) to be the index corresponding to
  \(\overline{S_u(S_{j,\beta})}\).
\end{citeddef}
\par Note that if \(I\) is finite to the right and both \(S\) and every space
\(S_i\) has finitely many irreducible components,
then the set \(\Sigma S_i\) is
finite for all \(i \in \Ob(I)\).
However, if one starts with a non-algebraic object \(X\), for example a complex
analytic space that is not compactifiable,
the set \(\Sigma S_i\) may be infinite even if \(X\) is
irreducible.
See \cite[Remark 3.2]{CG14}.\medskip
\par We can now prove Theorem \ref{thm:gnapgp26}.
\begin{proof}[Proof of Theorem \ref{thm:gnapgp26}]
  We prove Theorem \ref{thm:gnapgp26} in all cases
  \((\ref{setup:introalgebraicspaces})\textnormal{--}(\ref{setup:introadicspaces})\)
  simultaneously.
  Replacing \(S\) by \(Y\), we may assume that \(g = \id_S\).\smallskip
  \par We define a regular \(I\)-space \(X\) and a
  morphism \(f\colon X \to S\) of \(I\)-spaces by ascending induction on
  \(I\) so that the following properties hold:
  \begin{enumerate}[label=\((\alph*)\)]
    \item \(f\) induces an isomorphism of \(I\)-sets \(\Sigma f\colon \Sigma X
      \to \Sigma S\).
    \item For every \(i \in \Ob(I)\) and every \(\alpha \in \Sigma X_i\), the
      morphisms \(f_i\) restrict to proper modifications \(X_{i,\alpha} \to
      S_{i,\alpha}\).
      In situation \((\ref{thm:gnapgp26locallyprojective})\) (resp.\
      \((\ref{thm:gnapgp26compact})\)), these morphisms are locally
      birational/bimeromorphic blowups
      (resp.\ are birational/bimeromorphic blowups).\smallskip
    \item \(\dim(D_i) < \dim(S)\) and \(\dim(f_i^{-1}(D_i)) < \dim(S)\) for
      every \(i \in \Ob(I)\).
  \end{enumerate}
  \par For the base case, let \(i \in \Ob(I)\) be a minimal element.
  By resolution of singularities for pairs \cite[Theorems 1.1.11 and
  1.1.13]{Tem18},
  for each \(\alpha \in \Sigma S_i\), there exists a resolution of singularities
  \(X_{i,\alpha} \to S_{i,\alpha}\) for the pair \((S_{i,\alpha},S_{i,\alpha}
  \cap S'_i)\) that is locally a birational/bimeromorphic
  blowup (where locality is in terms of affinoid subdomains of \(S\)),
  and in situation \((\ref{thm:gnapgp26compact})\), the adverb
  ``locally'' can be omitted.
  We set
  \[
    X_i \coloneqq \coprod_{\alpha \in \Sigma S_i} X_{i,\alpha}
  \]
  and let \(f_i\colon X_i \to S_i\) be the coproduct over \(\alpha \in \Sigma S_i\)
  of the compositions \(X_{i,\alpha} \to S_{i,\alpha} \to S_i\).
  Note that over every affinoid subdomain \(U\) of \(S_i\), the space \(X_i
  \times_{S_i} U\) is proper over \(U\) since there are only finitely many
  components of \(X_i\) lying over \(U\), and hence \(X_i\) is proper over
  \(S_i\).
  We then see that
  \begin{equation}
    \begin{gathered}
      \dim(D_i) < \sup_{\alpha \in \Sigma S_i}
      \Set[\big]{\dim(S_{i,\alpha})} \le \dim(S_i) \le \dim(S)\\
      \dim\bigl(f_i^{-1}(D_i)\bigr) < \dim(X_i) = \sup_{\alpha \in \Sigma S_i}
      \Set[\big]{\dim(S_{i,\alpha})} \le \dim(S_i) \le  \dim(S).
    \end{gathered}\smallskip
    \label{eq:veryweakineq}
  \end{equation}
  \par For the inductive step, suppose that \(i \in \Ob(I)\) is arbitrary and
  that we have constructed \(X_j\) for all \(j < i\).
  For each morphism \(u \colon j \to i\) of \(I\) such that \(j \ne i\) and each
  \(\alpha \in \Sigma S_i\), we set
  \[
    \beta \coloneqq \Sigma S_u(\alpha) \in \Sigma S_j
  \]
  and let \(W^u_{i,\alpha}\) be the strict transform (in the sense of Lemma
  \ref{lem:gnapgp261}\((\ref{lem:gnapgp261i})\)) of the dominant morphism
  \(S_{i,\alpha} \to S_{j,\beta}\) under the proper modification \(X_{j,\beta}
  \to S_{j,\beta}\).
  Note the proper morphisms \(X_{j,\beta} \to S_{j,\beta}\) are locally
  birational/bimeromorphic blowups (birational/bimeromorphic blowups in
  situation \((\ref{thm:gnapgp26compact})\))
  by the inductive hypothesis.
  Let
  \[
    \Set[\big]{W_{i,\alpha} \longrightarrow W^u_{i,\alpha}}_u \coloneqq
    \sup_{\substack{u\colon j \to i\\j \ne i}}
    \Set[\big]{W^u_{i,\alpha} \longrightarrow S_{i,\alpha}}_u
  \]
  be the upper envelope (in the
  sense of Lemma \ref{lem:gnapgp262}) of the
  family of proper modifications \(\{W^u_{i,\alpha} \to S_{i,\alpha}\}_u\).
  Consider the composition
  \[
    h_{i,\alpha} \colon W_{i,\alpha} \longrightarrow S_{i,\alpha}.
  \]
  Let \(X_{i,\alpha} \to W_{i,\alpha}\) be a resolution of singularities for the
  pair \((W_{i,\alpha},h_{i,\alpha}^{-1}(S'_i))\) for each
  \(\alpha\), which exists
  by
  \cite[Theorems 1.1.11 and 1.1.13]{Tem18}.
  Moreover, these resolutions of singularities are locally
  birational/bimeromorphic blowups and in situation
  \((\ref{thm:gnapgp26compact})\), they are birational/bimeromorphic blowups
  globally.
  Then, the composition
  \begin{equation}\label{eq:veryweakresindcomp}
    X_{i,\alpha} \longrightarrow W_{i,\alpha} 
    \xrightarrow{h_{i,\alpha}} S_{i,\alpha}
  \end{equation}
  is a resolution of singularities for the pair \((S_{i,\alpha},S_{i,\alpha} \cap
  S')\) that does not depend on \(u\).
  We set
  \[
    X_i \coloneqq \coprod_{\alpha \in \Sigma S_i} X_{i,\alpha}
  \]
  and let \(f_i\colon X_i \to S_i\) be the coproduct over \(\alpha \in \Sigma S_i\)
  of the compositions \eqref{eq:veryweakresindcomp}.
  Finally, if \(u\colon j \to i\) is a morphism of \(I\), we define the morphism
  \(X_u\colon X_i \to X_j\) as the coproduct over \(\alpha \in \Sigma S_i\)
  of the compositions
  \[
    X_{i,\alpha} \longrightarrow W^u_{i,\alpha} \longrightarrow X_{j,u(\alpha)}.
  \]
  Note that over every affinoid subdomain \(U\) of \(S_i\), the morphism \(X_i
  \times_{S_i} U \to X_j \times_{S_i} U\) is proper.
  Since \(X_i\) and \(X_j\) only have finitely many components lying over \(U\),
  we see that \(X_i \to X_j\) is proper.
  By Lemma \ref{lem:gnapgp261}\((\ref{lem:gnapgp261ii})\), the \(X_i\),
  \(X_u\), and \(f_i\) defined above define a 
  \(I\)-space \(X\) and a
  morphism \(f\colon X \to S\) of \(I\)-spaces satisfying the extra properties
  in \((\ref{thm:gnapgp26locallyprojective})\) and
  \((\ref{thm:gnapgp26compact})\) under those respective hypotheses.
  Since \(X\) is regular by construction and we have
  \begin{equation}\label{eq:veryweakineq2}
    \begin{gathered}
      \dim(D_i) < \sup_{\alpha \in \Sigma S_i}
      \Set[\big]{\dim(S_{i,\alpha})} \le \sup_{\substack{j < i
      \\\beta \in \Sigma S_j}} \Set[\big]{\dim(S_{j,\beta})} \le \dim(S)\\
      \dim\bigl(f_i^{-1}(D_i)\bigr) < \dim(X_i) = 
      \sup_{\alpha \in \Sigma S_i}
      \Set[\big]{\dim(S_{i,\alpha})} \le \sup_{\substack{j < i
      \\\beta \in \Sigma S_j}} \Set[\big]{\dim(S_{j,\beta})} \le \dim(S)
    \end{gathered}
  \end{equation}
  by the inductive hypothesis,
  we see that \(f\colon X \to S\) is a very weak
  resolution of the pair \((S,S')\) satisfying the extra properties in
  \((\ref{thm:gnapgp26locallyprojective})\) and \((\ref{thm:gnapgp26compact})\)
  under those respective hypotheses.
\end{proof}
\begin{remark}\label{rem:ste17proofcomment}
  Steenbrink's definition of a very weak resolution is necessary because
  to verify the definition of a resolution in \cite[Expos\'e I, D\'efinition
  2.5]{GNAPGP88}, one would have to prove stronger versions of the inequalities
  \eqref{eq:veryweakineq} and \eqref{eq:veryweakineq2} where
  \(\dim(S)\) is replaced by \(\dim(S_i)\).
  These stronger inequalities do not hold for Steenbrink's example \cite[Example
  1]{Ste17}.
\end{remark}
\subsection{Augmented 2-resolutions}
To define augmented cubical hyperresolutions, we first define a distinguished
class of Cartesian diagrams in \(\Sp\) and use it to define augmented
2-resolutions.
See \S\ref{sect:blowingup} for references to the notion of relative \(\PProj\)
and blowing up in \(\Sp\).
\begin{definition}[Acyclic squares {\citeleft\citen{GNAPGP88}\citemid Expos\'e
  I, Proposition 6.8\citepunct \citen{GNA02}\citemid D\'efinition
  2.1.1\citeright}]
  Consider the Cartesian square
  \begin{equation}\label{eq:acyclicsquare}
    \begin{tikzcd}
      \tilde{Y} \rar[hook]{j} \dar[swap]{g} & \tilde{X} \dar{f}\\
      Y \rar[hook]{i} & X
    \end{tikzcd}
  \end{equation}
  in \(\Sp\).
  We say the square \eqref{eq:acyclicsquare}
  is an \textsl{acyclic square} if the following properties hold:
  \begin{enumerate}[label=\((\roman*)\)]
    \item \(i\) is a closed immersion.
    \item \(f\) is proper.
    \item \(Y\) contains the discriminant of \(f\).
      In other words, \(f\) induces an isomorphism
      \[
        \tilde{X} - \tilde{Y} \overset{\sim}{\longrightarrow} X - Y.
      \]
  \end{enumerate}
  We say an acyclic square \eqref{eq:acyclicsquare}
  is an \textsl{elementary acyclic square} and that the
  morphism \(f\) is an \textsl{elementary proper modification} if the
  following additional properties hold:
  \begin{enumerate}[label=\((\roman*)\),resume]
    \item All spaces in \eqref{eq:acyclicsquare} are regular and irreducible.
    \item \(f\colon \tilde{X} \to X\) is the blowup of \(X\) along \(Y\) (see
      \S\ref{sect:blowingup}).
  \end{enumerate}
\end{definition}
We now define augmented 2-resolutions.
The definition we adopt is from \cite{GNAPGP88} and is
weaker than the definition from \cite{GNA02}.
We do so because
augmented 2-resolutions in the sense of \cite{GNA02} do not always
exist (see Remark \ref{rem:steenbrink2}).
Note that \cite{GNAPGP88} omits the adjective ``augmented'' in
``augmented 2-resolution.''
\begin{definition}[Augmented 2-resolutions;
  cf.\ {\citeleft\citen{GNAPGP88}\citemid Expos\'e I,
  D\'efinition 2.7\citepunct \citen{GNA02}\citemid p.\ 43\citeright}]
  \label{def:aug2res}
  Let \(S\) be an \(I\)-space and let \(Z_\bullet\) be a \((\square_1^+ \times
  I)\)-space.
  We say that \(Z_\bullet\) is a \textsl{augmented 2-resolution of \(S\)} if
  \(Z_\bullet\) is defined by the Cartesian square
  \begin{equation}\label{eq:def2res}
    \begin{tikzcd}
      Z_{11} \rar[hook] \dar & Z_{01} \dar{f}\\
      Z_{10} \rar[hook] & Z_{00}
    \end{tikzcd}
  \end{equation}
  of \(I\)-spaces
  where the following properties hold:
  \begin{enumerate}[label=\((\roman*)\),ref=\roman*]
    \item \(Z_{00} = S\).
    \item \eqref{eq:def2res} is an acyclic square in \(\Sp\) after restricting to
      each \(i \in \Ob(I)\).
    \item \(Z_{01}\) is a regular \(I\)-space.
    \item\label{def:aug2resdiscriminant} \(Z_{10}\) contains the discriminant
      of \(f\).
      In other words, \(f\) induces an isomorphism
      \[
        f\bigr\rvert_{Z_{01i} - Z_{11i}}\colon Z_{01i} - Z_{11i}
        \overset{\sim}{\longrightarrow} Z_{00i} - Z_{10i}
      \]
      for every \(i \in \Ob(I)\).
  \end{enumerate}
\end{definition}
\begin{remark}\label{rem:steenbrink2}
  We have changed condition \((\ref{def:aug2resdiscriminant})\) in Definition
  \ref{def:aug2res} compared to \cite{GNA02} because Steenbrink's example
  \cite[Example 1]{Ste17} (see Remarks \ref{rem:steenbrink} and
  \ref{rem:ste17proofcomment})
  does not have an augmented 2\hyph{}resolution in the sense
  of \cite[p.\ 43]{GNA02}.
  The same example shows this stronger version of an augmented 2-resolution and the
  corresponding notion of a cubical hyperresolution do not satisfy
  the analogues of Propositions \ref{prop:gnapgp28} and \ref{prop:gnapgp214}.
\end{remark}
The following result is immediate from the definitions.
\begin{proposition}[cf.\ {\cite[Expos\'e I, Proposition
  2.8]{GNAPGP88}}]\label{prop:gnapgp28}
  Suppose that \(\Sp\) is a subcategory of a category of reduced spaces of one of the
  types
  \((\ref{setup:introalgebraicspaces})\textnormal{--}(\ref{setup:introadicspaces})\)
  that is essentially stable under fiber products, immersions, and proper
  morphisms.
  Let
  \(I\) be a small category.
  Let \(K\) be an \(I\)-object of \(\Cat\) and let \(S\) be a
  \(\tot(K)\)-space.
  Then, \(Z_\bullet\) is an augmented 2-resolution of \(S\) if and only if
  \(Z_{\bullet i}\) is an augmented 2-resolution of the
  \(K_i\)-space \(S_i\) for every \(i \in \Ob(I)\).
\end{proposition}
Very weak resolutions yield augmented 2-resolutions.
\begin{proposition}[Augmented 2-resolutions exist;
  cf.\ {\cite[Expos\'e I, Proposition
  2.10]{GNAPGP88}}]\label{prop:gnapgp210}
  Suppose that \(\Sp\) is a subcategory of a category of reduced spaces of one of the
  types
  \((\ref{setup:introalgebraicspaces})\textnormal{--}(\ref{setup:introadicspaces})\)
  that is essentially stable under fiber products, immersions, and proper
  morphisms.
  Let
  \(I\) be an orderable small category that is finite to the right.
  Let \(S\) be a finite-dimensional \(I\)-space and let \(f\colon X \to S\) be a
  very weak resolution of \(S\).
  Then, there exists an augmented
  2-resolution \(Z_{\bullet\bullet}\) of \(S\) such that
  \(Z_{0\bullet} = (X \to S)\) and
  \[
    \dim(Z_{1\bullet}) < \dim(S).
  \]
\end{proposition}
\begin{proof}
  Let \(Z_{10} \coloneqq \Delta(f)\)
  be the discriminant of \(f\), which is a closed sub-\(I\)-space of
  \(S\).
  It then suffices to define \(Z_{11}\) by setting
  \(Z_{11i} \coloneqq (Z_{10i} \times_{S_i} X_i)_\red\)
  for every \(i \in \Ob(I)\).
\end{proof}
\subsection{Augmented cubical hyperresolutions}\label{sect:thmbplus}
The goal of this subsection is to construct a cubical
version of resolutions of singularities.
We start with the following definition.
\begin{definition}[Reductions {\cite[Expos\'e I, (2.11)]{GNAPGP88}}]
  Let \(r \ge 1\) be an integer and let \(X^n_\bullet\) be \((\square^+_n
  \times I)\)-spaces for each \(n \in \{1,2,\ldots,r\}\).
  Suppose that for every \(n\), the \((\square^+_{n-1} \times I)\)-spaces
  \(X^{n+1}_{00\bullet}\) and \(X^n_{1\bullet}\) are equal.
  We define the \((\square^+_r \times I)\)-space
  \begin{align*}
    Z_\bullet ={}& \rd\bigl(X^1_\bullet,X^2_\bullet,\ldots,X^r_\bullet\bigr),
    \shortintertext{which we call the \textsl{reduction of
    \((X^1_\bullet,X^2_\bullet,\ldots,X^r_\bullet)\)}, by induction on \(r\) as
    follows.
    If \(r = 1\), then \(Z_\bullet = X^1_\bullet\).
    If \(r = 2\), we define \(Z_{\bullet\bullet} =
    \rd(X^1_\bullet,X^2_{\bullet\bullet})\) by}
    Z_{\alpha\beta} \coloneqq{}& \begin{cases}
      X^1_{0\beta} & \text{if}\ \alpha = (0,0)\\
      X^2_{\alpha\beta} & \text{if}\ \alpha \in \square_1
    \end{cases}
  \end{align*}
  for each \(\beta \in \square^+_0\) with morphisms coming from \(X^1_{0\beta}\)
  and \(X^2_{\alpha\beta}\) using the equality \(X^1_{1\beta} =
  X^2_{00\beta}\).
  Finally, if \(r > 2\), we set
  \[
    Z_\bullet \coloneqq \rd\Bigl(
    \rd\bigl(X^1_\bullet,X^2_\bullet,\ldots,X^{r-1}_\bullet\bigr),X^r_\bullet\Bigr).
  \]
\end{definition}
We can now define augmented cubical hyperresolutions.
\begin{definition}[Augmented cubical hyperresolutions {\cite[Expos\'e I,
  (2.11)]{GNAPGP88}}]\label{def:gnapgp211}
  Let \(S\) be an \(I\)-space.
  An \textsl{augmented cubical hyperresolution of \(S\)} is a \((\square^+_r
  \times I)\)-space \(Z_\bullet\) for an integer \(r \ge 1\) such that
  \[
    Z_\bullet = \rd\bigl(X^1_\bullet,X^2_\bullet,\ldots,X^{r}_\bullet\bigr)
  \]
  where the following properties hold:
  \begin{enumerate}[label=\( (\roman*)\)]
    \item \(X^1_\bullet\) is an augmented 2-resolution of \(S\).
    \item For each \(1 \le n < r\), \(X_\bullet^{n+1}\) is an augmented
      2-resolution of \(X^n_{1\bullet}\).
    \item \(Z_\alpha\) is regular for every \(\alpha \in \square_r\).
  \end{enumerate}
  \par Now suppose \((S,S')\) is an \(I\)-pair of spaces.
  An \textsl{augmented cubical hyperresolution of \((S,S')\)} is an augmented
  cubical hyperresolution \(a\colon Z_\bullet \to S\) of \(S\) such that setting
  \(Z'_\bullet = a^{-1}(S')\), the \(I\)-pair \((Z_\bullet,Z'_\bullet)\) is regular.
\end{definition}
The following result is immediate from the definitions and Proposition
\ref{prop:gnapgp28}.
\begin{proposition}[cf.\ {\cite[Expos\'e I, Proposition
  2.14]{GNAPGP88}}]\label{prop:gnapgp214}
  Suppose that \(\Sp\) is a subcategory of a category of reduced spaces of one of the
  types
  \((\ref{setup:introalgebraicspaces})\textnormal{--}(\ref{setup:introadicspaces})\)
  that is essentially stable under fiber products, immersions, and proper
  morphisms.
  Let \(I\) be a small category.
  Let \(K\) be an \(I\)-object of \(\Cat\), let \(S\) be a
  \(\tot(K)\)-space, and let \(Z_\bullet\) be an \((\square_r^+ \times
  \tot(K))\)-space augmented over \(S\).
  Then, \(Z_\bullet\) is an augmented cubical hyperresolution of the
  \(\tot(K)\)-space \(S\) if and only if \(Z_i\)
  is an augmented cubical hyperresolution of the \(K_i\)-space \(S_i\)
  for every \(i \in \Ob(I)\).
\end{proposition}
We now prove the second main result of this section.
\begin{theorem}[Augmented cubical hyperresolutions exist;
  cf.\ {\citeleft\citen{GNAPGP88}\citemid Expos\'e I, Th\'eor\`eme
  2.15 and (3.11.2) and Expos\'e IV, Th\'eor\`eme 1.26\citepunct
  \citen{Ste17}\citemid Theorem 3\citeright}]
  \label{thm:gnapgp215}
  Suppose that \(\Sp\) is a subcategory of a category of reduced spaces of one of the
  types
  \((\ref{setup:introalgebraicspaces})\textnormal{--}(\ref{setup:introadicspaces})\)
  that is essentially stable under fiber products, immersions, and proper
  morphisms.
  Let \(I\) be a finite orderable category and let \((S,S')\) be an \(I\)-pair
  of spaces such that \(S\) is finite-dimensional and \(S - S'\) is dense in
  \(S\).
  In cases \((\ref{setup:introalgebraicspaces})\) and
  \((\ref{setup:introformalqschemes})\), suppose that
  \(S\) is quasi-excellent of equal characteristic zero.
  In cases \((\ref{setup:introberkovichspaces})\),
  \((\ref{setup:introrigidanalyticspaces})\), and
  \((\ref{setup:introadicspaces})\), suppose that \(k\) is of characteristic
  zero.
  Then, there exists an augmented cubical hyperresolution
  \[
    Z_\bullet = \rd\bigl(X^1_\bullet,X^2_\bullet,\ldots,X^r_\bullet\bigr)
  \]
  of the pair \((S,S')\) such that
  \[
    \dim(Z_\alpha) \le \dim(S) - \abs{\alpha} + 1
  \]
  for every \(\alpha \in \square_r\).
  Moreover:
  \begin{enumerate}[label=\((\roman*)\),ref=\roman*]
    \item\label{thm:gnapgp215factorsthrough}
      Let \(g\colon Y \to S\) be a proper morphism such that denoting by \(D\)
      the discriminant \(\Delta(g)\) of \(g\), the open subspaces \(S_i - D_i\) are
      dense in \(S_i\) for every \(i \in \Ob(I)\).
      We can construct \(Z_\bullet\) such that every non-identity morphism to
      \(Z_{0} = S\) in \(Z_\bullet\) factors through \(g\).
    \item\label{thm:gnapgp215locallyprojective}
      Suppose that
      \(g\) in \((\ref{thm:gnapgp215factorsthrough})\) is locally a
      birational/bimeromorphic blowup.
      We can construct \(Z_\bullet\) such that all morphisms in \(Z_{\bullet
      i}\) are locally birational/bimeromorphic blowups
      for every \(i \in \Ob(I)\), where locality is in terms of affinoid
      subdomains of \(S\).
    \item\label{thm:gnapgp215compact}
      Suppose that \(S\) is compactifiable and that
      \(g\) in \((\ref{thm:gnapgp215factorsthrough})\) is projective.
      We can construct \(Z_\bullet\)
      such that all morphisms in \(Z_{\bullet i}\) are birational/bimeromorphic
      blowups for every \(i \in \Ob(I)\).
  \end{enumerate}
\end{theorem}
\begin{proof}
  We proceed by induction on \(\dim(S)\).
  If \(\dim(S) = 0\), we set \(Z_\bullet = S\), which we think of as a
  \((\square_{-1}^+ \times I)\)-space.
  \par It remains to show the inductive case, when \(\dim(S) > 0\).
  By Theorem \ref{thm:gnapgp26} and Proposition \ref{prop:gnapgp210}, there
  exists an augmented 2-resolution \(X^1_{\bullet\bullet}\) of \(S\) such that
  \(\dim(X^1_{1\bullet}) < \dim(S)\) and the augmentation morphism factors
  through \(g\).
  Note that \(X^1_{\bullet\bullet}\) is a \((\square^+_1 \times I)\)-space and
  that \(X^1_{1\bullet}\) is a \((\square^+_0 \times I)\)-space.
  By the inductive hypothesis, there exists an augmented cubical hyperresolution
  \begin{align*}
    X'_\bullet &= \rd\bigl(X^2_\bullet,X^3_\bullet,\ldots,X^{r}_\bullet\bigr)
    \shortintertext{of \(X^1_{1\bullet}\) such that
    \[
      \dim\bigl(X'_{\alpha\beta}\bigr) \le \dim\bigl(X^1_{1\bullet}\bigr) -
      \abs{\alpha} + 1
    \]
    for every \((\alpha,\beta) \in \square_{r-1} \times \square_0^+\).
    Note that}
    \typ\bigl(X^i_\bullet\bigr) &=
    \square^+_{i-1} \times \square^+_0 \times I = \square^+_i \times I
    \shortintertext{for every \(i \in \{2,3,\ldots,r'\}\).
    Thus, we can define the reduction}
    X_\bullet &= \rd\bigl(X^1_\bullet,X^2_\bullet,\ldots,X^{r}_\bullet\bigr)
    \shortintertext{which is an augmented cubical hyperresolution of \(S\) such that}
    X_{\alpha\beta} &= \begin{cases}
      X^1_{0\beta} & \text{if}\ \alpha = 0\\
      X'_{\alpha\beta} & \text{if}\ \alpha \in \square_{r-1}
    \end{cases}
  \end{align*}
  for every \(\beta \in \square^+_0\).
  Moreover, \(X_\bullet\) satisfies the inequality
  \begin{align*}
    \dim\bigl(X_{0\beta}\bigr) &\le \dim(S) - \abs{\beta} + 1
    \shortintertext{for every \(\beta \in \square_0\) by the definition of a
    very weak resolution (Definition \ref{def:veryweakres}) and the inequality}
    \dim\bigl(X_{\alpha\beta}\bigr) \le \dim\bigl(X^1_{1\bullet}\bigr) &-
    \abs{\alpha} + 1 \le \dim(S) - \abs{\alpha} - \abs{\beta} + 1
  \end{align*}
  for every \((\alpha,\beta) \in \square_{r-1} \times \square_0^+\).
  The ``Moreover'' statements
  are true by construction because of what we showed in
  Theorem \ref{thm:gnapgp26} and by construction.
\end{proof}
As a consequence, we obtain the following.
\begin{corollary}[{cf.\ \citeleft\citen{Del74}\citemid \S6.2\citepunct
  \citen{Gui87}\citemid Corollaire 2.5.6\citeright}]
  Suppose that \(\Sp\) is a subcategory of a category of reduced spaces of one of the
  types
  \((\ref{setup:introalgebraicspaces})\textnormal{--}(\ref{setup:introadicspaces})\)
  that is essentially stable under fiber products, immersions, and proper
  morphisms.
  Let \((S,S')\) be a pair of spaces such that \(S\) is finite-dimensional and
  \(S - S'\) is dense in \(S\).
  In cases \((\ref{setup:introalgebraicspaces})\) and
  \((\ref{setup:introformalqschemes})\), suppose that
  \(S\) is quasi-excellent of equal characteristic zero.
  In cases \((\ref{setup:introberkovichspaces})\),
  \((\ref{setup:introrigidanalyticspaces})\), and
  \((\ref{setup:introadicspaces})\), suppose that \(k\) is of characteristic
  zero.
  Then, there exists a simplicial resolution
  \(X_\bullet \to S\) of the pair \((S,S')\).
  There also exists a semi-simplicial resolution of the pair \((S,S')\)
  such that
  \[
    \dim(X_n) \le \dim(S) - n
  \]
  for every \(n \ge 0\).
\end{corollary}
\begin{proof}
  The statement for semi-simplicial resolutions holds by applying the
  construction in \cite[(1.1.5) and (2.1.6)]{Gui87} to the cubical hyperresolution
  constructed in Theorem \ref{thm:gnapgp215}.
  We can then apply Guill\'en's version of the Dold--Puppe transform
  \cite[(1.2.1) and Proposition 1.3.7]{Gui87} to obtain a simplicial
  resolution.
\end{proof}

\section{The homotopy category of iterated cubical
hyperresolutions\texorpdfstring{\except{toc}{\\}}{} and cubical descent}
In this section, we discuss the homotopy category
\(\Ho\Hrc(\Diag_I(\Sp))\)
of iterated cubical hyperresolutions.
This is new for spaces of the form
\((\ref{setup:introalgebraicspaces})\textnormal{--}(\ref{setup:introadicspaces})\)
that are not varieties or complex analytic spaces.
We will show that \(\Ho\Hrc(\Diag_I(\Sp))\) is equivalent to
\(\Diag_I(\Sp)\) (Theorem \ref{thm:gnapgp38}),
which allows us to prove a version of descent for
cubical spaces (Corollary \ref{cor:gnapgp310}).
This gives a precise analogue of the fact that any two resolutions of
singularities can be dominated by another resolution.
As in \cite{GNAPGP88},
we will show that any two cubical hyperresolutions are connected
by a sequence of \emph{iterated} cubical hyperresolutions (see Lemma
\ref{lem:how-1connected} and Remark \ref{rem:how-1connected}).
Although in applications we compute objects and functors like
\(\underline{\Omega}_X^\bullet\) using cubical hyperresolutions, \emph{iterated}
cubical
hyperresolutions are unavoidable when showing these objects and functors are
well-defined.
See Definition \ref{def:gnapgp32} for the definition of an iterated cubical
hyperresolution.
\subsection{Conventions}
\par This section is where we start imposing smallness conditions on \(\Sp\).
Throughout this section, we denote by \(\Sp\) a 
subcategory of a small category of \emph{reduced} spaces of one of the types
\((\ref{setup:introalgebraicspaces})\textnormal{--}(\ref{setup:introadicspaces})\)
that is essentially stable under fiber products, immersions, and proper
morphisms.
We will assume \(\Sp\) is chosen according to the strategy outlined in
Remark \ref{rem:smallspaces}.
We denote by \(I\) a small category.
While we will remind the reader of these assumptions in statements of results in
this section, we will not restate these assumptions in definitions or remarks.
\par We will also continue to use the terminology \textsl{\(I\)-space} from
Definition \ref{def:ispaces} to refer to objects in \(\Diag_I(\Sp)\).
\begin{remark}\label{rem:iteratednotpairs}
  For simplicity, we do not work with pairs in the sense of Definition
  \ref{def:closedpairs} in this section.
  However, one can readily adapt the proofs to the context of pairs 
  using our results for pairs in the previous section as was pointed out
  in the context of varieties by Puerta \cite[Expos\'e IV, (1.27)]{GNAPGP88}.
  We will use this generality in the proof of the
  extension criterion for pairs (Theorem \ref{thm:cg39}).
\end{remark}
\subsection{Iterated and non-augmented cubical hyperresolutions}
We define iterated and non-augmented versions of augmented cubical
hyperresolutions.
These notions will be used to define a category of cubical hyperresolutions in
the next subsection.
The two definitions in this subsection do not require \(\Sp\) to be
small despite our conventions in the rest of this section.
\begin{definition}[Iterated augmented cubical hyperresolutions;
  cf.\ {\cite[Expos\'e I, D\'e\-fi\-ni\-tion 3.1]{GNAPGP88}}]
  Let \(S\) be an \(I\)-space.
  For every integer \(n \ge 1\), we define the notion of an \(n\)-iterated
  augmented cubical hyperresolution \(X_\bullet\) of \(S\) as follows.
  A \textsl{1-iterated augmented cubical hyperresolution \(X_\bullet\) of \(S\)}
  is an augmented cubical hyperresolution of \(S\).
  For each \(n \ge 2\), an \textsl{\(n\)-iterated augmented cubical
  hyperresolution \(X_\bullet\) of \(S\)} is a 1-iterated augmented cubical
  hyperresolution of a \((n-1)\)-iterated augmented cubical hyperresolution of
  \(S\).
  \par An \textsl{iterated augmented cubical hyperresolution of
  \(S\)} refers to an \(n\)-iterated augmented cubical hyperresolution of
  \(S\) for some integer \(n \ge 1\).
\end{definition}
\begin{remark}
  In \cite[Expos\'e I, D\'efinition 3.1]{GNAPGP88}, Guill\'en drops the
  adjective ``iterated'' in ``iterated augmented cubical hyperresolution.''
  This new terminology is used in the rest of \cite{GNAPGP88}.
  We keep the adjective ``iterated'' since these iterations are
  important for the proofs of the results in \cite{GNAPGP88,GNA02} and in
  the rest of this paper.
\end{remark}
\begin{definition}[Cubical hyperresolutions and iterated cubical
    hyperresolutions
  {\cite[Expos\'e I, D\'efinition 3.2]{GNAPGP88}}]
  \label{def:gnapgp32}
  Let \(X^+_\bullet\) be a \((\square_r^+ \times I)\)-space and let
  \(\pi\colon \square_r \times I \rightarrow I\)
  be the projection functor.
  If \(X^+_\bullet\) is an augmented cubical hyperresolution (resp.\ an iterated
  augmented cubical hyperresolution, an \(n\)-iterated augmented cubical
  hyperresolution for an integer \(n \ge 1\)) of an
  \(I\)-space \(S\), the restriction of \(X^+_\bullet\) to \(\square_r \times
  I\) defines a \((\square_r \times I)\)-space \(X_\bullet\) together with a
  \(\pi\)-augmentation \(a \colon X_\bullet \to S\) such that \(X^+_\bullet =
  \tot(a)\).
  In this situation, we will call \(a\) or \(X_\bullet\) a \textsl{cubical
  hyperresolution of \(S\)} (resp.\ \textsl{iterated cubical
  hyperresolution of \(S\)}, \textsl{\(n\)-iterated cubical
  hyperresolution of \(S\)}).
\end{definition}
\subsection{The homotopy category
\texorpdfstring{\(\Ho\Hrc(\Diag_I(\Sp))\)}{Ho Hrc(DiagI(Sp))}
of iterated cubical hyperresolutions}
We now define the category of iterated cubical hyperresolutions.
Starting here, we need to impose the condition that \(\Sp\) is small.
\begin{definition}[The category of iterated cubical hyperresolutions
  {\cite[Expos\'e I, D\'e\-fi\-ni\-tion 3.3]{GNAPGP88}}]\label{def:gnapgp33}
  Let \(u \colon S \to T\) be a morphism of \(I\)-spaces.
  Suppose \(a\colon X_\bullet \to S\) and \(b\colon Y_\bullet \to T\) are
  iterated cubical hyperresolutions of \(S\) and \(T\), respectively.
  A \textsl{morphism of iterated cubical hyperresolutions from \(X_\bullet\) to
  \(Y_\bullet\) over \(u\)} is a morphism
  \(f_\bullet\colon X_\bullet \to Y_\bullet\)
  of 1-diagrams of \(\Sp\) augmented by \(u\)
  such that
  \(\typ(f_\bullet)\colon I \times \square_r \to I \times \square_s\)
  is \(\id_I \times \delta\) for a face
  functor \(\delta\) (see Definition \ref{def:gnapgp117}).
  By Definition \ref{def:gnapgp13}, \(f_\bullet\) is a
  morphism of augmented 1-diagrams of \(\Sp\)
  if and only if the
  diagram
  \[
    \begin{tikzcd}
      X_\bullet \rar{f_\bullet}\dar[swap]{a} & Y_\bullet\dar{b}\\
      S \rar{u} & T
    \end{tikzcd}
  \]
  of morphisms of 1-diagrams of spaces commutes.
  \par We denote by \(\Hrc(\Diag_I(\Sp))\) the category of iterated cubical
  hyperresolutions of \(I\)-spaces with the notion of morphisms defined above.
  We obtain a functor
  \[
    w_I\colon \Hrc\bigl(\Diag_I(\Sp)\bigr) \longrightarrow \Diag_I(\Sp)
  \]
  defined by \(w_I(X_\bullet) = S\) if \(X_\bullet\) is an iterated
  cubical hyperresolution
  of \(S\) and \(w_I(f_\bullet) = u\) if \(f_\bullet\) is a morphism of iterated
  cubical hyperresolutions over \(u\).
  We denote by \(\Hrc_\locproj(\Diag_I(\Sp))\) (resp.\
  \(\Hrc_\proj(\Diag_I(\Sp))\)) the full subcategories of
  \(\Hrc(\Diag_I(\Sp))\) whose objects are the iterated cubical hyperresolutions
  \(X_\bullet\) of \(S\) such that all morphisms in \(X_{\bullet i}\) are
  locally birational/bimeromorphic blowups
  (resp.\ birational/bimeromorphic blowups)
  and such that \(X_{\bullet i} \to S_i\) is locally a birational/bimeromorphic blowup
  (resp.\ a birational/bimeromorphic blowup) for every \(i \in \Ob(I)\).
\end{definition}
To define the homotopy category of iterated
cubical hyperresolutions, we use the notion
of a fiber category recalled in \cite[Expos\'e I, (3.4)]{GNAPGP88}
(see also \citeleft\citen{stacks-project}\citemid
\href{https://stacks.math.columbia.edu/tag/02XJ}{Tag 02XJ}\citepunct
\citen{Joh021}\citemid \S B1.3\citeright).
We start by defining a set of morphisms in \(\Hrc(\Diag_I(\Sp))\) that we
will later invert.
\begin{citeddef}[{\cite[(3.5) and (3.6)]{GNAPGP88}}]
  Let \(S\) be an \(I\)-space.
  We denote by \(\Hrc(S)\) the fiber category \(w_I^{-1}(S)\).
  If \(u \colon S \to T\) is a morphism of \(I\)-spaces, we denote by
  \(\Hrc(u)\) the category \(w_I^{-1}(u)\).
  We denote by \(\Sigma_{\Diag_I(\Sp)}\) the set of morphisms \(f\) of
  \(\Hrc(\Diag_I(\Sp))\) such that \(w_I(f)\) is an identity morphism in
  \(\Diag_I(\Sp)\).
  Similarly, if \(S\) is an \(I\)-space, we denote by \(\Sigma_S\) the set
  of all morphisms of \(\Hrc(S)\).
  We use the same notation for \(\Hrc_\locproj\) and \(\Hrc_\proj\).
\end{citeddef}
We now localize the category of iterated cubical hyperresolutions with
respect to the sets of morphisms defined above using categories of fractions
\cite[\S5.2]{Bor94}.
Localizing the categories \(\Hrc(\Diag_I(\Sp))\) and \(\Hrc(S)\)
with respect to the sets \(\Sigma_{\Diag_I(\Sp)}\) and \(\Sigma_S\),
respectively, we obtain the localization functors
\[
  \begin{tikzcd}[cramped,row sep=0,column sep=1.475em]
    \Hrc\bigl(\Diag_I(\Sp)\bigr) \rar &
    \Ho\Hrc\bigl(\Diag_I(\Sp)\bigr)\\
    \Hrc(S) \rar & \Ho\Hrc(S)
  \end{tikzcd}
\]
We use the same notation for \(\Hrc_\locproj\) and \(\Hrc_\proj\).
By the universal property of categories of fractions
\citeleft\citen{Bor94}\citemid
Proposition 5.2.2\citeright,
we can define the following functors on
\(\Ho\Hrc(\Diag_I(\Sp))\).
\begin{definition}[cf.\ {\cite[(3.6) and (3.7)]{GNAPGP88}}]
  Suppose that \(\Sp\) is a subcategory of a small category of reduced spaces of one of the
  types
  \((\ref{setup:introalgebraicspaces})\textnormal{--}(\ref{setup:introadicspaces})\)
  that is essentially stable under fiber products, immersions, and proper
  morphisms.
  In cases \((\ref{setup:introalgebraicspaces})\) and
  \((\ref{setup:introformalqschemes})\), suppose that all objects in \(\Sp\) are
  quasi-excellent of equal characteristic zero.
  In cases \((\ref{setup:introberkovichspaces})\),
  \((\ref{setup:introrigidanalyticspaces})\), and
  \((\ref{setup:introadicspaces})\), suppose that
  \(k\) is of characteristic zero.
  \par Let \(I\) be a finite orderable small
  category and let \(S\) be an \(I\)-space.
  By the universal property of categories of fractions
  \citeleft\citen{Bor94}\citemid
  Proposition 5.2.2\citeright,
  we have the commutative diagram of functors
  \[
    \begin{tikzcd}[column sep=huge]
      \Ho\Hrc(S) \rar{\Ho i_S} & \Ho\Hrc\bigl(\Diag_I(\Sp)\bigr) \rar{\Ho w_I} &
      \Diag_I(\Sp)\\
      \Ho\Hrc_\locproj(S) \rar{\Ho i_{S,\locproj}} \uar[hook]
      & \Ho\Hrc_\locproj\bigl(\Diag_I(\Sp)\bigr) \rar{\Ho w_{I,\locproj}}
      \uar[hook] &
      \Diag_I(\Sp) \uar[equal]\\
      \Ho\Hrc_\proj(S) \rar{\Ho i_{S,\proj}} \uar[hook]
      & \Ho\Hrc_\proj\bigl(\Diag_I(\Sp_{\comp})\bigr)
      \rar{\Ho w_{I,\proj}} \uar[hook] &
      \Diag_I(\Sp_{\comp})\mathrlap{.} \uar[hook]
    \end{tikzcd}
  \]
  In the third line, we assume that \(S\) is an object of \(\Sp_{\comp}\),
  the subcategory whose objects are compactifiable.
  Moreover,
  \(i_S\colon \Hrc(S) \to \Hrc(\Diag_I(\Sp))\)
  is the inclusion functor, and we use analogous notation for
  \(\Hrc_\locproj\) (resp.\ \(\Hrc_\proj\)).
\end{definition}
\subsection{\texorpdfstring{\(\Ho\Hrc(\Diag_I(\Sp))\)}{Ho Hrc(DiagI(Sp))} and
\texorpdfstring{\(\Diag_I(\Sp)\)}{DiagI(Sp)} are equivalent}
We can now state the main result of this section.
\begin{theorem}[cf.\ {\cite[Th\'eor\`eme 3.8, Remarque 3.8.4, and (3.11.2)]{GNAPGP88}}]
  \label{thm:gnapgp38}
  Suppose that \(\Sp\) is a subcategory of a small category of
  finite-dimensional reduced spaces of one of the
  types
  \((\ref{setup:introalgebraicspaces})\textnormal{--}(\ref{setup:introadicspaces})\)
  that is essentially stable under fiber products, immersions, and proper
  morphisms.
  In cases \((\ref{setup:introalgebraicspaces})\) and
  \((\ref{setup:introformalqschemes})\), suppose that all objects in \(\Sp\) are
  quasi-excellent of equal characteristic zero.
  In cases \((\ref{setup:introberkovichspaces})\),
  \((\ref{setup:introrigidanalyticspaces})\), and
  \((\ref{setup:introadicspaces})\), suppose that
  \(k\) is of characteristic zero.
  \par Let \(I\) be a finite orderable category and let \(S\) be an
  \(I\)-space.
  Consider the functors
  \[
    \begin{tikzcd}[cramped,row sep=0,column sep=scriptsize,
      /tikz/column 1/.append style={anchor=base east}]
      \Ho w_I\colon&[-2.375em] \Ho\Hrc\bigl(\Diag_I(\Sp)\bigr) \rar &
      \Diag_I(\Sp)\\
      \Ho w_{I,\locproj}
      \colon& \Ho\Hrc_\locproj\bigl(\Diag_I(\Sp)\bigr) \rar &
      \Diag_I(\Sp)\\
      \Ho w_{I,\proj}
      \colon& \Ho\Hrc_\proj\bigl(\Diag_I(\Sp_{\comp})\bigr) \rar &
      \Diag_I(\Sp_{\comp})\\
      \Ho i_S\colon& \Ho\Hrc(S) \rar &
      (\Ho w_I)^{-1}(S)\\
      \Ho i_{S,\locproj}
      \colon& \Ho\Hrc_\locproj(S) \rar &
      (\Ho w_{I,\locproj})^{-1}(S)\\
      \Ho i_{S,\proj}\colon& \Ho\Hrc_\proj(S) \rar &
      (\Ho w_{I,\proj})^{-1}(S)
    \end{tikzcd}
  \]
  where \(S\) is compactifiable in the last line.
  We have the following:
  \begin{enumerate}[label=\((\roman*)\),ref=\roman*]
    \item The functors \(\Ho w_I\) and \(\Ho i_S\) are equivalences of
      categories.
      The source and target of \(\Ho i_S\) are equivalent to the punctual
      category \(\underline{1}\).
    \item The functors \(\Ho w_{I,\locproj}\) and \(\Ho
      i_{S,\locproj}\) are equivalences of categories.
      The source and target of \(\Ho
      i_{S,\locproj}\) are equivalent to the punctual
      category \(\underline{1}\).
    \item\label{thm:gnapgp38compact}
      Suppose \(S\) is compactifiable.
      The functors \(\Ho w_{I,\proj}\) and \(\Ho i_{S,\proj}\) are equivalences of
      categories.
      The source and target
      of \(\Ho i_{S,\proj}\) are equivalent to the punctual
      category \(\underline{1}\).
  \end{enumerate}
\end{theorem}
Given the results in this paper so far, the
proof of Theorem \ref{thm:gnapgp38} is almost identical to that in
\cite{GNAPGP88}.
We state some of the preliminary results from \cite{GNAPGP88}
explicitly.
\begin{lemma}[cf.\ {\cite[Expos\'e I, Lemme 3.8.3, Remarque 3.8.4, and Lemme
  3.8.5]{GNAPGP88}}]\label{lem:how-1connected}
  Suppose that \(\Sp\) is a subcategory of a small category of
  finite-dimensional reduced spaces of one of the
  types
  \((\ref{setup:introalgebraicspaces})\textnormal{--}(\ref{setup:introadicspaces})\)
  that is essentially stable under fiber products, immersions, and proper
  morphisms.
  In cases \((\ref{setup:introalgebraicspaces})\) and 
  \((\ref{setup:introformalqschemes})\), suppose that all objects in \(\Sp\) are
  quasi-excellent of equal characteristic zero.
  In cases \((\ref{setup:introberkovichspaces})\),
  \((\ref{setup:introrigidanalyticspaces})\), and
  \((\ref{setup:introadicspaces})\), suppose that
  \(k\) is of characteristic zero.
  \par Let \(I\) be a finite orderable category and
  let \(u\colon S \to T\) be a morphism of \(I\)-spaces.
  We have the following:
  \begin{enumerate}[label=\((\roman*)\),ref=\roman*]
    \item\label{lem:fiberoversconn}
      The categories \((\Ho w_I)^{-1}(S)\), \((\Ho w_{I,\locproj})^{-1}(S)\),
      \(\Ho\Hrc(S)\), and \(\Ho\Hrc_\locproj(S)\)
      are groupoids that are connected and simply connected.
      If \(S\) is compactifiable, the categories \((\Ho w_{I,\proj})^{-1}(S)\) and
      \(\Ho\Hrc_\proj(S)\) are groupoids that are connected and simply connected.
    \item The fiber categories \((\Ho w_I)^{-1}(u)\) and \((\Ho
      w_{I,\locproj})^{-1}(u)\) are groupoids that are connected and nonempty.
      If \(S\) and \(T\) are compactifiable, then \((\Ho w_{I,\proj})^{-1}(u)\)
      is a groupoid that is connected and nonempty.
  \end{enumerate}
\end{lemma}
\begin{proof}[Proof of Theorem \ref{thm:gnapgp38} and
  Lemma \ref{lem:how-1connected}]
  The proofs of \cite[Expos\'e I, Lemme 3.8.1, Lemme 3.8.2, Lemme 3.8.3, Lemme 3.8.5, Lemme
  3.8.6, and Th\'eor\`eme 3.8]{GNAPGP88} apply using our existence result for
  augmented cubical hyperresolutions (Theorem \ref{thm:gnapgp215}).
\end{proof}
\begin{remark}\label{rem:how-1connected}
  In the proof of Lemma \ref{lem:how-1connected}\((\ref{lem:fiberoversconn})\)
  in \cite[Expos\'e I, Lemme 3.8.3]{GNAPGP88}, Guill\'en
  shows the following:
  With assumptions as in Lemma \ref{lem:how-1connected}, if \(S\) is an
  \(I\)-space and if \(X\) and \(Y\) are two iterated cubical hyperresolutions
  of \(S\), then there is a diagram
  \[
    \begin{tikzcd}[column sep=small]
      X \arrow{dr} & & S'\arrow{dl}\arrow{dr} & &
      Y\arrow{dl}\\
      & X^t & & Y^t
    \end{tikzcd}
  \]
  of morphisms of iterated cubical hyperresolutions of \(S\), where every morphism
  in this diagram maps to the identity on \(S\) under \(w\).
\end{remark}
Using the Axiom of Choice, we can specify an inverse functor for \(\Ho w\) by
specifying its behavior on objects.
\begin{corollary}[cf.\ {\cite[Expos\'e I, (3.9)]{GNAPGP88}}]\label{cor:gnapgp39}
  Suppose that \(\Sp\) is a subcategory of a small category of
  finite-dimensional reduced spaces of one of the
  types
  \((\ref{setup:introalgebraicspaces})\textnormal{--}(\ref{setup:introadicspaces})\).
  Suppose that \(\Sp\) is a subcategory of a category of finite-dimensional
  reduced spaces of one of the
  types
  \((\ref{setup:introalgebraicspaces})\textnormal{--}(\ref{setup:introadicspaces})\)
  that is essentially stable under fiber products, immersions, and proper
  morphisms.
  In cases \((\ref{setup:introalgebraicspaces})\) and
  \((\ref{setup:introformalqschemes})\), suppose that all objects in \(\Sp\) are
  quasi-excellent of equal characteristic zero.
  In cases \((\ref{setup:introberkovichspaces})\),
  \((\ref{setup:introrigidanalyticspaces})\), and
  \((\ref{setup:introadicspaces})\), suppose that
  \(k\) is of characteristic zero.
  \par Let \(I\) be a finite orderable category.
  Consider a section
  \begin{align*}
    \Ob(\eta) \colon{}& \Ob\bigl(\Diag_I(\Sp)\bigr) \longrightarrow
    \Ob\bigl(\Hrc\bigl(\Diag_I(\Sp)\bigr)\bigr)
    \shortintertext{of the map \(\Ob(w_I)\).
    Then, there exists a unique extension of \(\Ob(\eta)\) to a functor}
    \eta\colon{}& \Diag_I(\Sp) \longrightarrow \Ho\Hrc\bigl(\Diag_I(\Sp)\bigr)
  \end{align*}
  that is quasi-inverse to \(\Ho(w_I)\).
  Moreover, if \(\Ob(\eta)\) and \(\Ob(\eta')\) are two sections of \(\Ob(w_I)\),
  there exists a unique natural equivalence of functors \(\Theta \colon \eta'
  \Leftrightarrow \eta\).
  In particular, by Theorem \ref{thm:gnapgp215} and the Axiom of Choice,
  there exists a quasi-inverse \(\eta\) of \(\Ho w_I\) such that
  \begin{enumerate}[label=\((\roman*)\),ref=\roman*]
    \item\label{cor:gnapgp39i} \(\eta(S) = S\) if \(S\) is a regular \(I\)-space.
    \item\label{cor:gnapgp39ii}
      \(\dim(\eta(S)_\alpha) \le \dim(S) - \abs{\alpha} + 1\) for every
      \(\alpha \in \square_r\).
  \end{enumerate}
  \par The analogous statement holds for \(\Hrc_\locproj\).
  If we replace \(\Sp\) with \(\Sp_\comp\), then the analogous statement holds
  for \(\Hrc_\proj\).
\end{corollary}
\begin{proof}
  This result follows from Theorem \ref{thm:gnapgp38} and
  \cite[Proposition 3.4.3]{Bor94}.
\end{proof}
\subsection{Cubical descent for functors}
We also obtain the following descent result as an immediate consequence of
Theorem \ref{thm:gnapgp38}.
\begin{corollary}[Cubical descent for functors;
  {cf.\ 
  \citeleft\citen{GNAPGP88}\citemid Expos\'e I, Corollaire
  3.10\citeright}]\label{cor:gnapgp310}
  Suppose that \(\Sp\) is a subcategory of a small category of
  finite-dimensional reduced spaces of one of the
  types
  \((\ref{setup:introalgebraicspaces})\textnormal{--}(\ref{setup:introadicspaces})\)
  that is essentially stable under fiber products, immersions, and proper morphisms.
  In cases \((\ref{setup:introalgebraicspaces})\) and
  \((\ref{setup:introformalqschemes})\), suppose that all objects in \(\Sp\) are
  quasi-excellent of equal characteristic zero.
  In cases \((\ref{setup:introberkovichspaces})\),
  \((\ref{setup:introrigidanalyticspaces})\), and
  \((\ref{setup:introadicspaces})\), suppose that
  \(k\) is of characteristic zero.
  \par Let \(I\) be a finite orderable category and
  let \(\sC\) be a category.
  We denote by
  \[
    w_I^*\colon \Hom\bigl(\Diag_I(\Sp),\sC\bigr) \longrightarrow
    \Hom\bigl(\Hrc\bigl(\Diag_I(\Sp)\bigr),\sC\bigr)
  \]
  the functor defined by \(w_I^*(F) = F \circ w_I\).
  Then, \(w_I^*\) induces an equivalence of categories between
  \(\Hom(\Diag_I(\Sp),\sC)\) and the full subcategory of
  \(\Hom(\Hrc(\Diag_I(\Sp)),\sC)\) consisting of functors
  \[
    G\colon \Hrc\bigl(\Diag_I(\Sp)\bigr) \longrightarrow \sC
  \]
  satisfying the following condition:
  \begin{enumerate}[label={\((\mathrm{DC})\)},ref=\mathrm{DC}]
    \item\label{cond:descentcubic} For every morphism \(f \in
      \Sigma_{\Diag_I(\Sp)}\), its image \(G(f)\) is an isomorphism in \(\sC\).
  \end{enumerate}
  In particular, if \(\eta\) is a quasi-inverse of \(\Ho w\), for every
  functor \(G\colon \Hrc(\Diag_I(\Sp)) \to \sC\) satisfying 
  \((\ref{cond:descentcubic})\), there exists a natural equivalence of functors
  \(\tau_\eta\colon w_I^*(G_\eta) \Leftrightarrow G\),
  where
  \[
    G_\eta \coloneqq (\Ho G) \circ \eta \colon \Diag_I(\Sp) \longrightarrow \sC.
  \]
  If \(\eta'\) is another quasi-inverse of \(\Ho w_I\), there exists an
  equivalence of functors
  \(G_\Theta\colon G_{\eta'} \Leftrightarrow G_\eta\)
  such that \(\tau_{\eta'} = \tau_\eta \circ (G_\Theta \star w_I)\), where
  \(G_\Theta\) is defined by \(G_\Theta = (\Ho G) \star \Theta\).
  \par The analogous statement holds for \(\Hrc_\locproj\).
  If we replace \(\Sp\) with \(\Sp_\comp\), then the analogous statement holds
  for \(\Hrc_\proj\).
\end{corollary}
\begin{proof}
  Use
  the universal property of categories of fractions
  \citeleft\citen{Bor94}\citemid
  Proposition 5.2.2\citeright.
\end{proof}
We therefore obtain the following diagram of functors:
\[
  \begin{tikzcd}[row sep=large]
    \Diag_I\bigl(\Sp_\reg\bigr)\arrow{dd}[swap]{i}\arrow{dr}{j}
    \arrow[bend left=15]{drr}{G \circ j} &[-1em] &[1.5em]\\
    & \Hrc\bigl(\Diag_I(\Sp)\bigr)
    \arrow[r,"G","\displaystyle\Updownarrow"'{pos=0.33,yshift=-2pt}]
    \arrow{dl}[swap]{w_I} & \sC\mathrlap{.}\\
    \Diag_I(\Sp)
    \arrow[bend right=15]{urr}[swap]{G_\eta}
  \end{tikzcd}
\]
Here,
\(i\) and \(j\) are the inclusion functors and \(G\) is
a functor satisfying \((\ref{cond:descentcubic})\).
The left triangle is commutative and the bottom right triangle commutes up to
the equivalence \(\tau_\eta\) that we denote by \(\Updownarrow\).
The outermost triangle is commutative
as long as \(\eta\) satisfies condition \((\ref{cor:gnapgp39i})\) from Corollary
\ref{cor:gnapgp39}.
\subsection{Cubical descent for cohomology}
Finally, we discuss (lisse-)\'etale and pro-\'etale
sheaves on \(I\)-spaces and their cohomology.
In this subsection, we do not need to assume the category \(\Sp\) is
small.\medskip
\par
We will mostly work with the \'etale topology \(X_\et\) on spaces \(X\) (see
\S\ref{sect:etaletop}).
In case \((\ref{setup:introalgebraicspaces})\)
where we use the lisse-\'etale topology \(X_{\lisse\mhyphen\et}\).
As mentioned in \S\ref{sect:etaletop},
in the complex analytic case \((\ref{setup:introcomplexanalyticgerms})\), we
will call the Euclidean topology the \'etale topology to treat all cases
\((\ref{setup:introalgebraicspaces})\textnormal{--}(\ref{setup:introadicspaces})\)
at once.
\par We will also work with the pro-\'etale topology when it is defined.
\begin{definition}[Sheaves on 1-diagrams of spaces
  {\cite[Expos\'e I, (5.3) and D\'efinition 5.4]{GNAPGP88}}]
  Let \(X_\bullet\) be an \(I\)-space and let \(\sC\) be a category.
  A \textsl{(lisse-)\'etale presheaf on \(X_\bullet\) with values in \(\sC\)} is a 2-diagram
  \begin{align*}
    \sF^\bullet\colon{}& I \longrightarrow \Diag_1(\sC)
    \shortintertext{of \(\sC\) such that \(\typ_1(\sF^\bullet) =
    X_{\et\bullet}\) (\(X_{\lisse\mhyphen\et\bullet}\) in case
    \((\ref{setup:introalgebraicspaces})\)).
    Equivalently (see Definition \ref{def:gnapgp16}),
    \(\sF^\bullet\) is a \(\tot(X_{\et\bullet})\)-object (a
    \(\tot(X_{\lisse\mhyphen\et\bullet})\)-object in case
    \((\ref{setup:introalgebraicspaces})\)) of \(\sC\):}
    \tot(\sF^\bullet)\colon{}& \tot(X_{\et\bullet})^\op \longrightarrow \sC,\\
    \tot(\sF^\bullet)\colon{}& \tot(X_{\lisse\mhyphen\et\bullet})^\op \longrightarrow \sC.
  \end{align*}
  The data of \(\sF^\bullet\) is equivalent to specifying the following data:
  \begin{enumerate}[label=\((\roman*)\),ref=\roman*]
    \item A (lisse-)\'etale presheaf \(\sF^i\) on \(X_i\) with values in \(\sC\) for every \(i
      \in \Ob(I)\), and
    \item An \(X_u\)-morphism of (lisse-)\'etale presheaves \(\sF^u\colon \sF^i \to \sF^j\),
      i.e., a morphism \(X_u^{-1}\sF^i \to \sF^j\), for
      every morphism \(u\colon i \to j\) of \(I\).
  \end{enumerate}
  We say that a (lisse-)\'etale presheaf \(\sF^\bullet\) is a \textsl{(lisse-)\'etale
  sheaf on \(X_\bullet\)}
  if for every \(i \in \Ob(I)\), the (lisse-)\'etale presheaf \(\sF^i\) is a
  (lisse-)\'etale
  sheaf on \(X_i\).
  We denote by \(\Sh(X_{\et\bullet},\sC)\)
  (\(\Sh(X_{\lisse\mhyphen\et\bullet},\sC)\) in case
  \((\ref{setup:introalgebraicspaces})\)) the category of (lisse-)\'etale
  sheaves on \(X_\bullet\) with values in \(\sC\).
  \par We define pro-\'etale (pre)sheaves in the same way in cases
  \((\ref{setup:introalgebraicspaces})\) and \((\ref{setup:introadicspaces})\).
\end{definition}
We can compare this definition to the definition of the total topos associated
to a fibered site from \cite[Expos\'e VI, \S7]{SGA42} and use this description
to discuss derived categories of sheaves on \(X_\bullet\).
\begin{remark}[Total topoi; cf.\ {\cite[Expos\'e I, (5.8) and Remarque 5.18]{GNAPGP88}}]
  By the Grothendieck construction
  \citeleft\citen{Ill72}\citemid Chapitre VI, Exemple
  5.6\((b)\)\citepunct \citen{Joh021}\citemid Theorem B1.3.5\citeright,
  we can think of
  \(\Sh(X_{\et\bullet},\Sets)\) (\(\Sh(X_{\lisse\mhyphen\et\bullet},\sC)\) in case
  \((\ref{setup:introalgebraicspaces})\)) as the total topos associated to a fibered site as
  follows.
  First, by taking associated sites for each \(i \in \Ob(I)\), the
  (lisse-)\'etale sites on each \(X_i\)
  define a fibered site over \(I^\op\) in the
  sense of \cite[Expos\'e VI, D\'efinition 7.1.1]{SGA42}.
  The associated total site \cite[Expos\'e VI, D\'efinition 7.4.1]{SGA42} has an
  associated topos of sheaves, which is called the \textsl{total topos}
  \(\Top(X_{\et\bullet})\) (\(\Top(X_{\lisse\mhyphen\et\bullet})\) in case
  \((\ref{setup:introalgebraicspaces})\)) \cite[Expos\'e VI, Remarque 7.4.3(3)]{SGA42}.
  This topos is equivalent to the category
  \(\Sh(X_{\et\bullet},\Sets)\) (\(\Sh(X_{\lisse\mhyphen\et\bullet},\Sets)\) in case
  \((\ref{setup:introalgebraicspaces})\)) by \citeleft\citen{Ill72}\citemid
  Chapitre VI, Exemple 5.6\((b)\)\citeright.
  \par Now let \(A\) be a torsion ring whose characteristic is a unit on every
  \(X_i\).
  Since \(\Top(X_{\et\bullet})\) (\(\Top(X_{\lisse\mhyphen\et\bullet})\) in case
  \((\ref{setup:introalgebraicspaces})\)) arises as the topos of sheaves on a site, the
  subcategory \(\Sh(X_{\et\bullet},A\mhyphen\Mod)\)
  (\(\Sh(X_{\lisse\mhyphen\et\bullet},A\mhyphen\Mod)\) in case
  \((\ref{setup:introalgebraicspaces})\)) of sheaves of \(A\)-modules is
  Grothendieck Abelian by \cite[\href{https://stacks.math.columbia.edu/tag/07A5}{Tag
  07A5}]{stacks-project}.
  Even though this category is not small, we can still consider the associated
  derived categories
  \(\D^{\natural}(X_{\et\bullet},A\mhyphen\Mod)\)
  (\(\D^{\natural}(X_{\lisse\mhyphen\et\bullet},A\mhyphen\Mod)\) in case
  \((\ref{setup:introalgebraicspaces})\))
  for \(\natural \in \{*,+,-,b\}\) by
  \cite[\href{https://stacks.math.columbia.edu/tag/05RU}{Tag
  05RU} and \href{https://stacks.math.columbia.edu/tag/09PA}{Tag
  09PA}]{stacks-project}.
  We can then define right derived functors for additive functors out of
  \(\Sh(X_{\et\bullet},A\mhyphen\Mod)\) (\(\Sh(X_{\lisse\mhyphen\et\bullet},\Sets)\) in case
  \((\ref{setup:introalgebraicspaces})\))
  using \(K\)-injective
  resolutions 
  \citeleft\citen{stacks-project}\citemid
  \href{https://stacks.math.columbia.edu/tag/07A5}{Tag 07A5}\citeright.
  \par The same discussion applies to the pro-\'etale site \(X_\proet\) in cases
  \((\ref{setup:introalgebraicspaces})\) and \((\ref{setup:introadicspaces})\).
\end{remark}
We will use this translation between the language
of 1-diagrams and of total topoi to utilize results from
\cite[Chapitre VI]{Ill72} instead of reproving them in the context of
1-diagrams.\medskip
\par Inverse images and direct images for morphisms of 1-diagrams of
spaces are defined as follows.
\begin{definition}[Inverse images and direct images of sheaves on 1-diagrams of
  spaces {\citeleft\citen{Ill72}\citemid Chapitre VI, (5.5) and (5.8)\citepunct
  \citen{GNAPGP88}\citemid Expos\'e I, (5.5)\citeright}]
  Let \(X_\bullet\) and \(Y_\bullet\) be 1-diagrams of spaces of types \(I\) and
  \(J\), respectively.
  Let \(\varphi\colon I \to J\) be a functor and let \(f_\bullet\colon X_\bullet
  \to Y_\bullet\) be a \(\varphi\)-morphism of 1-diagrams of spaces.
  If \(\sG^\bullet\) is a sheaf on \(Y_\bullet\) with values in a cocomplete
  category \(\sC\), we denote by \(f_\bullet^{-1}\sG^\bullet\) the sheaf on
  \(X_\bullet\) defined by
  \[
    \bigl(f_\bullet^{-1}\sG^\bullet\bigr)^i \coloneqq
    f_i^{-1}\bigl(\sG^{\varphi(i)}\bigr)
  \]
  for every \(i \in \Ob(I)\).
  This definition yields a functor
  \begin{alignat*}{3}
    f_\bullet^{-1}&{}\colon{}& \Sh(Y_\bullet,\sC) &\longrightarrow
    \Sh(X_\bullet,\sC).
    \shortintertext{If \(\sC\) is also complete, then \(f_\bullet^{-1}\) admits a
    right adjoint}
    f_{\bullet*}&{}\colon{}& \Sh(X_\bullet,\sC) &\longrightarrow \Sh(Y_\bullet,\sC)
  \end{alignat*}
  (see \cite[Expos\'e I, Proposition 5.1]{SGA41}), defined as follows.
  If \(\sF^\bullet\) is a sheaf on \(X_\bullet\) with values in \(\sC\), then
  \(f_{\bullet*}\sF^\bullet\) is the sheaf on \(Y_\bullet\) defined by
  \[
    \bigl(f_{\bullet*}\sF^\bullet\bigr)^j \coloneqq
    \varprojlim_{(i,u) \in \Ob(j \backslash \varphi)} (Y_{u})_*\bigl(f_{i*}\sF^i\bigr)
  \]
  for every \(j \in \Ob(J)\).
\end{definition}
We can compute derived direct images component-wise as follows.
\begin{proposition}[cf.\ {\citeleft\citen{Ill72}\citemid Chapitre VI, (5.8) and
  Proposition 6.3.1\citepunct \citen{GNAPGP88}\citemid Expos\'e I, Proposition
  5.10, Proposition 5.11, (5.12), Corollaire 5.1, and Corollaire 5.14\citeright}]
  \label{prop:componentwise}
  Suppose that \(\Sp\) is a subcategory of a category of spaces of one of the types
  \((\ref{setup:introalgebraicspaces})\textnormal{--}(\ref{setup:introadicspaces})\)
  that is essentially stable under fiber products, immersions, and proper
  morphisms.
  Let \(\varphi\colon I \to J\) be a functor between small categories.
  Let \(X\) and \(Y\) be an \(I\)-space and a \(J\)-space, respectively, and
  consider a \(\varphi\)-morphism \(f_\bullet\colon X_{\bullet} \to
  Y_\bullet\).
  Consider the factorization
  \[
    f_\bullet\colon X_{\bullet} \xrightarrow{f_{\bullet\bullet}} Y_\bullet
    \times_J I \overset{\varphi_\bullet}{\longrightarrow} Y_\bullet.
  \]
  Consider an object \(\sF^\bullet\) of \(\D^*(X_{\bullet\et},A\mhyphen\Mod)\)
  and an object
  \(\sG^\bullet\) of \(\D^*(Y_{\bullet\et} \times_J I,A\mhyphen\Mod)\).
  For every \(i \in \Ob(I)\) and every \(j \in \Ob(J)\), we have
  \begin{align*}
    \RR f_{\bullet*} &= \RR\varphi_{\bullet*} \circ \RR
    f_{\bullet\bullet*},\\
    \bigl(\RR f_{\bullet*}\sF^\bullet\bigr)^j &= \RRvarprojlim_{(i,u) \in \Ob(j
    \backslash \varphi)} \bigl(\RR f_{i\bullet} \sF^i\bigr),\\
    \bigl(\RR \varphi_{\bullet*}\sG^\bullet\bigr)^j &= \RRvarprojlim_{i \in
    \Ob(\varphi^{-1}(j))} \sG^i,\\
    \bigl(\RR f_{\bullet\bullet*}\sF^\bullet\bigr)^i &= \RR f_{i*}\sF^i.
  \end{align*}
  Now suppose \(K_\bullet\) is an \(I\)-object of \(\Cat\) and let \(\pi\colon
  \tot(K_\bullet) \to I\) be the projection functor.
  Let \(X_{\bullet\bullet}\) be a \(\tot(K_\bullet)\)-space with a
  \(\pi\)-augmentation \(a_\bullet\colon X_{\bullet\bullet} \to S_\bullet\) to
  an \(I\)-space \(S_\bullet\).
  Denote by \(a_i \colon X_{i\bullet} \to S_i\) the augmentation for the
  \(K_i\)-space \(X_{i\bullet}\) induced by \(a_\bullet\).
  Consider an object \(\sF^{\bullet\bullet}\) of
  \(\D^*(X_{\bullet\bullet\et},A\mhyphen\Mod)\).
  For every \(i \in \Ob(I)\), we have
  \[
    \bigl(\RR a_{\bullet*}\sF^{\bullet\bullet}\bigr)^i
    = \RR a_{i*}\sF^{i\bullet}
    = \RRvarprojlim_{k \in \Ob(K_i)} \RR a_{ik*} \sF^{ik}.
  \]
  The same statements hold for the pro-\'etale topology in cases
  \((\ref{setup:introalgebraicspaces})\) and \((\ref{setup:introadicspaces})\).
\end{proposition}
\begin{proof}
  Since restricting to each \(X_i\) or \((Y_\bullet
  \times_J I)_i\) has a left adjoint that is exact by definition in
  \cite[Chapitre VI, (5.3)]{Ill72},
  restricting to each \(X_i\) or \((Y_\bullet
  \times_J I)_i\) preserves \(K\)-injective objects by
  \cite[\href{https://stacks.math.columbia.edu/tag/08BJ}{Tag
  08BJ}]{stacks-project}.
  Thus, all properties follow from taking \(K\)-injective resolutions and
  computing component-wise because direct images are
  defined component-wise.
  For the last equation, we apply the third equation and the fact that \(K_i\)
  is coinitial with the fiber category \(\pi^{-1}(i)\).
\end{proof}
\par To state the main result of this subsection,
we define morphisms of cohomological descent.
\begin{definition}[Morphisms of cohomological descent; cf.\ {\cite[Expos\'e I, D\'efinition
  5.16]{GNAPGP88}}]\label{def:gnapgp516}
  Suppose \(K_\bullet\) is an \(I\)-object of \(\Cat\) and let \(\pi\colon
  \tot(K_\bullet) \to I\) be the projection functor.
  Let \(X_{\bullet\bullet}\) be a \(\tot(K_\bullet)\)-space with a
  \(\pi\)-augmentation
  \(a_\bullet\colon X_{\bullet\bullet} \rightarrow S_\bullet\)
  to an \(I\)-space \(S_\bullet\).
  We say that \(a_\bullet\) or \(X_{\bullet\bullet}\) is \textsl{of
  (lisse-)\'etale} (resp.\ \textsl{pro-\'etale}) \textsl{cohomological descent
  over \(S_\bullet\)}
  if, for every torsion ring \(A\) (resp.\ every ring \(A\)) and
  every (lisse-)\'etale (resp.\ pro-\'etale) sheaf \(\sF^\bullet\) of
  \(A\)-modules, the unit morphism
  \[
    \sF^\bullet \longrightarrow \RR a_{\bullet*}a_\bullet^{-1}\sF^\bullet
  \]
  is a quasi-isomorphism.
\end{definition}
Since \(K\)-injectivity in \(\D^\natural(X_{\et\bullet},A\mhyphen\Mod)\) can be checked
component-wise by the proof of Proposition \ref{prop:componentwise}
and derived functors are computed component-wise,
we can check whether a morphism
is of cohomological descent component-wise.
\begin{proposition}[cf.\ {\cite[Expos\'e I, Proposition 5.17]{GNAPGP88}}]
  \label{prop:gnapgp517}
  Suppose that \(\Sp\) is a subcategory of a category of spaces of one of the types
  \((\ref{setup:introalgebraicspaces})\textnormal{--}(\ref{setup:introadicspaces})\)
  that is essentially stable under fiber products, immersions, and proper
  morphisms.
  Let \(I\) be a small category.
  Suppose \(K_\bullet\) is an \(I\)-object of \(\Cat\) and let \(\pi\colon
  \tot(K_\bullet) \to I\) be the projection functor.
  Let \(X_{\bullet\bullet}\) be a \(\tot(K_\bullet)\)-space with a
  \(\pi\)-augmentation
  \(a_\bullet\colon X_{\bullet\bullet} \to S_\bullet\)
  to an \(I\)-space \(S_\bullet\).
  Then, \(a_\bullet\) is of (lisse-)\'etale (resp.\ pro-\'etale)
  cohomological descent over \(S_\bullet\) if and only
  if for every \(i \in \Ob(I)\), the augmentation
  \(a_i\colon X_{i\bullet} \to S_i\)
  over the \(K_i\)-space \(X_{i\bullet}\) is of (lisse-)\'etale (resp.\
  pro-\'etale) cohomological descent over
  \(S_i\).
\end{proposition}
Using Proposition \ref{prop:gnapgp517}, the proof of the following cubical
descent result is identical to that in \cite{GNAPGP88} after replacing
topological spaces with sites.
\begin{theorem}[Cubical descent for cohomology;
  cf.\ {\cite[Expos\'e I, Th\'eor\`eme 6.9]{GNAPGP88}}]\label{thm:gnapgp69}
  Suppose that \(\Sp\) is a subcategory of a category of spaces of one of the types
  \((\ref{setup:introalgebraicspaces})\textnormal{--}(\ref{setup:introadicspaces})\)
  that is essentially stable under fiber products, immersions, and proper
  morphisms.
  Let \(I\) be a small category and
  let \(S\) be an \(I\)-space.
  If \(X_\bullet \to S\) is an iterated cubical hyperresolution of \(S\), then
  \(X_\bullet\) is of cohomological descent over \(S\).
\end{theorem}
\begin{proof}
  The proof of \cite[Expos\'e I, Th\'eor\`eme 6.9]{GNAPGP88} applies directly to case
  \((\ref{setup:introcomplexanalyticgerms})\) and applies to the other cases
  once we have a six functor formalism for (lisse-)\'etale sheaves and
  pro-\'etale sheaves (when the latter are defined).
  For (lisse-)\'etale sheaves, see
  \cite{SGA43,LO08,LZ} in case \((\ref{setup:introalgebraicspaces})\) (which also
  applies to formal schemes by \cite[Expos\'e I, Corollaire 8.4]{SGA1new})
  and \cite{Ber93,Hub96} in cases \((\ref{setup:introberkovichspaces})\),
  \((\ref{setup:introrigidanalyticspaces})\), and
  \((\ref{setup:introadicspaces})\).
  For pro-\'etale sheaves, see \citeleft\citen{BS15}\citemid \S6\citepunct
  \citen{Cho20}\citeright\ in case
  \((\ref{setup:introalgebraicspaces})\) and see \cite[\S3]{Sch13} in case
  \((\ref{setup:introadicspaces})\).
\end{proof}

\begingroup
\makeatletter
\renewcommand{\@secnumfont}{\bfseries}
\part{Applications}
\label{part:applications}
\makeatother
\endgroup
In this part, we prove our applications of our results on cubical
hyperresolutions from Part \ref{part:cubical}.\bigskip
\section{An extension criterion for functors defined on regular spaces}\label{sect:gna02}
The goal of this section is to generalize the extension criterion for functors
defined on regular spaces, originally due to Guill\'en and Navarro Aznar for
smooth varieties
\cite[Th\'eor\`eme 2.1.5]{GNA02} and to Cirici and Guill\'en for
compactifiable complex analytic spaces \cite[Theorem
3.3]{CG14} to spaces of the form
\((\ref{setup:introalgebraicspaces})\textnormal{--}(\ref{setup:introadicspaces})\).
\subsection{Conventions}
We denote by \(\Sp\) a subcategory of a small category of reduced spaces of one of the types
\((\ref{setup:introalgebraicspaces})\textnormal{--}(\ref{setup:introadicspaces})\)
that is essentially stable under fiber products, immersions, and proper
morphisms.
We will assume \(\Sp\) is chosen according to the strategy outlined in
Remark \ref{rem:smallspaces}.
We denote by \(\Sp_\reg\) the subcategory of \(\Sp\) consisting of regular spaces.
We denote by \(I\) a small category.
While we will remind the reader of these assumptions in statements of results in
this section, we will not restate these assumptions in definitions or remarks.
\par We will also continue to use the terminology \textsl{\(I\)-space} from
Definition \ref{def:ispaces} to refer to objects in \(\Diag_I(\Sp)\).
\par The categories \(\Sp\) and \(\Sp_\reg\) correspond to what Guill\'en and Navarro
Aznar denote as \(\sM'\) and \(\sM\) in \cite{GNA02}, respectively.
\subsection{Categories of cohomological descent}
We define categories of cohomological descent.
See \cite[(1.5.2)]{GNA02} for the definition of a \textsl{quasi-strict monoidal
functor} and see \cite[D\'efinition 1.5.3]{GNA02} for a more explicit
description of some of the properties listed below.
\begin{definition}[Categories of cohomological descent
  {\cite[D\'efinitions 1.5.3 and 1.7.1]{GNA02}}]\label{def:catcohdescent}
  A \textsl{category of cohomological descent} is a quintuple
  \((\sD,E,\bfs,\mu,\lambda)\) satisfying the following properties.
  \begin{enumerate}[label=\((\mathrm{CD\arabic*})^{\op}\),
      ref=\mathrm{CD\arabic*}]
    \item \(\sD\) is a (not necessarily small) Cartesian category and has an
      initial object \(0\).
    \item \(E\) is a saturated class of morphisms of \(\sD\)
      stable under finite products.
    \item\label{def:cd3}
      \(\bfs\colon \Codiag_\Pi(\sD) \to \sD\) is a functor such that if
      \(\delta\colon \square \to \square'\) is a morphism of \(\Pi\) and
      \(X\colon \square \to \sD\) is a 1-codiagram of \(\sD\), then the morphism
      \(\bfs_\square X \to \bfs_{\square'}\delta_*X\)
      induced by the canonical morphism \(X \to \delta_*X\) is in \(E\).
    \item\label{def:cd4}
      For every object \(\square\) of \(\Pi\), the functor
      \(\bfs_\square\colon \Codiag_\square(\sD) \to \sD\)
      obtained from \(\bfs\) by restriction is a quasi-strict monoidal
      functor.
    \item Let \(f\colon X \to Y\) be a morphism of 1-codiagrams of \(\sD\) of
      type \(\square\) for an object \(\square\) of \(\Pi\).
      If \(f_\alpha\) is a member of \(E\) for every \(\alpha \in \square\), then
      \(\bfs_\square(f)\colon \bfs_\square X \to \bfs_\square Y\) is in \(E\).
    \item The functor
      \[
        \begin{tikzcd}[cramped,column sep=1.475em,row sep=0]
          (\bfs,\mu,\lambda_0)\colon &[-2.2em] \Pi \rar & \Coreal_\Pi(\sD)\\
          & \square \rar[mapsto] & \bigl(\bfs_\square\colon
          \Codiag_\square(\sD)
          \longrightarrow \sD\bigr)
        \end{tikzcd}
      \]
      is a quasi-strict monoidal functor.
    \item \(\lambda\) is a monoidal natural transformation between the monoidal functors
      \(\square \mapsto \id_{\sD}\) to the functor \(\square \mapsto \bfs_{\square} \circ
      i_{\square}\), which coincides with \(\lambda_0\) over \(\square_0\).
    \item\label{def:cd8}
      Let \(X\) be a 1-codiagram of \(\sD\) of type \(\square \coloneqq
      \square_S\), where \(S\) is a finite nonempty set, and consider an
      augmentation \(\varepsilon\colon X_0 \to X\).
      The morphism
      \(\lambda_\varepsilon \coloneqq \bfs_\square(\varepsilon) \circ
      \lambda_\square(X_0) \colon X_0 \to \bfs_\square X\)
      is in \(E\) if
      and only if the canonical morphism \(0 \to \bfs_{\square^+}X^+\) is
      in \(E\).
  \end{enumerate}
  By \cite[(1.5.1)]{GNA02},
  the product on \(\sD\) induces a monoidal structure
  on the category of fractions
  \(\Ho\sD \coloneqq \sD[E^{-1}]\)
  defined as in \cite[\S5.2]{Bor94} as long as \(\sD[E^{-1}]\) exists.
  \par The functor \(\bfs\) is called the \textsl{simple functor on \(\sD\)}.
  Morphisms in \(E\) are called \textsl{quasi-isomorphisms} or \textsl{weak
  equivalences} and objects of
  \(\sD\) quasi-isomorphic to \(0\) are called \textsl{acyclic}.
\end{definition}
\begin{examples}[Categories of cohomological descent]\label{ex:cohdescent}
  The following are some examples of categories of cohomological descent.
  We do not specify all the data necessary for the definition.
  See the references for details.
  \par Note that if \(\sA\) is a Grothendieck Abelian category in cases
  \((\ref{ex:cohdescentcochain})\) or \((\ref{ex:erqisos})\) below, then the
  homotopy category \(\Ho\sD \coloneqq \sD[E^{-1}]\) exists even though \(E\) is
  not a set by \cite[\href{https://stacks.math.columbia.edu/tag/05RU}{Tag
  05RU} and \href{https://stacks.math.columbia.edu/tag/09PA}{Tag
  09PA}]{stacks-project}.
  \begin{enumerate}[label=\((\roman*)\),ref=\roman*]
    \item\label{ex:cohdescentcochain}
      Let \(\sA\) be an additive category and consider the category
      \(\C^\natural(\sA)\)
      of cochain complexes of \(\sA\), where \(\natural \in \{*,+,b\}\).
      By \cite[Proposition 1.7.2]{GNA02} (resp.\ \cite[Proposition
      1.7.7]{GNA02}), \(\C^\natural(\sA)\) is of cohomological descent
      where \(E\) is the class of quasi-isomorphisms (resp.\ homotopy
      equivalences) and
      \(\bfs\) is the totalization functor.
    \item\label{ex:erqisos}
      Let \(\sA\) be an additive category and consider the category
      \(\C^\natural(\FF\sA)\) of cochain complexes of \(\sA\) together with a
      decreasing biregular filtration \(F\), where
      \(\natural \in \{*,+,b\}\).
      By \cite[Proposition 1.7.5]{GNA02}, \(\C^\natural(\FF\sA)\) is of cohomological descent
      where \(E\) is the class of filtered quasi-isomorphisms and
      \(\bfs\) is the filtered totalization functor.
      \par If \(\sA\) is Abelian, then \(\C^+(\FF\sA)\) is also of cohomological
      descent with respect to the class \(\cE_r\) of filtered morphisms that induce
      quasi-isomorphisms on the \(r\)-th page of the spectral sequence
      associated to the filtrations \cite[Theorem 2.8]{CG14}.
    \item\label{ex:cohdescentdubois} Let \(X\) be a complex analytic space.
      Following \cite[\S1]{DB81}, consider the category
      \(\CF_{\kern-1pt\diff}(X)\) (resp.\ \(\CF_{\kern-1pt\diff\kern-.5pt,\coh}(X)\))
      of cochain complexes \(K\) of \(\cO_X\)-modules that are bounded from below,
      together with a decreasing biregular filtration \(F\) by
      sub-\(\cO_X\)-modules and with a differential \(d\) that is a
      differential operator of order \(\le 1\) compatible with \(F\) such that
      \(\Gr_F(d)\) is \(\cO_X\)-linear (resp.\ such that \(\Gr_F(d)\) is
      \(\cO_X\)-linear and the cohomology of \(\Gr_F(K)\) is coherent).
      By \cite[Proposition 1.7.6]{GNA02}, \(\CF_{\kern-1pt\diff}(X)\) (resp.\
      \(\CF_{\kern-1pt\diff\kern-.5pt,\coh}(X)\)) is of cohomological descent
      where \(E\) is the class of filtered quasi-isomorphisms and
      \(\bfs\) is the filtered totalization functor.
      \par The same proof works when \(X\) is a space of type
      \((\ref{setup:introcomplexanalyticgerms})\textnormal{--}(\ref{setup:introadicspaces})\)
      or whenever \(X\) is differentially smooth over a field of characteristic
      zero and \(\Omega_X^1\) is of finite type
      in cases \((\ref{setup:introalgebraicspaces})\) and
      \((\ref{setup:introformalqschemes})\).
    \item Let \(k\) be a field of characteristic zero and consider the category
      \(\DGA(k)\) of anticommutative unital \(k\)-algebras
      differentially graded in degrees \(\ge 0\).
      By \cite[Proposition 1.7.4]{GNA02},
      \(\DGA(k)\) is of cohomological descent
      where \(E\) is the class of quasi-isomorphisms and
      \(\bfs\) is a cubical variant of the
      Thom--Whitney simple functor defined in \cite[(3.2)]{NA87}.
    \item By \cite[Example 1.4(3)]{CG14} (resp.\ \cite[Theorem 4.3]{CG14e1}),
      the category of mixed Hodge complexes \(\MHC\) (resp.\ the category of
      mixed Hodge diagrams \(\MHD\))
      together with a cubical analogue of Deligne's simple functor from
      \cite[\S8.1]{Del74} defined in \cite[Definition 2.5]{CG14}
      (resp.\ the Thom--Whitney simple functor
      \citeleft\citen{NA87}\citemid (3.2)\citepunct \citen{GNA02}\citemid
      Proposition 1.7.4\citeright) is of cohomological descent.
    \item By \cite[Theorem 6.1]{RiP07},
      the category \(\cM_f\) of fibrant objects of a pointed simplicial
      model category where the acyclicity criterion of 
      \cite[Definition 3.3.1]{RiP07} holds is of cohomological descent, where
      \(E\) is the class of weak equivalences.
    \item By \cite[Proposition 3.6]{PRiP09}, the category of fibrant spectra
      is of cohomological descent, where \(E\) is the class of homotopy
      equivalences and \(\bfs\) is the homotopy limit.
    \item Let \(R\) be a ring and let \(\mathcal{E}\) be a cofibrant
      \(E_\infty\)-operad over \(R\).
      By \cite[Theorem 4.7]{CC22}, the category \(\mathbf{E}_R\) of
      \(\mathcal{E}\)-algebras over \(R\) is of cohomological descent,
      where \(E\) is the class of quasi-isomorphisms.
  \end{enumerate}
\end{examples}
\subsection{Rectified functors}
To define rectified functors, we first describe how the Grothendieck
construction can be applied to the indexed category of 1-diagrams and
1-codiagrams over a homotopy category, where the indexing is by type.
This is a localized version of \(\Real_\fI(\sC)\) and \(\Coreal_\fI(\sC)\) from
Definition \ref{def:gna121}.
\begin{citeddef}[{\cite[(1.6.1)]{GNA02}}]\label{def:gna161}
  Let \(\fI\) be a subcategory of \(\Cat\).
  Let \(\sD\) be a category together with a saturated set of morphisms \(E\).
  We can form the localization functor
  \(\gamma\colon \sD \to \Ho\sD\)
  and if \(\sD\) is small, then \(\Ho\sD\) is also small
  \cite[Proposition 5.2.2]{Bor94}.
  The same discussion applies to the categories of 1-diagrams \(\Diag_I(\sD)\)
  and 1-codiagrams \(\Codiag_I(\sD)\) of \(\sD\) of type \(I\), where \(I\) is
  an object of \(\fI\).
  We therefore obtain functors
  \begin{alignat*}{3}
    \begin{tikzcd}[cramped,row sep=0,column sep=1.475em,ampersand replacement=\&,
      baseline=(HoDiag.base)]
      \fI^\op \rar \& \Cat\\
      I \rar[mapsto] \&|[alias=HoDiag]| \Ho\Diag_I(\sD)\\
      \varphi \rar[mapsto] \& \varphi^{\op*}\\
    \end{tikzcd}
    &\qquad \text{and} \qquad&
    \begin{tikzcd}[cramped,row sep=0,column sep=1.475em,ampersand replacement=\&,
      baseline=(HoCodiag.base)]
      \fI^\op \rar \& \Cat\\
      I \rar[mapsto] \&|[alias=HoCodiag]| \Ho\Codiag_I(\sD)\\
      \varphi \rar[mapsto] \& \varphi^*
    \end{tikzcd}
    \shortintertext{which by the Grothendieck construction \cite[Theorem B1.3.5]{Joh021}
    correspond to the categories}
    \omit\hfill\(\Ho_\fI \Diag_\fI(\sD)\)\hfill
    &\qquad \text{and} \qquad&
    \omit\hfill\(\Ho_\fI \Codiag_\fI(\sD)\)\hfill
  \end{alignat*}
  cofibered and fibered over \(\fI\), respectively.
\end{citeddef}
We need a 2-categorical version of the category of small categories \(\Cat\).
\begin{definition}[2-categories of types of diagrams {\cite[D\'efinition
  1.6.3]{GNA02}}]
  Let \(\fI\) be a 2-subcategory of the 2-category of small categories.
  We say that \(\fI\) is a \textsl{2-category of types of diagrams} if, for
  every pair of objects \(I,J\) of \(\fI\), we have
  \begin{equation}\label{eq:2catoftypesfaithful}
    \Hom_\fI(I,J) = \Hom_\TwoCat(I,J)
  \end{equation}
  and if the punctual category \(\underline{1}\) is an object of \(\fI\).
\end{definition}
\begin{citedex}[{\cite[(1.6.3)]{GNA02}}]
  The category \(\Phi\) of finite orderable categories is a 2-category of types
  of diagrams with the induced 2-categorical structure induced from \(\TwoCat\),
  but \(\Pi\) and \(\Pi^+\) are not 2-categories of types of diagrams
  since they do not satisfy
  \eqref{eq:2catoftypesfaithful}.
\end{citedex}
\par To discuss rectified functors, we need to consider the category of
1-diagrams as a 2-functor.
\begin{citeddef}[{\cite[(1.6.4)]{GNA02}}]
  Let \(\sC\) be a small category.
  Let \(\fI\) be a 2-category of types of diagrams.
  We then have the 2-functors
  \[
    \begin{tikzcd}[cramped,row sep=0,column sep=1.475em]
      \Diag_\fI(\sC)\colon &[-2.125em] \fI \rar & \TwoCat\\
      & I \rar[mapsto] & \Diag_I(\sC)\\
      & \varphi \rar[mapsto] & \varphi^{\op*}\\
      & \tau \rar[mapsto] & \tau^{\op*}
    \end{tikzcd}
    \qquad\text{and}\qquad
    \begin{tikzcd}[cramped,row sep=0,column sep=1.475em]
      \Codiag_\fI(\sC)\colon &[-2.125em] \fI \rar & \TwoCat\\
      & I \rar[mapsto] & \Codiag_I(\sC)\\
      & \varphi \rar[mapsto] & \varphi^{*}\\
      & \tau \rar[mapsto] & \tau^{*}
    \end{tikzcd}
  \]
  where the 2-categorical structures on 1-diagrams and 1-codiagrams of type \(I\)
  are given by the diagrams \eqref{eq:diagrams2cat}, and
  where for every 1-diagram \(Y\colon I^\op \to \sC\), the natural
  transformation
  \[
    \tau^{\op*}(Y)\colon \varphi^{\prime\op*}(Y) \longrightarrow \varphi^{\op*}Y
  \]
  is defined by the family \((Y(\tau_i))_{i \in I}\).
  If \(\sD\) is a category together with a saturated set of morphisms
  \(E\), we also consider the corresponding 2-functors
  \[
    \begin{tikzcd}[cramped,row sep=0,column sep=1.475em,
      /tikz/column 1/.append style={anchor=base east}]
      \Ho_\fI \Codiag_\fI(\sD) \colon &[-2.125em] \fI \rar & \TwoCat\\
      & I \rar[mapsto] & \Ho_\fI\Diag_\fI(\sD)\\
      \Ho_\fI \Diag_\fI(\sD) \colon &[-2.125em] \fI \rar & \TwoCat\\
      & I \rar[mapsto] & \Ho_\fI\Codiag_\fI(\sD)
    \end{tikzcd}
  \]
  using the definitions in Definition \ref{def:gna161}.
\end{citeddef}
\par We can now define rectified functors.
The class of functors defined below when \(\fI\) is the category \(\Phi\) of
finite orderable categories is the
class of functors that we will be able to extend.
\begin{definition}[Rectified functors {\cite[D\'efinition
  1.6.5 and (2.1.4)]{GNA02}}]\label{def:rectified}
  Let \(\fI\) be a 2-category of types of diagrams.
  Let \(\sD\) be a category together with a saturated set of morphisms \(E\).
  Consider a functor \(G\colon \sC^\op \to \Ho\sD\).
  A \textsl{\(\fI\)-rectification of \(G\)} is a pseudo-2-natural
  transformation of 2-functors
  \[
    G_\fI\colon \bigl(\Diag_\fI(\sC)\bigr)^\op
    \longrightarrow \Ho_\fI\Codiag_\fI(\sD)
  \]
  whose restriction to \(\sC\) under the restriction pseudo-2-natural
  transformation induced by \(\underline{1} \to \sC\) is \(G\).
  We say that \(G\colon \sC^\op \to \Ho\sD\) is a \textsl{\(\fI\)-rectified
  functor} if there exists a \(\fI\)-rectification of \(G\).
  \par Now consider two functors \(G\) and \(F\) with \(\fI\)-rectifications
  \(G_\fI\) and \(F_\fI\).
  A \textsl{morphism of \(\fI\)-rectified functors from \(G\) to \(F\)} is a
  modification \(\theta\colon G_\fI \to F_\fI\).
\end{definition}
See \cite[pp.\ 33--34]{GNA02} for an explicit list of conditions for a
rectification and a morphism of rectified functors.\medskip

\subsection{An extension criterion for functors defined on compactifiable
regular spaces}
\par We are now ready to state the main result of this section.
As mentioned before, the case for varieties
is due to Guill\'en and Navarro Aznar \cite[Th\'eor\`eme
2.1.5]{GNA02} and the case for complex analytic spaces is due to Cirici and
Guill\'en \cite[Theorem 3.3]{CG14}.
The key ingredients used to prove their results are the following:
\begin{enumerate}[label=\((\mathrm{\arabic*})\),ref=\mathrm{\arabic*}]
  \item\label{ingredient:gnares}
    Hironaka's resolutions of singularities \cite[Chapter 0, \S3, Main theorem
    I]{Hir64}.
  \item\label{ingredient:gnachow}
    A birational version of Chow's lemma using elementary proper modifications
    \citeleft\citen{Hir64}\citemid pp.\ 144--145\citepunct \citen{Hir75}\citemid p.\
    505\citeright,
    called the \textsl{Chow--Hironaka lemma} in
    \cite[(2.1.3)]{GNA02}.
\end{enumerate}
\par We extend the results of Guill\'en--Navarro Aznar \cite{GNA02}
and Cirici--Guill\'en \cite{CG14}
to compactifiable spaces of types
\((\ref{setup:introalgebraicspaces})\textnormal{--}(\ref{setup:introadicspaces})\).
The key innovation in our proof is to use the version of the weak factorization theorem
\cite{AKMW02,Wlo03} due to Abramovich and Temkin \cite{AT19} to replace the
Chow--Hironaka lemma.
This is necessary because appropriate versions of the Chow--Hironaka lemma are
not yet known for arbitrary spaces of the form
\((\ref{setup:introalgebraicspaces})\textnormal{--}(\ref{setup:introadicspaces})\).
\begin{theorem}[An extension criterion for functors defined on compactifiable
  regular spaces]
  \label{thm:gna215}
  Suppose that \(\Sp\) is a subcategory of a small category of
  finite-dimensional reduced spaces of one of the
  types
  \((\ref{setup:introalgebraicspaces})\textnormal{--}(\ref{setup:introadicspaces})\).
  In cases \((\ref{setup:introalgebraicspaces})\) and
  \((\ref{setup:introformalqschemes})\), suppose that all objects in \(\Sp\) are
  Noetherian and quasi-excellent of equal characteristic zero.
  In cases \((\ref{setup:introberkovichspaces})\),
  \((\ref{setup:introrigidanalyticspaces})\), and
  \((\ref{setup:introadicspaces})\), suppose that
  \(k\) is of characteristic zero.
  \par Let \(\Sp_\reg\) be the full subcategory of \(\Sp\) consisting of regular
  spaces.
  Let \(\sD\) be a small category of cohomological descent and let
  \[
    G\colon \bigl(\Sp_\reg\bigr)^{\op} \longrightarrow \Ho(\sD)
  \]
  be a \(\Phi\)-rectified functor
  satisfying the following conditions:
  \begin{enumerate}[label=\((\mathrm{F}\arabic*)\),ref=\mathrm{F\arabic*}]
    \item\label{thm:gna251f1}
      \(G(\emptyset) = 1\) and the canonical morphism
      \(G(X \sqcup Y) \to G(X) \times G(Y)\)
      is an isomorphism.
    \item\label{thm:gna251f2}
      For every elementary acyclic square \(X_\bullet\) in \(\Sp_\reg\),
      \(\bfs G(X_\bullet)\) is acyclic.
  \end{enumerate}
  Then, there exists an extension of \(G\) to a \(\Phi\)-rectified functor
  \[
    G'\colon \Sp^{\op} \longrightarrow \Ho(\sD)
  \]
  satisfying the following condition:
  \begin{enumerate}[label=\((\mathrm{D})\),ref=\mathrm{D}]
    \item\label{thm:gna251d}
      For every acyclic square \(X_\bullet\) in \(\Sp\), \(\bfs
      G'(X_\bullet)\) is acyclic.
  \end{enumerate}
  \par Moreover, suppose
  \(F\colon (\Sp_\reg)^\op \to \Ho(\sD)\) and
  \(G\colon (\Sp_\reg)^\op \to \Ho(\sD)\)
  are \(\Phi\)-rectified
  functors satisfying conditions \((\ref{thm:gna251f1})\) and
  \((\ref{thm:gna251f2})\) with extensions \(F'\) and \(G'\) satisfying
  \((\ref{thm:gna251d})\).
  If \(\tau\colon F \to G\) is a morphism of \(\Phi\)-rectified functors, then
  \(\tau\) extends uniquely to a morphism of \(\Phi\)-rectified functors
  \(\tau'\colon F' \to G'\).
\end{theorem}
\begin{proof}
  We follow the proofs in \cite{GNA02,CG14}, using \(\Hrc_\proj\) instead of
  \(\Hrc\) and replacing the definition of an augmented 2-resolution in
  \cite[p.\ 43]{GNA02} with ours (see Remark \ref{rem:steenbrink2}).
  The main idea is to consider the diagram
  \begin{equation}\label{eq:2fun}
    \begin{tikzcd}[baseline=(midline.base)]
      \bigl(\Hrc_\proj\bigl(\Diag_\Phi(\Sp)\bigr)\bigr)^\op \rar[hook] \dar
      & \bigl(\Diag_\Phi\bigl(\Diag_\Pi(\sM)\bigr)\bigr)^\op \dar{G_\Phi}\\
      \bigl(\Ho_\Phi\Hrc_\proj\bigl(\Diag_\Phi(\Sp)\bigr)\bigr)^\op
      \arrow[d,"(\Ho w_{\Phi,\proj})^\op"',"\sim"{sloped}]
      &|[alias=midline]|
      \Ho_\Phi \Codiag_\Phi\bigl(\Diag_\Pi(\sD)\bigr) \dar{\bfs}\\
      \bigl(\Diag_\Phi(\Sp)\bigr)^\op \rar[dashed]{G'_\Phi}
      & \Ho_\Phi\Codiag_\Phi(\sD)
    \end{tikzcd}
  \end{equation}
  of pseudo-2-natural transformations of 2-functors.
  The goal is to construct \(G'_\Phi\) as a dashed 2-natural
  pseudo-transformation of 2-functors that makes the diagram commute.
  In the top right,
  \(G_\Phi\) comes from restricting the \(\Phi\)-rectification
  induced by the action of \(G\) on 1-diagrams with types in \(\Pi \times
  \Pi^+\) \cite[Proposition 1.6.8]{GNA02}.
  The pseudo-2-natural equivalence
  in the bottom left comes from Theorem \ref{thm:gnapgp38}, \cite[Theorem
  B1.3.6]{Joh021}, and the compatibility
  of \(w_{I,\proj}\) (resp.\ \(\Ho w_{I,\proj}\)) with pullbacks and natural
  transformations in \(\Phi\).
  The assignment \(I \mapsto \Hrc_\proj(\Diag_I(\Sp))\) (resp.\
  \(I \mapsto \Ho_\Phi\Hrc_\proj(\Diag_I(\Sp))\)) is the 2-functor which we denote by
  \(\Hrc_\proj(\Diag_\Phi(\Sp))\) (resp.\
  \(\Ho_\Phi\Hrc_\proj(\Diag_\Phi(\Sp))\)).
  By \cite[Theorem B1.3.6]{Joh021}, it suffices to show that the dashed
  pseudo-2-natural transformation in \eqref{eq:2fun} exists after taking fibers
  over each object \(I\) of \(\Phi\).
  By Corollary \ref{cor:gnapgp310}, it suffices to show that for every morphism
  of iterated cubical hyperresolutions \(X_\bullet \to X'_\bullet\) of a fixed
  \(I\)-space \(X\), the induced morphism
  \begin{equation}\label{eq:wanttoshow215}
    G(X'_\bullet) \longrightarrow G(X_\bullet)
  \end{equation}
  is a quasi-isomorphism.
  By Lemma \ref{lem:how-1connected} and Remark \ref{rem:how-1connected}, we
  reduce to the case when \(X_\bullet \to X'_\bullet\) is a cubical hyperresolution.
  Moreover, since cubical hyperresolutions and morphisms 
  between them are defined component-wise as \(i \in \Ob(I)\) varies by
  Proposition \ref{prop:gnapgp214}, Definition \ref{def:gnapgp33}, and
  Definition \ref{def:gnapgp32}, we reduce to the case when \(I =
  \underline{1}\) and \(X_\bullet = X\).
  \par The main ingredients in the remainder of the proof in \cite{GNA02} are
  \cite[Expos\'e I, Proposition 2.10, Th\'eor\`eme 2.15, Th\'eor\`eme 3.8,
  and Corollaire 3.10]{GNAPGP88} and the application of the Chow--Hironaka lemma
  in \cite[Lemme 1]{GNA02}.
  The former can be replaced with
  our existence results for augmented 2-resolutions (Proposition
  \ref{prop:gnapgp210}) and projective cubical hyperresolutions (Theorem
  \ref{thm:gnapgp215}\((\ref{thm:gnapgp215compact})\)) in \(\Sp\),
  together with our results for the homotopy
  category of iterated hyperresolutions in \(\Sp\)
  (Theorem \ref{thm:gnapgp38}\((\ref{thm:gnapgp38compact})\) and
  Corollary \ref{cor:gnapgp310}).
  Our assumption on compactifiability is used to ensure the spaces
  involved have only finitely many irreducible components, which is used in
  \cite[Lemme 1]{GNA02} (see Lemma
  \ref{lem:gna1} below) and in \cite[(2.1.7)]{GNA02}.
  See \cite[Remark 3.2]{CG14}.\medskip
  \par It therefore remains to show the analogue of \cite[Lemme
  1]{GNA02} in our context.
  To make sense of the statement of this result, we explain some of the notation
  in \cite{GNA02}.
  For every integer \(n \ge 0\), denote by \(\Sp_n\) the full subcategory of
  \(\Sp\) consisting
  of objects of dimension \(\le n\) and let \(\Sp_{n,0}\) be the
  full subcategory of \(\Sp_n\) whose objects are disjoint unions of objects in
  \(\Sp_{n-1}\) and objects in \(\Sp_\reg\) of dimension \(\le n\).
  For \(\nu \in \{n,(n,0)\}\), consider the following assertion:
  \begin{enumerate}[label=\((\mathrm{T}_\nu)\),ref=\mathrm{T}]
    \item\label{cond:tnu}
      \(G\) induces a \(\Phi\)-rectified functor
      \(G'\colon (\Sp_\nu)^\op \rightarrow \Ho\sD\)
      that satisfies \((\ref{thm:gna251d})\).
  \end{enumerate}
  The proof in \cite{GNA02} shows that to show Theorem \ref{thm:gna215}, it
  suffices to show that
  \((\ref{cond:tnu}_n)\) holds for every \(n \ge 0\).
  To prove \((\ref{cond:tnu}_n)\) holds for every \(n \ge 0\),
  the proof can be broken up into the following two assertions:
  \begin{alignat*}{3}
    (\ref{cond:tnu}_{n-1}) &{}\Longrightarrow
    (\ref{cond:tnu}_{n,0}) &\qquad& \text{for every}\
    n > 0,\\
    (\ref{cond:tnu}_{n,0}) &{}\Longrightarrow
    (\ref{cond:tnu}_{n}) &\qquad& \text{for every}\ n
    > 0.
  \end{alignat*}
  With the replacements listed in the second paragraph of this proof, the proofs
  of \((\ref{cond:tnu}_{0})\) and of \((\ref{cond:tnu}_{n,0}) \Rightarrow
  (\ref{cond:tnu}_{n})\) in \cite{GNA02} work in our context.\medskip
  \par Most of the proof of \((\ref{cond:tnu}_{n-1}) \Rightarrow
  (\ref{cond:tnu}_{n,0})\) works in our context as well.
  The only part of the proof that is substantially different is the following:
  \begin{lemma}[cf.\ {\cite[Lemme 1]{GNA02}}]\label{lem:gna1}
    Suppose \((\ref{cond:tnu}_{n-1})\) holds.
    Let \(X_\bullet\) be an augmented 2-resolution 
    \[
      \begin{tikzcd}
        \tilde{Y} \rar[hook]{j} \dar[swap]{g} & \tilde{X} \dar{f}\\
        Y \rar[hook]{i} & X
      \end{tikzcd}
    \]
    of an object \(X\) of \(\Sp_\reg\) of pure dimension \(n\) such that \(f\) is
    a birational/bimeromorphic blowup morphism.
    Then, \(\bfs G'(X_\bullet)\) is acyclic.
  \end{lemma}
  In \cite{GNA02}, there is no assumption that \(f\) is a blowup.
  In our context, we are allowed to assume that \(f\) is a blowup because we use Theorem
  \ref{thm:gnapgp215}\((\ref{thm:gnapgp215compact})\) to construct
  iterated cubical hyperresolutions using blowups
  and are working with \(\Hrc_\proj\).
  Following the proof in \cite{GNA02} (we use the compactifiability of
  \(X_{00}\) to induce on the number of irreducible components in case
  \((iv)\) in \cite[p.\ 48]{GNA02}; cf.\ \cite[Remark 3.2]{CG14}),
  it suffices to show the special case of
  Lemma \ref{lem:gna1} when both \(X\) and \(\tilde{X}\) are irreducible objects
  of \(\Sp_\reg\) of dimension \(n\).\medskip
  \par Instead of using a version of the Chow--Hironaka lemma, we use the weak
  factorization theorem for spaces of types
  \((\ref{setup:introalgebraicspaces})\textnormal{--}(\ref{setup:introrigidanalyticspaces})\)
  due to Abramovich and Temkin \cite[Theorem 1.4.1]{AT19}.
  In case \((\ref{setup:introadicspaces})\), the same proof as in \cite{AT19}
  applies using the GAGA theorems from \cite[\S6]{Hub07} together with the fact
  that transition morphisms between affinoid subdomains of adic spaces locally
  of weakly finite type over a field are regular by \cite[Lemma 24.9]{LM}.
  The weak factorization theorem \cite[Theorem 1.4.1]{AT19}
  says that \(f\colon \tilde{X} \to X\) can be
  written as a composition
  \[
    \begin{tikzcd}[cramped,column sep=1.475em]
      \tilde{X} = V_0 \rar[dashed]{\varphi_1}
      &[0.125em] V_1 \rar[dashed]{\varphi_2}
      &[0.125em] \cdots \rar[dashed]{\varphi_{l-1}}
      &[0.775em] V_{l-1} \rar[dashed]{\varphi_l}
      & V_l = X
    \end{tikzcd}
  \]
  of rational/meromorphic maps between regular spaces such that
  given a choice of normal crossing divisor \(D \subseteq
  X\), the following properties hold:
  \begin{enumerate}[label=\((\arabic*)\),ref=\arabic*]
    \item \(f = \varphi_l \circ \varphi_{l-1} \circ \cdots \circ \varphi_1\).
    \item The compositions \begin{tikzcd}[cramped,column sep=small]V_i \rar[dashed]
      & X\end{tikzcd}
      are morphisms and
      induce isomorphisms along \(X - (\Delta(f) \cup D)\).
    \item For every \(i\), either \(\varphi_i\) or \(\varphi_i^{-1}\) is the blowup along a
      regular closed subspace \(Z_i\) in \(V_i\) or \(V_{i-1}\), respectively,
      where the \(Z_i\) are disjoint from the inverse image of \(X - (\Delta(f) \cup D)\).
    \item\label{cond:at19nc}
      The inverse images \(D_{V_i} \subseteq V_i\) of the \(D \subseteq X\)
      are normal crossing divisors and have normal crossings with \(Z_i\).
    \item For every \(i\), the morphism \(V_i \to X\) is the blowup along a
      coherent ideal sheaf cosupported in \(D \cup \Delta(f)\).
  \end{enumerate}
  \par We proceed by induction on the length \(l \ge 1\) of the weak
  factorization.
  The base case when \(l = 1\) is case \((i)\) in \cite[pp.\
  45--46]{GNA02}.
  Consider the inductive case when \(l > 1\).
  Then, either \(\varphi_1\) or \(\varphi_1^{-1}\) is a blowup along a regular
  closed subspace.
  The case when \(\varphi_1\) is a blowup along \(Z_1 \subseteq V_1\) follows by
  the same argument as for case \((ii)\) in \cite[pp.\ 46--47]{GNA02} using
  the commutative diagram
  \[
    \begin{tikzcd}[column sep=large]
      \tilde{Y} \rar \dar[hook] & W_1 \rar \dar[hook] & Y \dar[hook]\\
      \tilde{X} \rar{\varphi_1} & V_1 \rar{\varphi_l \circ
      \cdots \circ \varphi_2} & X
    \end{tikzcd}
  \]
  where both squares are Cartesian.
  Now suppose that \(\varphi_1\) is a blowup along \(Z_1 \subseteq \tilde{X}\).
  We then have the commutative diagram
  \[
    \begin{tikzcd}[column sep=large]
      \tilde{Y} \rar[leftarrow] \dar[hook] & W_1 \rar \dar[hook] & Y \dar[hook]\\
      \tilde{X} \rar[leftarrow]{\varphi_1} & V_1 \rar{\varphi_l \circ
      \cdots \circ \varphi_2} & X
    \end{tikzcd}
  \]
  where both squares are Cartesian (using the fact that in \(\Sp\), fiber
  products are reduced).
  By case \((i)\) in \cite[pp.\ 45--46]{GNA02} and the inductive hypothesis,
  respectively, we know that the morphisms
  \[
    \bfs_{\square^+_0} G'(\tilde{Y} \to \tilde{X}) \longrightarrow
    \bfs_{\square^+_0} G'(W_1 \to V_1) \longleftarrow
    \bfs_{\square^+_0} G'(Y \to X)
  \]
  are quasi-isomorphisms.
  The special case of Lemma \ref{lem:gna1} when both \(X\) and \(\tilde{X}\) are
  irreducible objects of \(\Sp_\reg\) of dimension \(n\) now follows from
  \cite[Propositions 1.5.7(2), 1.5.9(1), and 1.5.6(1)]{GNA02} together with
  \((\ref{def:cd8})^\op\).
\end{proof}
\subsection{An extension criterion for functors defined on regular pairs
\texorpdfstring{\except{toc}{(\emph{X},\emph{U})}\for{toc}{\((X,U)\)}}{(X,U)}}
\label{sect:cg14}
To prove the existence of weight filtrations on cohomology of compactifiable
spaces, we need a version of our extension criterion
(Theorem \ref{thm:gna215}) for pairs consisting of a space and an \emph{open}
subspace.
For varieties and complex analytic spaces,
this version of the extension criterion for pairs
is due to Guill\'en and Navarro Aznar \cite[Th\'eor\`emes 2.3.3, 2.3.6, and
4.8]{GNA02} and to Cirici and Guill\'en \cite[Theorem
3.9]{CG14}, respectively.\medskip
\par We start by defining what we mean by pairs.
Note that this notion of a pair is equivalent to the notion of pairs \((S,S')\) from
Definition \ref{def:closedpairs}, except we replace the second argument \(S'\)
by its complement.
Because compactifications of spaces are not unique, even up to
birational/bimeromorphic equivalence (in case
\((\ref{setup:introcomplexanalyticgerms})\), see Serre's example in
\cite[Chapter VI, Example 3.2]{Har70} and \cite[Example 6.5]{CG14}),
we have to work with a category that
remembers the birational/bimeromorphic equivalence class of a choice of
compactification.
Recall that \(\Sp\) denotes a small category of spaces.
\begin{definition}[The category \(\Sp^2\)
  {\citeleft\citen{GNA02}\citemid (2.3) and (4.5)\citepunct
  \citen{CG14}\citemid pp.\ 89--90\citeright}]
  \label{def:sp2}
  Suppose that \(\Sp\) is a subcategory of a small category of
  reduced spaces of one of the
  types
  \((\ref{setup:introalgebraicspaces})\textnormal{--}(\ref{setup:introadicspaces})\).
  We let \(\Sp^{2}\) denote the category of pairs \((X,U)\) where \(X\) is
  a space and \(U\) is an open subspace such that \(D = X - U\) is a closed
  subspace of \(X\).
  Morphisms \(f\colon (X',U') \to (X,U)\) are morphisms \(f\colon X' \to X\)
  such that \(f\) is a morphism of pairs in the sense of Definition
  \ref{def:closedpairs}.
  We let \(\Sp_\reg^2\) denote the full subcategory of \(\Sp^{2}\)
  consisting of pairs \((X,U)\) such that \(X\) is regular and \(D = X - U\) is a
  normal crossing divisor.
  We let \(\Sp^{2}_\comp\) and \(\Sp^2_{\reg,\comp}\) denote the full
  subcategories of \(\Sp^{2}\) and \(\Sp_\reg^2\), respectively, consisting of
  pairs \((X,U)\) such that \(X\) is compactifiable.
\end{definition}
We define acyclic squares as follows.
\begin{definition}[Acyclic squares in \(\Sp^2\) {\citeleft\citen{GNA02}\citemid p.\
  60\citepunct \citen{CG14}\citemid Definitions 3.4--3.6\citeright}]
  \label{def:acyclicsquarepairs}
  Consider the square
  \begin{equation}\label{eq:acyclicsquarepairs}
    \begin{tikzcd}
      (\tilde{Y},\tilde{U} \cap \tilde{Y})
      \rar[hook]{j} \dar[swap]{g} & (\tilde{X},\tilde{U}) \dar{f}\\
      (Y,U \cap Y) \rar[hook]{i} & (X,U)
    \end{tikzcd}
  \end{equation}
  in \(\Sp^{2}\).
  We say the square \eqref{eq:acyclicsquarepairs}
  is an \textsl{acyclic square} if the following properties hold:
  \begin{enumerate}[label=\((\roman*)\)]
    \item \(i\colon Y \to X\) is a closed immersion.
    \item \(f\colon \tilde{X} \to X\) is proper.
    \item The diagram of the first components is Cartesian.
    \item \(f^{-1}(U) = \tilde{U}\).
    \item The diagram of the second components is an
      acyclic square of \(\Sp\).
  \end{enumerate}
  We say an acyclic square \eqref{eq:acyclicsquarepairs}
  is an \textsl{elementary acyclic square} and that the
  morphism \(f\) is an \textsl{elementary proper modification} if the
  following additional properties hold:
  \begin{enumerate}[label=\((\roman*)\),resume,ref=\roman*]
    \item\label{def:acyclicsquarepairsnc}
      \(f\colon \tilde{X} \to X\) is the blowup of \(X\) along the
      center \(Y\) which is regular and has normal crossings with 
      \(D = X - U\).
    \item The diagram of the second components is an elementary acyclic square
      of \(\Sp_\reg\).
  \end{enumerate}
\end{definition}
We will in fact work with a localization of the category \(\Sp^{2}_\comp\).
We denote by
\[
  \begin{tikzcd}[cramped,column sep=1.475em,row sep=0]
    \gamma\colon &[-2.125em] \Sp^{2}_\comp \rar & \Sp\\
    & (X,U) \rar[mapsto] & U
  \end{tikzcd}
\]
the forgetful functor and let \(\Sigma\) be the set of morphisms \(s\) of
\(\Sp^{2}_\comp\) such that \(\gamma(s)\) is an isomorphism.
By the universal property of categories of fractions
\cite[Proposition 5.2.2]{Bor94}, we have the induced functor
\[
  \eta\colon \Sp^{2}_\comp[\Sigma^{-1}] \longrightarrow \Sp.
\]
By \cite[Lemme 4.6]{GNA02} (whose proof applies in \(\Sp\) with only
notational changes), the localized category \(\Sp^{2}_\comp[\Sigma^{-1}]\) is
equivalent to the category \(\Sp_\infty\) defined as follows.
\begin{definition}[The category \(\Sp_\infty\); cf.\ 
  {\citeleft\citen{GNA02}\citemid (4.7)\citepunct
  \citen{CG14}\citemid Definition 3.8\citeright}]
  Suppose that \(\Sp\) is a subcategory of a small category of
  reduced spaces of one of the
  types
  \((\ref{setup:introalgebraicspaces})\textnormal{--}(\ref{setup:introadicspaces})\).
  The category \(\Sp_\infty\) is defined as follows.
  An object \(U_\infty\) of \(\Sp_\infty\) is a space \(U\) in \(\Sp\)
  together with an equivalence class of compactifications \(X\) of \(U\), where
  two compactifications are \textsl{equivalent} if they are
  birational/bimeromorphic.
  In this situation, we say that \(U\) is a \textsl{compactifiable space
  with the birational/bimeromorphic structure \(U_\infty\) at infinity}.
  A \textsl{compactification} of a morphism \(g\colon U' \to U\) of
  compactifiable spaces with birational/bimeromorphic structures \(U'_\infty\)
  and \(U_\infty\) at infinity is an extension of \(f\colon X' \to X\) of \(g\)
  where \(X'\) and \(X\) are representatives of \(U'_\infty\) and \(U_\infty\),
  respectively.
  \par Two compactifications \(f'_1\colon X'_1 \to X_1\) and \(f'_2\colon X'_2 \to
  X_2\) of \(g\) are \textsl{equivalent} if there exists a third
  compactification \(f'_3\colon X'_3 \to X_3\) that dominates \(f'_1\) and
  \(f'_2\).
  Morphisms \(g_\infty\colon U'_\infty \to U_\infty\) of \(\Sp_\infty\)
  are equivalence classes of morphisms \(g\colon U' \to U\) with respect to the
  structures \(U'_\infty\) and \(U_\infty\).
\end{definition}
We are now ready to state the extension criterion for
pairs \((X,U)\).
\begin{theorem}[An extension criterion for functors defined on regular pairs
  \((X,U)\)]\label{thm:cg39}
  Suppose that \(\Sp\) is a subcategory of a small category of
  finite-dimensional reduced spaces of one of the types
  \((\ref{setup:introalgebraicspaces})\textnormal{--}(\ref{setup:introadicspaces})\).
  In cases \((\ref{setup:introalgebraicspaces})\) and
  \((\ref{setup:introformalqschemes})\), suppose that all objects in \(\Sp\) are
  Noetherian and quasi-excellent of equal characteristic zero.
  In cases \((\ref{setup:introberkovichspaces})\),
  \((\ref{setup:introrigidanalyticspaces})\), and
  \((\ref{setup:introadicspaces})\), suppose that
  \(k\) is of characteristic zero.
  \par
  Let \(\sD\) be a small category of cohomological descent and let
  \[
    G\colon \bigl(\Sp^{2}_{\reg,\comp}\bigr)^\op \longrightarrow \Ho(\sD)
  \]
  be a \(\Phi\)-rectified functor
  satisfying the following conditions:
  \begin{enumerate}[label=\((\mathrm{F}\arabic*)^2\),ref=\mathrm{F\arabic*}]
    \item\label{thm:cg39f1}
      \(G(\emptyset,\emptyset) = 1\) and the canonical morphism
      \(G((X,U) \sqcup (Y,V)) \to G(X,U) \times G(Y,V)\)
      is an isomorphism.
    \item\label{thm:cg39f2}
      For every elementary acyclic square \((X_\bullet,U_\bullet)\) in
      \(\Sp^{2}_{\reg,\comp}\), \(\bfs G(X_\bullet,U_\bullet)\) is acyclic.
  \end{enumerate}
  Then, there exists a \(\Phi\)-rectified functor
  \[
    G'\colon \bigl(\Sp_\infty\bigr)^\op \longrightarrow \Ho(\sD)
  \]
  satisfying the following conditions:
  \begin{enumerate}[label=\((\arabic*)\),ref=\arabic*]
    \item\label{thm:cg39e} For every object \((X,U)\) of \(\Sp^{2}_{\reg,\comp}\), we have
      \(G'(U_\infty) \simeq G(X,U)\).
    \item\label{thm:cg39d}
      For every acyclic square \((X_\bullet,U_\bullet)\) in \(\Sp^{2}_{\comp}\), \(\bfs
      G'(X_\bullet,U_\bullet)\) is acyclic.
  \end{enumerate}
  \par Moreover, suppose
  \(F\colon (\Sp^{2}_{\comp})^\op \to \Ho(\sD)\) and
  \(G\colon (\Sp^{2}_{\comp})^\op \to \Ho(\sD)\)
  are \(\Phi\)-rectified
  functors satisfying conditions \((\ref{thm:cg39f1})^2\) and
  \((\ref{thm:cg39f2})^2\) with functors \(F'\) and \(G'\) satisfying
  \((\ref{thm:cg39e})\) and \((\ref{thm:cg39d})\).
  If \(\tau\colon F \to G\) is a morphism of \(\Phi\)-rectified functors, then
  \(\tau\) extends uniquely to a morphism of \(\Phi\)-rectified functors
  \(\tau'\colon F' \to G'\).
\end{theorem}
\begin{proof}
  We combine the proof of 
  \citeleft\citen{GNA02}\citemid Th\'eor\`eme 2.3.3\citeright\ with our changes in
  Theorem \ref{thm:gnapgp215}.
  First, we need versions of augmented 2-resolutions and cubical
  hyperresolutions for pairs, which exist by Theorem \ref{thm:gnapgp26}, Proposition
  \ref{prop:gnapgp210}, and Theorem \ref{thm:gnapgp215}\((\ref{thm:gnapgp215compact})\).
  Next, we need versions of our results for the homotopy
  category of iterated hyperresolutions in \(\Sp\)
  (Theorem \ref{thm:gnapgp38}\((\ref{thm:gnapgp38compact})\) and
  Corollary \ref{cor:gnapgp310}) for pairs.
  These hold by Remark \ref{rem:iteratednotpairs}.
  As in the proof of Theorem \ref{thm:gnapgp215}, we replace the Chow--Hironaka
  lemma used in \cite{GNA02} with the weak factorization theorem proved in
  \cite[Theorem 1.4.1]{AT19}.
  Note that this weak factorization theorem is already adapted to the setting of
  pairs \((X,U)\); for the normal crossings condition in Definition
  \ref{def:acyclicsquarepairs}\((\ref{def:acyclicsquarepairsnc})\), see
  \((\ref{cond:at19nc})\) on p.\ \pageref{cond:at19nc}.
  With these changes, we follow the proofs of \cite[Th\'eor\`eme 2.3.3]{GNA02}
  and Theorem \ref{thm:gnapgp215}
  to construct a \(\Phi\)-rectified functor
  \begin{align*}
    G''\colon \bigl(\Sp^{2}_{\comp}\bigr)^\op &\longrightarrow \Ho \sD
    \shortintertext{satisfying analogues of \((\ref{thm:cg39e})\) and
    \((\ref{thm:cg39d})\) such that \(G''\) is an essentially unique extension
    of \(G\).
    We can then follow the proof of \cite[Th\'eor\`eme 2.3.6]{GNA02} to construct
    the \(\Phi\)-rectified functor \(G'\) using the equivalence of categories}
    \eta\colon \Sp^{2}_{\comp}[\Sigma^{-1}] &\overset{\sim}{\longrightarrow}
    \Sp_\infty,
  \end{align*}
  which holds by \cite[Lemme 4.6]{GNA02} (whose proof applies in \(\Sp\) with
  only notational changes).
\end{proof}
\section{The Deligne--Du Bois complex and Du Bois
singularities}\label{sect:delignedubois}
In this section, we construct the Deligne--Du Bois complex
\(\underline{\Omega}_X^\bullet\) and the logarithmic Deligne--Du Bois complex
\(\underline{\Omega}_X^\bullet\langle D \rangle\) for many spaces in the categories
\((\ref{setup:introalgebraicspaces})\textnormal{--}(\ref{setup:introadicspaces})\).
The construction for varieties is due to Du Bois \cite[\S\S3--4]{DB81} (see also
\citeleft\citen{GNAPGP88}\citemid Expos\'e V, \S\S3--4\citeright),
the construction for
complex analytic spaces is due to Guill\'en and Navarro Aznar
\cite[Th\'eor\`emes 4.1 and 4.3]{GNA02}, and
the construction for adic spaces of finite type over a \(p\)-adic field is due
to Guo \cite[\S5]{Guo23}.
We extend their results to algebraic spaces and formal schemes of finite type
over a field of characteristic zero and various kinds of non-Archimedean
analytic spaces over more general fields of characteristic zero.\medskip
\par Additionally, we construct the 0-th graded piece \(\underline{\Omega}_X^0\) of
the Deligne--Du Bois complex whenever our existence result for cubical
hyperresolutions (Theorem \ref{thm:gnapgp215}) applies.
We will use \(\underline{\Omega}_X^0\) to define Du Bois singularities in this
generality (Definition \ref{def:dubois}).
\subsection{Conventions}
We denote by \(\Sp\) a subcategory of a (not necessarily small) category of reduced spaces
of one of the types
\((\ref{setup:introalgebraicspaces})\textnormal{--}(\ref{setup:introadicspaces})\)
that is essentially stable under fiber products, immersions, and proper
morphisms.
We will assume \(\Sp\) is chosen according to the strategy outlined in
Remark \ref{rem:smallspaces}.
We denote by \(\Sp_\reg\) the subcategory of \(\Sp\) consisting of regular spaces.
We denote by \(I\) a small category.
While we will remind the reader of these assumptions in statements of results in
this section, we will not restate these assumptions in definitions or remarks.
\par We will also continue to use the terminology \textsl{\(I\)-space} from
Definition \ref{def:ispaces} to refer to objects in \(\Diag_I(\Sp)\).
\subsection{The Deligne--Du Bois complex}
We prove the Deligne--Du Bois complex \(\underline{\Omega}_X^\bullet\)
exists using our version of the extension criterion (Theorem \ref{thm:gna215}).
In \cite{GNA02}, Guill\'en and Navarro Aznar used their extension criterion
to construct \(\underline{\Omega}_X^\bullet\) for complex analytic spaces.
\par Our analogue of the Poincar\'e lemma \((\ref{thm:gna413})\) in case
\((\ref{setup:introadicspaces})\) uses uses the pro-\'etale topology, the
(positive) de Rham sheaf, and the (positive) structural de Rham sheaf from \cite[\S3
and Definition 6.8\((iii)\)]{Sch13} (with the correction in \cite[(3)]{Sch16}).
See \cite[D\'efinition 16.10.1]{EGAIV4} for the definition of differentially
smooth morphisms used in \((\ref{thm:gna41db})\) below.
\begin{theorem}[The Deligne--Du Bois complex exists]
  \label{thm:gna41}
  Let \(\Sp\) be a subcategory of a category of reduced spaces of one of the
  types
  \((\ref{setup:introalgebraicspaces})\textnormal{--}(\ref{setup:introadicspaces})\).
  Let \(X\) be a finite-dimensional space in \(\Sp\).
  In cases \((\ref{setup:introalgebraicspaces})\) and
  \((\ref{setup:introformalqschemes})\), suppose that
  \(X\) is quasi-excellent of equal characteristic zero.
  In cases \((\ref{setup:introberkovichspaces})\),
  \((\ref{setup:introrigidanalyticspaces})\), and
  \((\ref{setup:introadicspaces})\), suppose that \(k\) is of characteristic
  zero.
  \begin{enumerate}[label=\((\roman*)\),ref=\roman*]
    \item\label{thm:gna41db}
      Let \(D \subseteq X\) be a closed subspace of \(X\) and set \(U
      \coloneqq X - D\).
      In case \((\ref{setup:introalgebraicspaces})\), suppose that \(X\) is an
      algebraic space over a field \(k\) of characteristic zero and
      that either \(D = \emptyset\) and the regular objects \(Y\)
      in \(\Sp\) of finite
      type over \(X\) are differentially smooth over
      \(k\) with \(\Omega_Y^1\) of finite rank, or \(D \ne \emptyset\) and
      \(X\) is essentially of finite type
      over \(k\).
      In case \((\ref{setup:introformalqschemes})\), suppose that
      \(X\) is locally of finite type over a
      field \(k\) of characteristic zero and \(D = \emptyset\).
      In cases
      \((\ref{setup:introberkovichspaces})\),
      \((\ref{setup:introrigidanalyticspaces})\), and
      \((\ref{setup:introadicspaces})\),
      suppose that \(D = \emptyset\).
      Then, there exists an object \((\underline{\Omega}_X^\bullet\langle D
      \rangle,\FF)\) of \(\Ho\CF_{\kern-1pt\diff\kern-.5pt,\coh}(X)\) satisfying
      the following properties:
      \begin{enumerate}[label=\((\arabic*)\),ref=\arabic*]
        \item\label{thm:gna411} There exists a natural morphism of filtered complexes
          \[
            \bigl(\Omega_{X}^\bullet(\log D),\FF\bigr) \longrightarrow
            \bigl(\underline{\Omega}_X^\bullet\langle D \rangle,\FF\bigr)
          \]
          that is a filtered quasi-isomorphism if \(X\) is regular and \(D\) is
          a normal crossing divisor.
          On the left-hand side, \(\FF\) is the Hodge filtration and coincides
          with the brutal truncation \(\sigma_{\ge \bullet}\)
          when \((X,D)\) is regular.
        \item\label{thm:gna412}
          If \(f\colon (X',U') \to (X,U)\) is a morphism (resp.\ a proper morphism)
          in \(\Sp^2\), then
          \[
            \RR f_*\bigl(\underline{\Omega}_{X'}^\bullet\langle D'
            \rangle,\FF\bigr)
          \]
          is a well-defined object of \(\Ho\CF_{\kern-1pt\diff}(X)\)
          (resp.\ \(\Ho\CF_{\kern-1pt\diff\kern-.5pt,\coh}(X)\)) where \(D' = X' - U'\).
          Moreover, the assignments
          \begin{equation}\label{eq:gna412functors}
            \begin{tikzcd}[cramped,row sep=0,column sep=1.475em]
              \bigl(\Sp^2_{/(X,U)}\bigr)^\op \rar & \Ho\CF_{\kern-1pt\diff}(X)\\
              \bigl(f\colon (X',U') \to (X,U)\bigr) \rar[mapsto]
              & \RR f_*\bigl(\underline{\Omega}_{X'}^\bullet\langle D'\rangle,\FF\bigr)\\
              \bigl(\Sp^2_{\proper/(X,U)}\bigr)^\op \rar & \Ho\CF_{\kern-1pt\diff\kern-.5pt,\coh}(X)\\
              \bigl(f\colon (X',U') \to (X,U)\bigr) \rar[mapsto]
              & \RR f_*\bigl(\underline{\Omega}_{X'}^\bullet\langle D'\rangle,\FF\bigr)
            \end{tikzcd}
          \end{equation}
          are \(\Phi\)-rectified functors satisfying property
          \((\ref{thm:gna251d})\) from Theorem \ref{thm:gna215}.
        \item\label{thm:gna413}
          \emph{(Poincar\'e lemma)}
          Let \(j\colon U \hookrightarrow X\) be the inclusion morphism.
          In case \((\ref{setup:introalgebraicspaces})\),
          if \(X\) is separated and of finite type over a field \(k \subseteq
          \CC\), then
          the complex
          \[
            \bigl(\underline{\Omega}_X^\bullet\langle D \rangle \otimes_{\cO_X}
            \cO_{X \otimes_k \CC}\bigr)^\an
          \]
          is naturally quasi-isomorphic to the complex \(\RR j_*\CC_U\).
          In case \((\ref{setup:introcomplexanalyticgerms})\),
          the complex \(\underline{\Omega}_X^\bullet\langle D \rangle\) is
          naturally quasi-isomorphic to the complex \(\RR j_*\CC_U\).
          In case \((\ref{setup:introadicspaces})\), if \(X\) is
          locally of finite type
          over \(\Spa(k,\cO_k)\) where \(k\) is a discretely valued
          non-Archimedean extension of \(\QQ_p\) and we consider the functors
          \eqref{eq:gna412functors} as landing in the homotopy category of
          filtered complexes of sheaves on the pro-\'etale topology, then the complex
          \[
            \cO\BB^+_{\dR}\otimes_{\cO_{X_\proet}} \underline{\Omega}_{X_\proet}^\bullet
          \]
          is naturally quasi-isomorphic to the complex \(\BB^+_{\dR}\), where
          the tensor product is as sheaves on the pro-\'etale topology.
        \item\label{thm:gna414}
          In case \((\ref{setup:introalgebraicspaces})\),
          if \(X\) is a separated scheme of finite type over
          \(\CC\), then \((\underline{\Omega}_X^\bullet,\FF)^\an\)
          coincides with the Deligne--Du Bois complex constructed by
          Du Bois in \emph{\cite[Notation 3.18]{DB81}}.
        \item\label{thm:gna415}
          We have \(\Gr_\FF^p\underline{\Omega}_X^\bullet\langle D \rangle
          \simeq 0\) if \(p \notin [0,\dim(X)]\).
      \end{enumerate}
    \item\label{thm:gna41ox} In cases
      \((\ref{setup:introalgebraicspaces})\textnormal{--}(\ref{setup:introadicspaces})\),
      there exists an object \(\underline{\Omega}_X^0\) of
      \(\D_{\coh}^b(X)\) satisfying the following properties:
      \begin{enumerate}
        \item[{\((\ref{thm:gna411})^0\)}]
          There exists a natural morphism of complexes \(\cO_X \to
          \underline{\Omega}_X^0\) that is a quasi-isomorphism if \(X\) is
          regular.
        \item[{\((\ref{thm:gna412})^0\)}]
          If \(f\colon X' \to X\) is a morphism (resp.\ a proper morphism)
          in \(\Sp\), then \(\RR f_*\underline{\Omega}_{X'}^0\) is a well-defined
          object of \(\D^b(X)\) (resp.\ \(\D^b_\coh(X)\)).
          Moreover, the assignments
          \begin{equation}\label{eq:gna412functorsox}
            \begin{tikzcd}[cramped,row sep=0,column sep=1.475em]
              \bigl(\Sp_{/X}\bigr)^\op \rar & \D^b(X)\\
              \bigl(f\colon X' \to X\bigr) \rar[mapsto] & \RR
              f_*\underline{\Omega}_{X'}^0\\
              \bigl(\Sp_{\proper/X}\bigr)^\op \rar & \D^b_\coh(X)\\
              \bigl(f\colon X' \to X\bigr) \rar[mapsto] & \RR
              f_*\underline{\Omega}_{X'}^0
            \end{tikzcd}
          \end{equation}
          are \(\Phi\)-rectified functors satisfying property
          \((\ref{thm:gna251d})\) from Theorem \ref{thm:gna215}.
        \item[{\((\ref{thm:gna414})^0\)}]
          If \(X\) is a separated scheme of finite type over
          \(\CC\), then \((\underline{\Omega}_X^0)^\an\) coincides with the 0-th
          graded piece of the Deligne--Du Bois complex constructed by Du Bois in
          \emph{\cite[Notation 3.18]{DB81}}.
      \end{enumerate}
  \end{enumerate}
  Finally, the objects \((\underline{\Omega}_X^\bullet,\FF)\) and
  \(\underline{\Omega}_X^0\) are essentially uniquely determined by
  properties \((\ref{thm:gna411})\) and \((\ref{thm:gna412})\) and
  properties \((\ref{thm:gna411})^0\) and \((\ref{thm:gna412})^0\), respectively.
\end{theorem}
\begin{proof}
  We prove \((\ref{thm:gna41db})\) and \((\ref{thm:gna41ox})\) simultaneously.
  We want to define
  \begin{align*}
    \RR f_*\bigl(\underline{\Omega}^\bullet_{X'}\langle D' \rangle,\FF\bigr)
    &\coloneqq \RR (f \circ
    \varepsilon)_*\Bigl(\underline{\Omega}^\bullet_{X'_\bullet}(\log
    D'_\bullet),\sigma_{\ge\bullet}\Bigr)\\
    \RR f_*\underline{\Omega}_X^\bullet &\coloneqq \RR (f \circ
    \varepsilon)_*\cO_{X_\bullet}
  \end{align*}
  where \(\varepsilon\colon X'_\bullet \to X'\) is a cubical hyperresolution of
  a pair \((X',D')\) over \((X,D)\) for \((\ref{thm:gna41db})\) and a space
  \(X'\) over \(X\) for \((\ref{thm:gna41ox})\), respectively.
  Here, \(f\colon X' \to X\) is the structure morphism for \(X'\) and the
  ``differentially smooth'' assumption is used to make sense of
  \(\Omega_{X'_\bullet}\).
  Theorems \ref{thm:gna215} and \ref{thm:cg39} do not apply directly since \(X\) is not
  necessarily compactifiable.
  However, following the proof of Theorem \ref{thm:gna215}, to prove a version
  of Theorems \ref{thm:gna215} and \ref{thm:cg39} that applies in our setting it
  suffices to show that \eqref{eq:wanttoshow215} in the proof of Theorem
  \ref{thm:gna215} is a quasi-isomorphism.
  For \((\ref{thm:gna41db})\), we want to show that for every cubical
  hyperresolution as above such that \((X',D')\) is a \emph{regular} pair,
  the induced morphism
  \begin{align*}
    \RR f_*\bigl(\Omega^\bullet_{X'}(\log D'),\sigma_{\ge\bullet}\bigr) &\longrightarrow
    \RR (f \circ \varepsilon)_*\bigl(\Omega^\bullet_{X'_\bullet}(\log
    D'_\bullet),\sigma_{\ge\bullet}\bigr)
  \shortintertext{is a quasi-isomorphism.
  For \((\ref{thm:gna41ox})\), we want to show that for every cubical
  hyperresolution as above such that \(X\) is regular,
  the induced morphism}
    \RR f_*\cO_X &\longrightarrow \RR (f \circ \varepsilon)_*\cO_{X_\bullet}
  \end{align*}
  is a quasi-isomorphism.
  Since these questions are \'etale-local on \(X\), we may replace \(X\) by an
  affinoid subdomain.
  In particular, \(X\) is compactifiable, and hence the rest of the proofs of
  Theorems \ref{thm:gna215} and \ref{thm:cg39} apply.\smallskip
  \par We now apply the versions of Theorems \ref{thm:gna215} and \ref{thm:cg39}
  proved in the previous paragraph to show
  \((\ref{thm:gna411})\), \((\ref{thm:gna412})\),
  \((\ref{thm:gna411})^0\), \((\ref{thm:gna412})^0\), and the
  ``essentially uniquely determined'' statements.
  First, we note that
  \(\CF_{\kern-1pt\diff\kern-.5pt,\coh}(X)\), \(\CF_{\kern-1pt\diff}(X)\),
  \(\C^b_\coh(X)\), and \(\C^b(X)\) are
  categories of cohomological descent by \cite[Propositions 1.7.2 and 1.7.6]{GNA02} (see
  Examples
  \ref{ex:cohdescent}\((\ref{ex:cohdescentcochain})\),\((\ref{ex:cohdescentdubois})\))
  and that the functors \eqref{eq:gna412functors} and
  \eqref{eq:gna412functorsox} restricted to regular objects in \(\Sp^2\)
  are \(\Phi\)-rectified since they factor through
  \(\CF_{\kern-1pt\diff\kern-.5pt,\coh}(X)\), \(\CF_{\kern-1pt\diff}(X)\),
  \(\C^b_\coh(X)\), and \(\C^b(X)\), respectively \cite[p.\ 34]{GNA02}.
  It therefore suffices to show that the functors in \eqref{eq:gna412functors} and
  \eqref{eq:gna412functorsox}
  satisfy \((\ref{thm:cg39f1})^2\) and \((\ref{thm:cg39f2})^2\).
  Condition \((\ref{thm:cg39f1})^2\) holds by the additivity of higher direct
  images, and hence it remains to show \((\ref{thm:cg39f2})^2\).
  Condition \((\ref{thm:cg39f2})^2\) holds by definition of structure
  sheaves for \((\ref{thm:gna41ox})\).
  For \((\ref{thm:gna41db})\),
  since the assertion is local, we may replace \(X\) by an affinoid subdomain to
  assume that \(X\) is affinoid.
  Cases \((\ref{setup:introalgebraicspaces})\) and
  \((\ref{setup:introcomplexanalyticgerms})\) now follow from \cite[Propositions
  3.3 and 4.4]{GNA02}.
  Otherwise, we can reduce to the scheme case (a subcase of
  \((\ref{setup:introalgebraicspaces})\)) using the GAGA correspondence
  \citeleft\citen{EGAIII1}\citemid \S5\citepunct
  \citen{Kop74}\citepunct \citen{Poi10}\citemid Annexe A\citepunct
  \citen{Hub07}\citemid \S6\citeright\ since
  sheaves of differentials are compatible with analytification
  \cite[Proposition 3.3.11]{Ber93}.
  \par Next, we show \((\ref{thm:gna413})\).
  The subcase of \((\ref{setup:introalgebraicspaces})\) where \(X\) is separated
  and of finite type over a field \(k \subseteq \CC\) and case
  \((\ref{setup:introcomplexanalyticgerms})\) can be proved in the same way as
  in \cite[Th\'eor\`emes 4.1 and 4.3]{GNA02}.
  It remains to consider case \((\ref{setup:introadicspaces})\).
  The assignments
  \[
    \begin{tikzcd}[cramped,row sep=0,column sep=1.475em]
      \bigl(f\colon X' \to X\bigr) \rar[mapsto]
      & \RR f_*\bigl(\cO\BB^+_\dR \otimes_{\cO_{X_\proet}}
      \underline{\Omega}^\bullet_{X_\proet}\bigr)\\
      \bigl(f\colon X' \to X\bigr) \rar[mapsto]
      & \RR f_*\BB^+_\dR
    \end{tikzcd}
  \]
  are isomorphic \(\Phi\)-rectified functors on
  the subcategory of regular spaces of finite type over \(X\)
  by \cite[Corollary 6.13]{Sch13}.
  They therefore extend to isomorphic functors on \(\Sp_{/X}\), and these
  extension functors are defined by taking a cubical hyperresolution of \(X'\)
  and then computing the pushforward to \(X\) of the sheaves on the right.
  This construction yields the functor \(f \mapsto \RR f_*\BB^+_\dR\) for the
  second functor described above by cohomological descent in the pro-\'etale
  topology (Theorem \ref{thm:gnapgp69}).
  Statement
  \((\ref{thm:gna413})\) now follows by the essential uniqueness statement in
  Theorem \ref{thm:gna215}.
  \par Finally, \((\ref{thm:gna414})\) and \((\ref{thm:gna414})^0\) are proved
  in \cite[Expos\'e V, Corollaire 3.7]{GNAPGP88}, and \((\ref{thm:gna415})\)
  follows from Theorem \ref{thm:gnapgp215}, which says there exists a cubical
  hyperresolution \(X_\bullet\) of \(X\) of type \(\square_r \in \Ob(\Pi)\) such
  that \(\dim(X_\alpha) \le \dim(X) - \abs{\alpha} + 1\) for every \(\alpha \in
  \square_r\).
\end{proof}
\begin{remark}
  Other than in cases \((\ref{setup:introalgebraicspaces})\) and
  \((\ref{setup:introcomplexanalyticgerms})\),
  we have assumed that \(D = \emptyset\) in Theorem
  \ref{thm:gna41}\((\ref{thm:gna41db})\).
  Logarithmic versions of these statements would hold in the other cases
  if one could show the analogue of \cite[Proposition 4.4]{GNA02} in these
  categories.
\end{remark}
As a result, we can make the following definition.
\begin{definition}[The (logarithmic) Deligne--Du Bois complex;
  cf.\ {\citeleft\citen{DB81}\citemid Notation 3.18 and
  Proposition-Notation 6.4\citepunct
  \citen{GNAPGP88}\citemid D\'efinition 3.5 and (3.6)\citepunct 
  \citen{GNA02}\citemid (4.2)\citeright}]\label{def:logderham}
  With notation as in Theorem \ref{thm:gna41}\((\ref{thm:gna41db})\), the
  \textsl{Deligne--Du Bois complex} is \(\underline{\Omega}_X^\bullet\) and the
  \textsl{logarithmic Deligne--Du Bois complex} is
  \(\underline{\Omega}_X^\bullet\langle D \rangle\).
  We set
  \begin{align*}
    \underline{\Omega}_X^{p} &\coloneqq \Gr^p_\FF
    \underline{\Omega}_X^\bullet[p]\\
    \underline{\Omega}_X^{p}\langle D \rangle &\coloneqq \Gr^p_\FF
    \underline{\Omega}_X^\bullet\langle D \rangle[p]
  \end{align*}
  which is a coherent complex functorial in \(X\).
  Note that if \(p = 0\), this notation matches the definition in Theorem
  \ref{thm:gna41}\((\ref{thm:gna41ox})\).
\end{definition}
\begin{remark}\label{rem:somehodge}
  With notation as in Definition \ref{def:logderham},
  by \((\ref{thm:gna415})\) in Theorem \ref{thm:gna41}\((\ref{thm:gna41db})\),
  we see that \(\underline{\Omega}_X^p\langle D \rangle\) is nonzero if and only
  if \(p \in [0,\dim(X)]\).
  Moreover, by our existence result for cubical hyperresolutions (Theorem
  \ref{thm:gnapgp215}) and the spectral sequence associated to a cubical
  hyperresolution \cite[p.\ 39]{GNAPGP88}, we know that
  \(H^i(\underline{\Omega}_X^p\langle D \rangle) = 0\) for all \(i < 0\) or \(i > \dim(X)\)
  and that
  \[
    \dim\Bigl(\Supp\Bigl(H^i\bigl(\underline{\Omega}_X^p\langle D \rangle\bigr)\Bigr)\Bigr)
    \le \dim(X) - i
  \]
  for all \(i \in [0,\dim(X)]\).
  See \citeleft\citen{GNAPGP88}\citemid Expos\'e III, Proposition 1.17 and
  Expos\'e V, (3.6)\citeright.
  We can also define Hodge cohomology and the Hodge-to-de Rham spectral
  sequence in this setting.
  See \cite[(4.2)]{GNA02}.
\end{remark}
We note two useful facts coming from the construction of the Deligne--Du Bois
complex.
\begin{corollary}[Compatibility of the Deligne--Du Bois complex with \'etale
  base change and ground field extensions;
  cf.\ {\cite[(4.3)]{DB81}}]\label{cor:delignedblocal}
  With notation as in Theorem \ref{thm:gna41}, the formation of
  \(\underline{\Omega}_X^\bullet\) and \(\underline{\Omega}_X^\bullet\langle D
  \rangle\) are compatible with \'etale base
  change on \(X\) and the formation of \(\underline{\Omega}_X^0\) is compatible
  with regular base change on \(X\).
  All three objects are compatible with ground field extensions where in case
  \((\ref{setup:introalgebraicspaces})\) (resp.\ 
  \((\ref{setup:introformalqschemes})\)), \(X\) is assumed to be essentially of
  finite type (resp.\ of finite type) 
  over a field \(k\) of characteristic zero.
\end{corollary}
\begin{proof}
  By Theorem \ref{thm:gna41} and its proof,
  all three objects are constructed from taking the derived pushforward of
  the de Rham complex on a cubical hyperresolution \(X_\bullet\) of \(X\).
  For the Deligne--Du Bois complexes, it suffices to note that cubical
  hyperresolutions, the de Rham complex on a regular space, and the logarithmic
  de Rham complex for a regular pair are compatible with pulling back along
  \'etale morphisms.
  For \(\underline{\Omega}_X^0\), it suffices to note that cubical
  hyperresolutions and structure sheaves are compatible with pulling back along
  smooth morphisms.
  For ground field extensions, we use the compatibility of cubical
  hyperresolutions with ground field extensions.
\end{proof}
The following fact says that we can compute the Deligne--Du Bois complex using
regular simplicial resolutions in the sense of \cite[D\'efinition 2.1.1]{DB81}.
\begin{corollary}[Computing the Deligne--Du Bois complex via regular simplicial
  resolutions]\label{cor:simplicial}
  With notation as in Theorem \ref{thm:gna41},
  \(\underline{\Omega}_X^\bullet\), \(\underline{\Omega}_X^\bullet\langle D
  \rangle\), and \(\underline{\Omega}_X^0\) can be computed using regular
  simplicial resolutions.
\end{corollary}
\begin{proof}
  By Theorem \ref{thm:gna41}, it suffices to show that the (filtered) complex
  obtained using regular simplicial resolutions (in the sense of
  \cite[D\'efinition 2.1.1]{DB81}) satisfies properties
  \((\ref{thm:gna411})\) and \((\ref{thm:gna412})\) and
  properties \((\ref{thm:gna411})^0\) and \((\ref{thm:gna412})^0\), respectively.
  \par For regular simplicial resolutions, we first note
  that regular simplicial resolutions exist by using the
  construction in \cite[\S6.2]{Del74} together with the version of resolutions
  of singularities in \cite[Theorems 1.2.1, 5.1.1, and 5.3.2]{Tem12},
  and the fact
  that admissible hypercoverings are of cohomological descent by
  \cite[(5.3.5)(V)]{Del74} (see \cite[Theorem 7.22]{Con04} for a proof).
  One can prove functoriality (\((\ref{thm:gna412})\) and
  \((\ref{thm:gna412})^0\)) as in \cite[(3.2)]{DB81}.
\end{proof}
\subsection{Du Bois singularities}
We can now define Du Bois singularities for all spaces that locally satisfy the
hypotheses in Theorem \ref{thm:gna41}.
Du Bois singularities were originally defined by Steenbrink \cite{Ste83} for
complex varieties.
\begin{definition}[Du Bois singularities; cf.\
  {\cite[(3.5)]{Ste83}}]\label{def:dubois}
  Let \(X\) be a space of one of the types 
  \((\ref{setup:introalgebraicspaces})\textnormal{--}(\ref{setup:introadicspaces})\)
  and consider a point \(x \in X\).
  In cases \((\ref{setup:introalgebraicspaces})\) and
  \((\ref{setup:introformalqschemes})\), suppose that
  \(\cO_{X,x}\) is quasi-excellent of equal characteristic zero.
  In cases \((\ref{setup:introberkovichspaces})\),
  \((\ref{setup:introrigidanalyticspaces})\), and
  \((\ref{setup:introadicspaces})\), suppose that \(k\) is of characteristic
  zero.
  We say that \(X\) \textsl{has Du Bois singularities at \(x\)} or \textsl{is Du
  Bois at \(x\)}
  if \(X\) is reduced at \(x\) and if the natural morphism
  \begin{equation}\label{eq:duboisdef}
    \cO_{X,x} \longrightarrow 
    \underline{\Omega}_{\Spec(\cO_{X,x})}^0
  \end{equation}
  from
  \((\ref{thm:gna411})^0\) in Theorem \ref{thm:gna41}\((\ref{thm:gna41ox})\) is a
  quasi-isomorphism.
  We say that \(X\) \textsl{has Du Bois singularities} or \textsl{is Du Bois}
  if \(X\) is Du Bois at every \(x \in X\).
  \par We note that if \(\underline{\Omega}_X^0\) exists globally on \(X\), then
  \eqref{eq:duboisdef} is a quasi-isomorphism if and only if \(\cO_X \to
  \underline{\Omega}_X^0\) is a quasi-isomorphism at \(x\) by Corollary
  \ref{cor:delignedblocal}.
\end{definition}

\section{Injectivity, torsion-freeness, and vanishing theorems}\label{sect:injnc}
In this section, we prove our injectivity, torsion-freeness,
and vanishing theorems (Theorem \ref{thm:maininj}) uniformly in all categories
\((\ref{setup:introalgebraicspaces})\textnormal{--}(\ref{setup:introadicspaces})\).
The only known previous cases of these results (beyond the smooth or klt
settings) are the generalized normal crossing case for varieties, which is due
to Ambro \cite{Amb03,Amb14,Amb20}, Fujino \cite{Fuj04,Fuj11,Fuj16,Fuj17}, and
Fujino--Fujisawa \cite{FF14}, and
the simple normal crossing case for complex analytic spaces, which is due to Fujino
\cite{Fujvan}.
The corresponding vanishing and injectivity results for klt pairs in the
categories
\((\ref{setup:introalgebraicspaces})\textnormal{--}(\ref{setup:introadicspaces})\)
are due to the author of the present paper \cite{Mur}.
As in our previous work 
\cite{Mur}, the key idea is to approximate proper morphisms of schemes of
equal characteristic zero by proper morphisms of schemes essentially of finite
type over \(\QQ\) and then to deduce the other cases
\((\ref{setup:introalgebraicspaces})\textnormal{--}(\ref{setup:introadicspaces})\)
using GAGA theorems.\medskip
\par A key obstacle when reducing from the generalized normal crossing case to
the case when \(X\) is regular and \(B\) has normal crossing support is that we
do not have sufficiently strong versions of resolutions of singularities of the
form in \citeleft\citen{Kol13}\citemid Proposition 10.59\citepunct
\citen{BVP13}\citeright\ available.
These results of resolutions of singularities hold for varieties.
A version of these resolution statements was shown for complex analytic spaces
by Fujino \cite[Lemma 5.1]{Fujvan}, who used his resolution result to
prove the injectivity, torsion-freeness, and vanishing
theorems for complex analytic spaces \cite{Fujvan}.
Instead, we use the theory we have built up so far in the paper concerning the
\(0\)-th graded piece \(\underline{\Omega}_X^0\) of the
Deligne--Du Bois complex \(\underline{\Omega}_X^\bullet\).
\subsection{Outline}
As in \S\ref{sect:intro}, we follow \cite{Amb20} for the nomenclature for
these results.
In \S\S\ref{sect:wlc}--\ref{sect:gnc},
we review the definitions of \(\omega_{(X,B)}\) and of generalized normal
crossing pairs due to Ambro \cite{Amb18,Amb20}.
We then prove our version of the Esnault--Viehweg injectivity theorem for
schemes (Theorem
\ref{thm:amb2012}).
As consequences, we deduce the Tankeev--Koll\'ar injectivity theorem for
scheems (Theorem
\ref{thm:amb2014gnc}), the Koll\'ar torsion-freeness theorem for schemes (Theorem
\ref{thm:amb2015gnc}), and the Ohsawa--Koll\'ar vanishing theorem for schemes (Theorem
\ref{thm:amb2016gnc}) following the strategy in \cite{Amb20}.
In \S\ref{sect:injectivityothercats}, we extend our injectivity,
torsion-freeness, and vanishing theorems to all
categories
\((\ref{setup:introalgebraicspaces})\textnormal{--}(\ref{setup:introadicspaces})\).
Finally, in Theorem \ref{thm:ma18}, we prove our extension of a result of Ma \cite{Ma18}
to the complete local case, which characterizes derived splinters of equal
characteristic zero in terms of the vanishing condition on maps of Tor.
Theorem \ref{thm:ma18} is an application of the Koll\'ar torsion-freeness
theorem (Theorem \ref{thm:amb2014gnc}).
\begin{remark}
  \par We have not formulated explicit statements of our theorems
  when \((X,B)\) is a normal crossing pair (in the sense of
  \citeleft\citen{Amb03}\citemid Definition 2.3\citepunct \citen{Fuj17}\citemid
  Definition 5.8.4\citeright) when \(B\) has \(\RR\) coefficients.
  These statements follow from statements for \(\QQ\) coefficients by perturbing
  the coefficients of \(B\) \cite[Lemma 5.2.13]{Fuj17}.
\end{remark}
\subsection{\texorpdfstring{\except{toc}{\(\bm{\omega_{(X,B)}}\)}\for{toc}{\(\omega_{(X,B)}\)}}{\unichar{"03C9}(X,B)}
and weakly log canonical pairs}\label{sect:wlc}
We start by defining a version of the canonical sheaf that works well for pairs
\((X,B)\) when \(X\) is not necessarily normal.
\begin{definition}[cf.\ {\cite[\S4]{Amb18}}]\label{def:weaklynormalpair}
  Let \(X\) be a weakly normal, \(S_2\), and locally Noetherian scheme
  satisfying J-1 \cite[(32.B)]{Mat80}.
  Suppose also that \(X\) satisfies N-1 \cite[(31.A)]{Mat80},
  i.e., the normalization
  morphism
  \[
    \pi\colon \bar{X} \longrightarrow X
  \]
  is finite.
  Denote by \(C \subseteq X\) and \(\bar{C} \subseteq \bar{X}\) the
  conductor subschemes.
  Let \(B\) be the closure in \(X\) of a \(\QQ\)-divisor \(B^0\) on \(X_\reg\)
  and let \(\bar{B}\) be the analogous closure in \(\bar{X}\).
  Both \(B\) and \(\bar{B}\) are \(\QQ\)-Weil divisors.
  \par Suppose that \(X\) has a dualizing complex \(\omega_X^\bullet\) with
  associated canonical sheaf \(\omega_X\).
  For every integer \(n\), we define the coherent \(\cO_X\)-module
  \[
    \omega^{[n]}_{(X,B)} \subseteq \omega_{\sK\kern-0.5pt(X)}^{\otimes n}
  \]
  as follows: For every open subset \(U \subseteq X\), we set
  \[
    \Gamma\Bigl(U,\omega^{[n]}_{(X,B)}\Bigr) \coloneqq \Set*{\omega \in
      \omega_{\sK\kern-0.5pt(X)}^{\otimes n} \given
      \!\parbox{0.63\textwidth}{\begin{enumerate}[label=\((\alph*)\),leftmargin=*,nosep]
        \item \((\pi^*\omega) + n(\bar{C}+\bar{B}) \ge 0\) on
          \(\pi^{-1}(U)\).
        \item For every irreducible component \(P\) of \(C \cap U\), there exists
          \(\eta \in \omega_{\smash{\sK\kern-0.5pt(P)}}^{\otimes n}\)
          such that \(\Res_Q \pi^*\omega = \pi^*\eta\) for every irreducible
          component \(Q\) of \(\bar{C} \cap \pi^{-1}(U)\) lying over \(P\).
    \end{enumerate}}}.
  \]
\end{definition}
We can now define weakly normal and weakly log canonical pairs.
\begin{definition}[{cf.\ \cite[Definition 4.6 and \S4.1]{Amb18}}]
  A \textsl{weakly normal pair} \((X,B)\) consists of the following data:
  \begin{enumerate}[label=\((\roman*)\)]
    \item \(X\) a weakly normal, \(S_2\), and locally Noetherian scheme
      satisfying J-1 and N-1 with a dualizing complex \(\omega_X^\bullet\); and
    \item \(B\) the closure of a \(\QQ\)-divisor \(B^0\) on \(X_\reg\)
      satisfying the following property:
      There exists an integer \(r \ge 1\) such that \(rB\) has integer
      coefficients and \(\omega^{[r]}_{(X,B)}\) is invertible.
  \end{enumerate}
  \par With notation as in Definition \ref{def:weaklynormalpair}, we say that a
  pair \((X,B)\) is \textsl{weakly log canonical} if it is weakly normal and
  \((\bar{X},\bar{C}+\bar{B})\) is log canonical.
\end{definition}

\subsection{Generalized normal crossing pairs}\label{sect:gnc}
We define generalized normal crossing pairs
following Ambro \cite[\S3]{Amb20}.
This notion generalizes Ambro's previous notion of normal crossing pairs
defined in
\cite[Definition 2.3]{Amb03} (see also \cite[Definition 5.8.4]{Fuj17}) and a
different notion of generalized normal crossing pairs due to Kawamata
\cite[\S4]{Kaw85}.\medskip
\par We start by defining the local models for generalized normal crossing pairs.
\begin{citeddef}[{\cite[Definition 4]{Amb20}}]
  Let \(k\) be a field.
  For every subset \(F \subseteq \{1,2,\ldots,N\}\), set
  \[
    \bA_{k,F} \coloneqq \bigcap_{i \in \{1,2,\ldots,N\}} \Set[\big]{z \in
    \bA^N_k \given z_i = 0} \subseteq \bA^N_k.
  \]
  A \textsl{generalized normal crossing local model} is a pair
  \((X,B)\) consisting of the following:
  \begin{enumerate}[label=\((\roman*)\)]
    \item \(X = \bigcup_F \bA_{k,F} \subseteq \bA_k^N\) where the union is
      indexed by a finite set of subsets \(F \subseteq \{1,2,\ldots,N\}\) called
      \textsl{facets}, not contained in one another.
      We assume \(X\) satisfies \(S_2\), which by \cite[Corollary 3.11]{Amb18}
      is equivalent to the condition that for every pair of facts \(F \ne F'\),
      there exists a chain of facets \(F = F_0,F_1,\ldots,F_l = F'\)
      such that for every \(0 \le i < l\), the intersection \(F_i \cap
      F_{i+1}\) contains \(F \cap F'\) and \(F_i \cap F_{i+1}\) has codimension
      1 in both \(F_i\) and \(F_{i+1}\).
    \item Set \(\sigma \coloneqq \bigcap_F F\).
      If \(\sigma \preceq \tau \preceq F\) and \(\tau\) has codimension 1 in
      \(F\), then there exists a facet \(F'\) such that \(\tau = F \cap F'\).
    \item \(B = (\sum_{i \in \sigma} b_iH_i)\rvert_X\), where \(b_i \in \QQ \cap
      [0,1]\) and \(H_i = \Set{z \in \bA^N_k \given z_i = 0}\).
      We may write
      \[
        B = \sum_F \sum_{i \in \sigma} b_i\,\bA_{k,F - \{i\}}.
      \]
  \end{enumerate}
\end{citeddef}
We now define generalized normal crossing pairs.
\begin{definition}[{cf.\ \cite[Definition 8]{Amb20}}]
  Let \((X,B)\) be a weakly log canonical pair of equal characteristic.
  We say that \((X,B)\) is a \textsl{generalized normal crossing pair} if,
  for every closed point \(x \in X\), there exists a generalized normal
  crossing local model \((X',B')\) over the residue field \(\kappa(x)\)
  and an isomorphism of complete local \(\kappa(x)\)-algebras
  \[
    \hat{\cO}_{X,x} \simeq \hat{\cO}_{X',0}
  \]
  such that
  \(\Bigl(\omega^{[r]}_{(X,B)}\Bigr)^\wedge_x\) corresponds to
  \(\Bigl(\omega^{[r]}_{(X',B')}\Bigr)^\wedge_0\) for all sufficiently divisible \(r\).
\end{definition}
Note that since generalized normal crossing pairs are defined using completions,
they will automatically be compatible with GAGA.
\subsection{The Esnault--Viehweg injectivity theorem}
We start with our version of the Esnault--Viehweg injectivity theorem
\cite[Theorem 5.1\((a)\)]{EV92}.
Because we do not know how dualizing complexes \(\omega_X^\bullet\) or canonical
divisors \(K_X\) behave under limits of schemes,
we formulate our results in terms of local cohomology instead of in terms of
higher direct images.
\par Compared to our previous work
\cite{Mur}, the approximation process is more subtle
because of the presence of boundary divisors, making it difficult to pass to the
Zariski--Riemann space \(\ZR(X)\) as we did in \cite{Mur}.
Instead of passing to the Zariski--Riemann space \(\ZR(X)\), we work one element
of local cohomology at a time to be able to construct appropriate boundary
divisors on the approximations we work with.\medskip
\par We first show our version of the Esnault--Viehweg injectivity theorem for
reduced simple normal crossing pairs.
\begin{theorem}[The Esnault--Viehweg injectivity theorem for reduced simple
  normal crossing pairs on regular schemes]\label{thm:evinjsncred}
  Let \(f\colon (X,B) \to Y\) be a proper morphism of locally Noetherian schemes
  of equal characteristic zero such that \(X\) is regular and \(B\) is a simple
  normal crossing divisor.
  Let \(D\) be an effective Cartier divisor with support in \(\Supp(B)\).
  Then, for every point \(y \in Y\), the natural maps
  \[
    H^i_{f_y^{-1}(y)}\bigl(X_y,\cO_{X_y}(-B_y-D_y)\bigr)
    \longrightarrow
    H^i_{f_y^{-1}(y)}\bigl(X_y,\cO_{X_y}(-B_y)\bigr)
  \]
  are surjective for every \(i\), where the subscripts \((\,\cdot\,)_y\)
  denote base change to \(\Spec(\cO_{Y,y})\).
\end{theorem}
\begin{proof}
  We first make some reductions.
  Working one connected component of \(X\) at a time, we may assume that \(X\)
  is integral.
  Replacing \(Y\) by the image of \(X\), we may assume that \(Y\) is integral
  and that \(f\) is surjective.
  Working one point \(y \in Y\) at a time, we may replace \(Y\) with
  \(\Spec(\cO_{Y,y})\) to assume that \(Y = \Spec(R)\) for an integral
  Noetherian local \(\QQ\)-algebra \((R,\fm)\).
  By flat base change for local cohomology \cite[Theorem 6.10]{HO08},
  we may replace \(R\) by \(\hat{R}\) to assume that \(R\) is
  excellent by applying \cite[Lemme 7.9.3.1]{EGAIV2} to each stratum of
  \((X,B)\).
  In this case, since \(X\) is of finite type over \(R\), we know that \(X\) is
  also excellent \cite[Scholie 7.8.3\((ii)\)]{EGAIV2}.
  Moreover, \(R\) has a dualizing complex \(\omega_R^\bullet\) and
  \(\omega_X^\bullet \coloneqq f^!\omega_R^\bullet\) is a dualizing complex on
  \(X\) \cite[p.\ 299]{Har66}.\smallskip
  \begin{step}
    Approximating \(f\) by morphisms over rings essentially of finite type over
    \(\QQ\).
  \end{step}
  Write \(R\) as the direct limit
  \[
    R \simeq \varinjlim_{\lambda \in \Lambda} R_\lambda
  \]
  of its sub-\(\QQ\)-algebras of finite type.
  By \cite[Lemma 4.2]{Mur},
  we can construct a Cartesian diagram
  \begin{equation}\label{eq:mur42diag}
    \begin{tikzcd}
      X \rar{f}\dar[swap]{v_\lambda} & \Spec(R)\dar\\
      X_\lambda \rar{f_\lambda} & \Spec(R_\lambda)
    \end{tikzcd}
  \end{equation}
  where
  the \(f_\lambda\) are proper and surjective.
  We replace \(R_\lambda\) with their localizations at \(\fm \cap
  R_\lambda\) to assume that the \(R_\lambda\) are local with maximal ideals
  \(\fm_\lambda\).
  This replacement does not affect the inverse limits of the inverse systems
  in \eqref{eq:mur42diag} since limits of the inverse systems satisfy the same
  universal property.\smallskip
  \begin{step}
    Approximating \(B\) and \(D\) by objects on the \(X_\lambda\).
  \end{step}
  By \cite[\href{https://stacks.math.columbia.edu/tag/0B8W}{Tag
  0B8W}(3)]{stacks-project}, there exist \(\alpha \in \Lambda\) and
  coherent ideal sheaves \(\cI_\alpha,\cJ_\alpha,\cJ_\alpha^0 \subseteq
  \cO_{X_\alpha}\) such that
  \begin{alignat*}{2}
    \cO_X(-B) &\simeq{}& v_\alpha^{-1}\cI_\alpha \cdot \cO_X,\\
    \cO_X(-D) &\simeq{}& v_\alpha^{-1}\cJ_\alpha \cdot \cO_X,\\
    \cO_X(-D_\red) &\simeq{}& v_\alpha^{-1}\cJ^0_\alpha \cdot \cO_X.
  \end{alignat*}
  We can then replace the schemes \(X_\lambda\)
  in the bottom row of \eqref{eq:mur42diag}
  by their normalized blowups along
  \(v_{\alpha\lambda}^{-1}\cI_\alpha\cJ_\alpha\cJ_\alpha^0\cdot\cO_{X_\lambda}\)
  to assume that
  \(\cI_\alpha\), \(\cJ_\alpha\), and \(\cJ^0_\alpha\)
  are invertible and that \(X_\lambda\) is
  normal.
  This replacement does not change the inverse limit of the inverse
  system 
  \[
    \Set[\big]{f_\lambda\colon X_\lambda \longrightarrow
    \Spec(R_\lambda)}_{\lambda \in \Lambda}
  \]
  in \eqref{eq:mur42diag} by the universal property of blowing up
  \cite[\href{https://stacks.math.columbia.edu/tag/0806}{Tag
  0806}]{stacks-project} and the universal property of normalization
  \cite[\href{https://stacks.math.columbia.edu/tag/035Q}{Tag
  035Q}(4)]{stacks-project}.
  Let \(B_\alpha\), \(D_\alpha\), and \(D^0_\alpha\) be the effective Cartier divisors
  corresponding to \(\cI_\alpha\), \(\cJ_\alpha\), and \(\cJ^0_\alpha\), respectively.
  As a result, we can write the injection \(\cO_X(-B-D) \hookrightarrow
  \cO_X(-B)\) as the pullback of the injections
  \[
    v_\lambda^*\bigl(\cO_{X_\lambda}(-B_\lambda
    -D_\lambda)\bigr) \hooklongrightarrow
    v_\lambda^*\bigl(\cO_{X_\lambda}(-B_\lambda)\bigr)
  \]
  where we denote
  \(B_\lambda \coloneqq v_{\alpha\lambda}^*B_\alpha\) and
  \(D_\lambda \coloneqq v_{\alpha\lambda}^*D_\alpha\), and since \(\Supp(D) =
  \Supp(D_\red)\), we have
  \[
    \Supp\bigl(v_{\alpha\lambda}^*D_\alpha\bigr)
    \subseteq \Supp\bigl(v_{\alpha\lambda}^*B_\alpha\bigr) 
  \]
  for all \(\lambda \ge \alpha\).
  \par For the remainder of the proof, we consider the commutative diagram
  \begin{equation}\label{eq:grothendiecklimitzr}
    \begin{tikzcd}
      H^i_{f^{-1}(\{\fm\})}\bigl(X,\cO_X(-B-D)\bigr) \rar &
      H^i_{f^{-1}(\{\fm\})}\bigl(X,\cO_X(-B)\bigr)\\
      H^i_{f_\lambda^{-1}(\{\fm_\lambda\})}\bigl(X_\lambda,
      \cO_{X_\lambda}(-B_\lambda-D_\lambda)\bigr) \rar \uar & \uar
      H^i_{f_\lambda^{-1}(\{\fm_\lambda\})}\bigl(X_\lambda,
      \cO_{X_\lambda}(-B_\lambda)\bigr)
      \mathrlap{.}
    \end{tikzcd}
  \end{equation}
  Let
  \[
    \eta \in H^i_{f^{-1}(\{\fm\})}\bigl(X,\cO_X(-B)\bigr)
  \]
  be an arbitrary element.
  We want to show that \(\eta\) is in the image of the top horizontal map in
  \eqref{eq:grothendiecklimitzr}.
  By the Grothendieck limit theorem for local cohomology \cite[Theorem
  3.13]{Mur}, the direct limit of the maps in the bottom row of
  \eqref{eq:grothendiecklimitzr} is the map in the top row of
  \eqref{eq:grothendiecklimitzr}.
  Thus, there exists \(\lambda \in \Lambda\) such that \(\eta\) is the image of
  \[
    \eta_\lambda \in
    H^i_{f_\lambda^{-1}(\{\fm_\lambda\})}\bigl(X_\lambda,\cO_{X_\lambda}
    (-B_\lambda)\bigr)
  \]
  under the bottom horizontal map in \eqref{eq:grothendiecklimitzr}.
  \begin{step}
    Finding a preimage for \(\eta_\lambda\) after pulling \(\eta_\lambda\) back
    to a log resolution
    \(W_{\lambda,p}\) of \((X_\lambda,B_\lambda)\).
  \end{step}
  \par Choose a log resolution of singularities
  \[
    g_{\lambda,p}\colon W_{\lambda,p} \longrightarrow X_{\lambda}
  \]
  such that \(\Exc(g_{\lambda,p}) \cup
  g_{\lambda,p\,*}^{-1}B_{\lambda,p}\) has simple normal crossing support.
  Such a log resolution exists by \cite[Chapter I, \S3, Main theorems
  I\((n)\) and II\((N)\)]{Hir64}.
  We then set
  \[
    B_{\lambda,p} \coloneqq g_{\lambda,p\,*}^{-1}B_{\lambda}+E_{\lambda,p}
    \qquad \text{and} \qquad
    D_{\lambda,p} \coloneqq
    g_{\lambda,p}^{*}D_{\lambda}
  \]
  where \(E_{\lambda,p}\)
  is the sum of \(g_{\lambda,p}\)-exceptional divisors that appear in
  \(g_{\lambda,p}^{*}B_{\lambda}\), all with coefficient \(1\).
  Note that \(B_{\lambda,p}\) is an effective divisor with simple normal crossing
  support on \(W_{\lambda,p}\) by the choice of the \(W_{\lambda,p}\) as a log
  resolution.
  Moreover,
  we have \(\Supp(D_{\lambda,p}) \subseteq \Supp(B_{\lambda,p})\) because we
  have added \(E_{\lambda,p}\) to \(g_{\lambda,p\,*}^{-1}B_{\lambda}\) in
  \(B_{\lambda,p}\).
  We also choose a log resolution \(g\colon W \to X\) of \((X,B)\)
  fitting into the commutative diagram below:
  \begin{equation}\label{eq:ulambdap}
    \begin{tikzcd}
      W \rar{g}\dar[swap]{u_{\lambda,p}} & X\dar{v_\lambda}\\
      W_{\lambda,p} \rar{g_{\lambda,p}} & X_\lambda\mathrlap{.}
    \end{tikzcd}
  \end{equation}
  We can find such a resolution by taking a log resolution factoring through the
  closure of the unique irreducible component of \(W_{\lambda,p}
  \times_{X_\lambda} X\) dominating \(W_{\lambda,p}\).
  We then have
  \begin{equation}\label{eq:computeexceptional}
    u_{\lambda,p}^*B_{\lambda,p}
    \le u_{\lambda,p}^*g_{\lambda,p}^* B_\lambda
    = g^*v_\lambda^*B_\lambda = g^*B.
  \end{equation}
  Since all morphisms in \eqref{eq:ulambdap} are dominant morphisms between
  integral schemes, we see that
  all coefficients in \(B_{\lambda,p}\) must be equal to \(1\).
  \par We now pull back the map on local cohomology
  to each \(W_{\lambda,p}\) and extend the commutative
  diagram \eqref{eq:grothendiecklimitzr} as follows:
  \begin{equation}\label{eq:largergrothendiecklimitzr}
    \mathclap{\begin{tikzcd}[ampersand
      replacement=\&,baseline=(midarrow.base)]
      H^i_{f^{-1}(\{\fm\})}\bigl(X,\cO_X(-B-D)\bigr) \rar
      \& H^i_{f^{-1}(\{\fm\})}\bigl(X,\cO_X(-B)\bigr)\\
      H^i_{f_\lambda^{-1}(\{\fm_\lambda\})}\bigl(X_\lambda,
      \cO_{X_\lambda}(-B_\lambda-D_\lambda)\bigr) \rar \uar\dar 
      \&|[alias=midarrow]| \uar\dar
      H^i_{f_\lambda^{-1}(\{\fm_\lambda\})}\bigl(X_\lambda,
      \cO_{X_\lambda}(-B_\lambda)\bigr)\\
      H^i_{(f_\lambda \circ g_{\lambda,p})^{-1}(\{\fm_\lambda\})}\bigl(W_{\lambda,p},
      \cO_{W_{\lambda,p}}(- B_{\lambda,p} - D_{\lambda,p})\bigr)
      \rar[twoheadrightarrow]
      \&
      H^i_{(f_\lambda \circ g_{\lambda,p})^{-1}(\{\fm_\lambda\})}\bigl(
      W_{\lambda,p},\cO_{W_{\lambda,p}}(- B_{\lambda,p})\bigr)\mathrlap{.}
    \end{tikzcd}}
  \end{equation}
  The bottom horizontal morphism is surjective by the 
  the local-global dual \citeleft\citen{Lip78}\citemid
  Theorem on p.\ 188\citeright\ (see \citeleft\citen{Mur}\citemid Lemma
  5.2\citeright) of Fujino's relative Hodge-theoretic injectivity theorem
  \cite[Theorem 5.5.1]{Fuj17}.
  We can therefore find
  \[
    \tilde{\eta}_{\lambda,p} \in
    H^i_{(f_\lambda \circ g_{\lambda,p})^{-1}(\{\fm_\lambda\})}\bigl(W_{\lambda,p},
    \cO_{W_{\lambda,p}}(- B_{\lambda,p} - D_{\lambda,p})\bigr)
  \]
  mapping to the image of \(\eta_\lambda\) in the bottom right term of the
  diagram \eqref{eq:largergrothendiecklimitzr} under the bottom horizontal
  morphism in \eqref{eq:largergrothendiecklimitzr}.\smallskip
  \begin{step}
    Conclusion of proof.
  \end{step}
  We first note that the difference \(E \coloneqq g^*B -
  u_{\lambda,p}^*B_{\lambda,p}\) is an effective Cartier divisor by
  \eqref{eq:computeexceptional}.
  Moreover, \(E\) is a sum of
  \(g\)-exceptional divisors all with coefficient \(1\) because \(B\) is a
  simple normal crossing divisor.
  We then claim that the natural pullback morphism
  \[
    \cO_X \longrightarrow \RR g_*\cO_W(E)
  \]
  has a retraction in the derived category \(\D^b_\coh(X)\).
  Since \(X\) is regular and \(E\) is a sum of \(g\)-exceptional divisors all
  with coefficient \(1\), we know that \(E \le K_{W/X}\), where \(K_{W/X}\) is
  defined as the nonzero cohomology sheaf of lowest degree of
  the relative dualizing sheaf \(g^!\cO_X\).
  The composition
  \begin{equation}\label{eq:derivedpullbackretraction}
    \cO_X \longrightarrow \RR g_*\cO_W(E) \longrightarrow \RR g_*g^!\cO_X
    \longrightarrow \cO_X
  \end{equation}
  is the identity, where the last morphism is the Grothendieck trace for \(g\)
  \cite[Chapter VII, Corollary 3.4\((b)\)]{Har66}, since this composition
  coincides with the identity after restricting to the isomorphism locus of
  \(g\) and using the reflexivity of \(\cO_X\).
  \par Now since \(\cO_X \to \RR g_*\cO_W(E)\) has a retraction, tensoring
  everything by \(\cO_X(-B)\), we see that
  \[
    \cO_X(-B) \longrightarrow \RR g_*\cO_W(-g^*B+E) = \RR
    g_*\cO_W(-u_{\lambda,p}^*B_{\lambda,p})
  \]
  has a retraction, and similarly after further tensoring everything by
  \(\cO_X(-D)\).
  By definition of the retraction in \eqref{eq:derivedpullbackretraction}, this
  retraction is natural with respect to the inclusion \(\cO_X(-B-D)
  \hookrightarrow \cO_X(-B)\).
  Thus, we obtain the diagram
  \[
    \begin{tikzcd}[column sep=1.475em,ampersand replacement=\&]
      H^i_{(f \circ g)^{-1}(\{\fm\})}\bigl(W,
      \cO_{W}(-u_{\lambda,p}^*B_{\lambda,p}-g^*D)
      \bigr) \rar \ar[d,bend right=40,xshift=-0.75em,start anchor=south,end
      anchor=north,dashed,twoheadrightarrow] \&
      H^i_{(f \circ g)^{-1}(\{\fm\})}\bigl(W,
      \cO_{W}(-u_{\lambda,p}^*B_{\lambda,p}) \bigr)
      \ar[d,bend left=40,xshift=0.75em,start anchor=south,end
      anchor=north,dashed,twoheadrightarrow]\\
      H^i_{f^{-1}(\{\fm\})}\bigl(X,\cO_X(-B-D)\bigr) \rar\ar[u,hook] \&
      H^i_{f^{-1}(\{\fm\})}\bigl(X,\cO_X(-B)\bigr) \ar[u,hook']\\
      H^i_{f_\lambda^{-1}(\{\fm_\lambda\})}\bigl(X_\lambda,
      \cO_{X_\lambda}(-B_\lambda-D_\lambda)\bigr) \rar \uar\dar 
      \& \uar\dar
      H^i_{f_\lambda^{-1}(\{\fm_\lambda\})}\bigl(X_\lambda,
      \cO_{X_\lambda}(-B_\lambda)\bigr)\\
      H^i_{(f_\lambda \circ g_{\lambda,p})^{-1}(\{\fm_\lambda\})}\bigl(W_{\lambda,p},
      \cO_{W_{\lambda,p}}(- B_{\lambda,p} - D_{\lambda,p})\bigr)
      \rar[twoheadrightarrow]
      \ar[uuu,bend left=40,xshift=-8.25em,start anchor=north,end
      anchor=south]
      \&
      H^i_{(f_\lambda \circ g_{\lambda,p})^{-1}(\{\fm_\lambda\})}\bigl(
      W_{\lambda,p},\cO_{W_{\lambda,p}}(- B_{\lambda,p})\bigr)
      \ar[uuu,bend right=40,xshift=7em,start anchor=north,end
      anchor=south]
    \end{tikzcd}
  \]
  where the curved dashed arrows are retractions for the morphisms the curved
  dashed arrows are next to.
  By the commutativity of the diagram, \(\eta\) is the image of
  the image of \(\tilde{\eta}_{\lambda,p}\) in
  \(H^i_{f^{-1}(\{\fm\})}(X,\cO_X(-B-D))\).
\end{proof}
We can now use cyclic covers and resolutions of singularities
to reduce the normal crossing case to the
\emph{reduced} simple normal crossing case.
\begin{theorem}[The Esnault--Viehweg injectivity theorem for
  normal crossing pairs on regular schemes]\label{thm:evinjnc}
  Let \(f\colon (X,B) \to Y\) be a proper morphism of locally Noetherian schemes
  of equal characteristic zero such that \(X\) is regular and \(B\) is an
  \(\RR\)-divisor with normal crossing support and with coefficients in
  \([0,1]\).
  Let \(L\) be a Cartier divisor on \(X\) such that \(L \sim_\RR B\) and
  let \(D\) be an effective Cartier divisor with support in \(\Supp(B)\).
  Then, for every point \(y \in Y\), the natural maps
  \[
    H^i_{f_y^{-1}(y)}\bigl(X_y,\cO_{X_y}(-L_y-D_y)\bigr)
    \longrightarrow
    H^i_{f_y^{-1}(y)}\bigl(X_y,\cO_{X_y}(-L_y)\bigr)
  \]
  are surjective for every \(i\), where the subscripts \((\,\cdot\,)_y\)
  denote base change to \(\Spec(\cO_{Y,y})\).
\end{theorem}
\begin{proof}
  As in the proof of Theorem \ref{thm:evinjsncred}, we may assume that \(X\) is
  integral and \(Y = \Spec(R)\) for an excellent integral local \(\QQ\)-algebra
  \((R,\fm)\).
  By perturbing the coefficients of \(B\)
  \cite[Lemma 5.2.13]{Fuj17}, we can replace \(\RR\) coefficients with \(\QQ\)
  coefficients.\smallskip
  \setcounter{step}{0}
  \begin{step}\label{step:evinjsncproj}
    The case when \(f\) is projective and \(B\) has simple normal crossing support.
  \end{step}
  It suffices to reduce to the case when \(B\) is reduced.
  We need the following variant of the version of the covering lemma we showed
  in \cite[Lemma 8.1 and Claim 8.2.1]{Mur}.
  \begin{claim}\label{claim:cycliccovers}
    Let \(X\) be an integral regular scheme projective over an integral
    Noetherian local \(\QQ\)-algebra \((R,\fm)\).
    Let \(M\) be a \(\QQ\)-Cartier divisor on \(X\) and let \(B\) be a simple
    normal crossing divisor on \(X\).
    Let \(L\) be a Cartier divisor on \(X\) such that
    \[
      L \sim_\QQ M + \sum_{j=1}^s a_j D_j + \sum_{j=s+1}^r D_j,
    \]
    where the \(D_j\) are regular (possibly disconnected)
    divisors, \(\sum_j D_j\) is a simple normal crossing divisor, and the
    \(a_j\) are rational numbers in \([0,1)\).
    Then, there exist a finite surjective morphism \(p \colon W \to X\) from a
    regular integral scheme \(W\) and a divisor \(M_W\) on \(W\) such that \(M_W
    \sim_\QQ p^*M\) and such that \(\cO_X(-L)\) is a direct summand of
    \(p_*\cO_W(-M_W-\sum_{j=1}^s D'_j)\) where \(D'_j = (p^*D_j)_\red\) are
    regular and
    \[
      \sum_{j=1}^s D'_j + \sum_{j=s+1}^r D_j
    \]
    is a simple normal crossing divisor.
  \end{claim}
  \begin{proof}[Proof of Claim \ref{claim:cycliccovers}]
    We proceed by induction on \(s\).
    If \(s = 0\), setting \(p = \id_X\) works.
    Suppose \(s > 0\).
    Write \(a_s = b/m\), where \(m\) is a positive integer.
    By \cite[Lemma 8.1]{Mur} applied to \(\sum_jD_j\), \(N_s = m\), and \(N_j =
    1\) for all \(j \ne s\), there exists a surjective morphism \(p_1\colon
    X_1 \to X\) such that \(p_1^*D_s \sim mD'\) for some divisor \(D'\) on
    \(X_1\).
    Moreover, \(D'\) and \(p_1^*D_j\) for \(j \ne s\) are regular and
    \(D' + \sum_{j \ne s} p_1^*D_j\) is a simple normal crossing
    divisor.
    By \cite[Theorem 2.64, Step 1]{KM98}, the canonical morphism \(\cO_X \to
    p_{1*}\cO_{X_1}\) splits, and hence \(\cO_X(-L) \to
    p_{1*}\cO_{X_1}(-p_1^*L)\) also splits.
    \par Now \(p_1^*D_s\) corresponds to a section of \(\cO_{X_1}(mD')\), and
    hence we can take the associated \(m\)-th cyclic cover \(p_2\colon X_2 \to
    X_1\) as in \cite[Definition 2.50]{KM98}.
    Then, \cite[Lemma 2.51]{KM98} implies that \(X_2\) is regular, the
    \(p_2^*p_1^*D_j\) are regular, and \(\sum_{j \ne s} p_2^*p_1^*D_j\) is a
    simple normal crossing divisor.
    We have the decompositions
    \begingroup
    \allowdisplaybreaks
    \begin{align*}
      p_{2*}\cO_{X_2} &= \bigoplus_{\ell = 0}^{m-1}
      \cO_{X_1}(-\ell D')\\
      p_{2*}\cO_{X_2}(-p_2^*p_1^*L+(b-1)\,p_2^*D') &= \bigoplus_{\ell = 0}^{m-1}
      \cO_{X_1}\bigl(-p_1^*L+(b-\ell-1)D'\bigr).
    \end{align*}
    \endgroup
    The \(\ell = b-1\) summand shows that \(\cO_{X_1}(-p_1^*L)\) is a direct
    summand of
    \[
      p_{2*}\cO_{X_2}\bigl(-p_2^*p_1^*L+(b-1)p_2^*D'\bigr).
    \]
    \par We now have the \(\QQ\)-linear equivalence
    \[
      p_2^*p_1^*L - (b-1)p_2^*D' \sim_\QQ p_2^*p_1^*M + \sum_{j=1}^{s-1} a_j\,
      p_2^*p_1^*D_j + p_2^*D' + \sum_{j=s}^r p_2^*p_1^*D_j,
    \]
    which satisfies the hypotheses of Claim \ref{claim:cycliccovers}.
    By the inductive hypothesis, there exists a finite surjective morphism \(W
    \to X_2\) satisfying the conclusion of Claim \ref{claim:cycliccovers} for
    \(X_2\).
    The composition \(W \to X_2 \to X\) then satisfies the conclusion of Claim
    \ref{claim:cycliccovers} for \(X\).
  \end{proof}
  We now return to the proof of the special case in Step
  \ref{step:evinjsncproj}.
  Write
  \[
    L \sim_\QQ \{B\} + \lfloor B \rfloor
  \]
  and let \(p\colon W \to X\) be the finite surjective morphism constructed in
  Claim \ref{claim:cycliccovers}.
  We then obtain the commutative diagram
  \[
    \begin{tikzcd}
      H^i_{f^{-1}(\{\fm\})}\bigl(X,\cO_X(-L-D)\bigr)\dar \rar[hook]
      & H^i_{(f \circ p)^{-1}(\{\fm\})}\Bigl(W,\cO_W\bigl(
      -\bigl(p^*\{B\}\bigr)_\red-\lfloor B \rfloor - p^*D\bigr)\Bigr)
      \dar[twoheadrightarrow]
      \ar[l,twoheadrightarrow,bend right=30,start anchor=west,end
      anchor=east,yshift=0.75em,dashed]\\
      H^i_{f^{-1}(\{\fm\})}\bigl(X,\cO_X(-L)\bigr) \rar[hook]
      & H^i_{(f \circ p)^{-1}(\{\fm\})}\Bigl(W,\cO_W\bigl(
      -\bigl(p^*\{B\}\bigr)_\red-\lfloor B \rfloor \bigr)\Bigr)
      \ar[l,twoheadrightarrow,bend right=30,start anchor=west,end
      anchor=east,yshift=0.75em,dashed]
    \end{tikzcd}
  \]
  where the curved dashed arrows are retractions for the horizontal injections.
  Since the right vertical map is surjective by Theorem \ref{thm:evinjsncred},
  we are done.\smallskip
  \begin{step}
    The general case.
  \end{step}
  By Chow's lemma \cite[Corollaire 5.6.2]{EGAII} and \cite[Chapter I, \S3, Main
  theorems I\((n)\) and II\((N)\)]{Hir64}, we can find a projective log
  resolution \(g\colon \tilde{X} \to X\) for the pair
  \((X,B)\) such that \(f \circ g\) is projective and \(\Exc(g) \cup
  g_*^{-1}B\) has simple normal crossing support.
  We then obtain the commutative diagram
  \[
    \begin{tikzcd}
      H^i_{f^{-1}(\{\fm\})}\bigl(X,\cO_X(-L-D)\bigr)\dar \rar{\sim}
      & H^i_{(f \circ g)^{-1}(\{\fm\})}\bigl(\tilde{X},\cO_{\tilde{X}}(
      -g^*L-g^*D)\bigr)
      \dar[twoheadrightarrow]\\
      H^i_{f^{-1}(\{\fm\})}\bigl(X,\cO_X(-L)\bigr) \rar{\sim}
      & H^i_{(f \circ g)^{-1}(\{\fm\})}\bigl(\tilde{X},\cO_{\tilde{X}}(
      -g^*L)\bigr)
    \end{tikzcd}
  \]
  where both horizontal maps are isomorphisms since regular \(\QQ\)-algebras
  are pseudo-rational \cite[\S4]{LT81}, and hence \(\cO_X \to \RR
  g_*\cO_{\tilde{X}}\) is a quasi-isomorphism \cite[Theorem 6.2]{Mur}.
  The right vertical map is surjective by Step \ref{step:evinjsncproj} since
  \(g^*L \sim_\QQ g^*B\) and \(g^*B\) is a simple normal crossings divisor.
\end{proof}
We now extend our Esnault--Viehweg injectivity theorem \ref{thm:evinjnc} to the
generalized normal crossing case.
We will use \emph{simplicial} schemes in the proof below.
See \cite[(1.2.1), Proposition 1.3.7, and (2.1.6)]{Gui87} for the
relationship between cubical and simplicial objects.
This injectivity theorem will follow from applying Theorem \ref{thm:evinjnc} to
each component of a regular simplicial resolution.
A similar strategy using simplicial schemes was originally used by Kawamata
\cite{Kaw85} to prove related theorems (see Theorems \ref{thm:amb2014gnc} and
\ref{thm:amb2015gnc} below).
\begin{theorem}[The Esnault--Viehweg injectivity theorem for generalized normal
  crossing pairs]
  \label{thm:amb2012}
  Let \(f\colon (X,B) \to Y\) be a proper morphism of locally Noetherian schemes
  of equal characteristic zero such that \((X,B)\) is a generalized normal
  crossing pair and \(Y\) is locally quasi-excellent and has a
  dualizing complex \(\omega_Y^\bullet\).
  Let \(\sL\) be an invertible \(\cO_X\)-module such that 
  \[
    \sL^{\otimes r} \simeq \omega^{[r]}_{(X,B)}
  \]
  for an integer \(r \ge 1\) such that \(rB\) has
  integer coefficients.
  Let \(D\) be an effective Cartier divisor with support in \(\Supp(B)\).
  Then, the natural map
  \[
    R^if_*\sL \longrightarrow R^if_*\bigl(\sL(D)\bigr)
  \]
  is injective for every \(i\).
\end{theorem}
\begin{proof}
  Since the question is local, we may assume that \(Y\) is affine and Noetherian.
  Set \(\Sigma \coloneqq \Supp(B)\) and \(U \coloneqq X - \Sigma\).
  Since \(rB\) is Cartier, we have the isomorphism
  \[
    \varinjlim_{m \ge 0} H^i\bigl(X,\sL(mrB)\bigr)
    \overset{\sim}{\longrightarrow} H^i\bigl(U,\sL\rvert_U \bigr).
  \]
  The assertion in the theorem as \(D\) ranges over all possible \(D\) is
  therefore equivalent to the injectivity of the restriction map
  \[
    H^i(X,\sL) \longrightarrow H^i\bigl(U,\sL\rvert_U\bigr).
  \]
  \par Denote by \(\pi\colon \bar{X} \to X\) the normalization morphism.
  Following \cite[p.\ 746]{DBJ74}, we can construct a simplicial scheme
  \[
    X_\bullet = \Biggl(
      \begin{tikzcd}[cramped,column sep=1.475em]
        \cdots 
        \arrow[shift left=15pt]{r}
        \arrow[shift left=10pt,leftarrow]{r}
        \arrow[shift left=5pt]{r}
        \arrow[leftarrow]{r}
        \arrow[shift right=5pt]{r}
        \arrow[shift right=10pt,leftarrow]{r}
        \arrow[shift right=15pt]{r}
        & \bar{X} \times_X \bar{X} \times_X \bar{X}
        \arrow[shift left=10pt]{r}
        \arrow[shift left=5pt,leftarrow]{r}
        \arrow{r}
        \arrow[shift right=5pt,leftarrow]{r}
        \arrow[shift right=10pt]{r}
        & \bar{X} \times_X \bar{X} \arrow[shift left=5pt]{r}
        \arrow[leftarrow]{r}
        \arrow[shift right=5pt]{r}
        & \bar{X}
      \end{tikzcd}
    \Biggr)
  \]
  where \(X_n = (\bar{X}/X)^{n+1}\) for every \(n \ge 0\).
  By \cite[Lemma 10(1)]{Amb20}, since \((X,B)\) is a generalized normal
  crossings log pair, the augmentation morphism \(\varepsilon\colon X_\bullet
  \to X\) is a regular simplicial resolution in the sense of \cite[D\'efinition
  2.1.1]{DB81}.
  Set \(\Sigma_n \coloneqq \varepsilon_n^{-1}(\Sigma)\) and \(U_n \coloneqq X_n
  - \Sigma_n\).
  By Corollary \ref{cor:simplicial},
  which allows us to compute the
  Deligne--Du Bois complex using regular simplicial resolutions,
  \(\RR\varepsilon_*\cO_{X_\bullet}\) coincides with \(\underline{\Omega}_X^0\). 
  Since \(\underline{\Omega}_X^0\) is compatible with
  localization and completion
  (Corollary \ref{cor:delignedblocal}),
  the natural maps
  \[
    \cO_X \longrightarrow \RR \varepsilon_*\cO_{X_\bullet} \qquad \text{and}
    \qquad
    \cO_U \longrightarrow \RR \varepsilon_*\cO_{U_\bullet}
  \]
  are quasi-isomorphisms by \cite[Lemma 10(1)]{Amb20}.
  It therefore suffices to show that the restriction maps
  \[
    \alpha \colon H^i(X_\bullet,\sL_\bullet) \longrightarrow
    H^i\bigl(U_\bullet,\sL_\bullet\rvert_{U_\bullet}\bigr)
  \]
  are injective, where \(\sL_\bullet \coloneqq \varepsilon^*\sL\).
  Now the rest of the proof of \cite[Theorem 12]{Amb20} applies, replacing
  \cite[Lemma 11]{Amb20} with the local-global dual
  \citeleft\citen{Lip78}\citemid Theorem on p.\ 188\citeright\ (see
  \citeleft\citen{Mur}\citemid Lemma 5.2\citeright) of Theorem
  \ref{thm:evinjnc}.
\end{proof}
\subsection{The Tankeev--Koll\'ar injectivity theorem}
Next, we prove our version of the Tankeev--Koll\'ar injectivity theorem 
\citeleft\citen{Tan71}\citemid Proposition
1\citepunct \citen{Kol86}\citemid Theorem 2.2\citeright\ for
normal crossing pairs.
\begin{theorem}[The Tankeev--Koll\'ar injectivity theorem for normal crossing
  pairs on regular schemes]\label{lem:amb2013}
  Let \(f\colon (X,B) \to Y\) be a proper morphism of locally Noetherian schemes
  of equal characteristic zero such that \(X\) is regular and \(B\) is an
  \(\RR\)-divisor with normal crossing support and with coefficients in
  \([0,1]\).
  Let \(L\) be a Cartier divisor on \(X\) such that the \(\RR\)-Cartier divisor
  \(A = L - (K_X+B)\) is \(f\)-semiample.
  Let \(D\) be an effective Cartier divisor such that \(D \sim_\RR uA\) for some
  \(u > 0\) and \(D\) contains no strata of \((X,B)\).
  Then, for every point \(y \in Y\), the natural maps
  \[
    H^i_{f_y^{-1}(y)}\bigl(X_y,\cO_{X_y}(-L_y-D_y)\bigr)
    \longrightarrow
    H^i_{f_y^{-1}(y)}\bigl(X_y,\cO_{X_y}(-L_y)\bigr)
  \]
  are surjective for every \(i\), where the subscripts \((\,\cdot\,)_y\)
  denote base change to \(\Spec(\cO_{Y,y})\).
\end{theorem}
\begin{proof}
  As in the proof of Theorem \ref{thm:evinjsncred}, we may assume that \(X\) is
  integral and \(Y = \Spec(R)\) for an excellent integral local \(\QQ\)-algebra
  \((R,\fm)\) admitting a dualizing complex \(\omega_R^\bullet\).
  By perturbing the coefficients of \(B\)
  \cite[Lemma 5.2.13]{Fuj17}, we may replace \(\RR\) coefficients with
  \(\QQ\) coefficients.\smallskip
  \begin{step}\label{step:amb20131}
    The case when \(B+\varepsilon D\) has simple normal crossing support
    for some \(0 < \varepsilon < 1/u\).
  \end{step}
  We have
  \[
    L = K_X+B + \varepsilon D + (A-\varepsilon D) \sim_\QQ K_X+B + \varepsilon D
    + (1-\varepsilon u)A.
  \]
  Let \(n \ge 1\) be such that \(\cO_X(nA)\) is \(f\)-generated.
  By our Bertini theorem for \(f\)-generated Cartier
  divisors over integral Noetherian local \(\QQ\)-algebras
  \cite[Theorem 10.1 and Remark 10.2]{LM}, there exists an effective
  Cartier divisor \(S\) on \(X\) such that
  \[
    L \sim_\QQ K_X+B + \varepsilon D + \frac{1-\varepsilon u}{n} S
  \]
  and \(B+\varepsilon D + \frac{1-\varepsilon u}{n} S\) has simple normal
  crossing support.
  Note that \(D\) has support in \(\Supp(B+\varepsilon D + \frac{1-\varepsilon
  u}{n} S)\).
  We are now done by the local-global dual \citeleft\citen{Lip78}\citemid
  Theorem on p.\ 188\citeright\ (see \citeleft\citen{Mur}\citemid Lemma
  5.2\citeright) of Theorem \ref{thm:evinjsncred}.\smallskip
  \begin{step}
    The general case.
  \end{step}
  By log resolution of singularities \cite[Chapter I, \S3, Main
  theorems I\((n)\) and II\((N)\)]{Hir64}, there exists
  a resolution of singularities \(\mu\colon X' \to X\) such that \(\Exc(\mu)
  \cup \mu_*^{-1}B\)
  has simple normal crossing support.
  We can now follow the proof of \cite[Lemma 13]{Amb20}, replacing Step 1 of
  \cite[Lemma 13]{Amb20} with our version of Step \ref{step:amb20131} above.
\end{proof}
We can now extend Theorem
\ref{lem:amb2013} to the generalized normal crossing case.
\begin{theorem}[The Tankeev--Koll\'ar injectivity theorem for generalized normal
  crossing pairs]\label{thm:amb2014gnc}
  Let \(f\colon (X,B) \to Y\) be a proper morphism of locally Noetherian schemes
  of equal characteristic zero such that \((X,B)\) is a generalized normal
  crossing pair and \(Y\) is locally quasi-excellent and has a
  dualizing complex \(\omega_Y^\bullet\).
  Let \(\sL\) be an invertible \(\cO_X\)-module such that
  \[
    \sL^{\otimes r} \simeq \omega^{[r]}_{(X,B)} \otimes_{\cO_X} \sH
  \]
  for an integer \(r \ge 1\) such that
  \(rB\) has integer coefficients and for an invertible \(\cO_X\)-module \(\sH\)
  such that \(f^*f_*\sH \to \sH\) is surjective.
  Let \(s \in \Gamma(X,\sH)\) be a global section that is invertible at the
  generic point of every log canonical center of \((X,B)\) and let \(D\) be the
  effective Cartier divisor defined by \(s\).
  Then, the natural maps
  \[
    R^i f_*\sL \longrightarrow R^if_*\bigl(\sL(D)\bigr)
  \]
  are injective for every \(i\).
\end{theorem}
\begin{proof}
  Since the question is local on \(Y\), we may replace \(Y\) by
  \(\Spec(\cO_{Y,y})\) for each \(y \in Y\) to assume that \(Y\) is the spectrum
  of an excellent local \(\QQ\)-algebra.
  The proof is then the same as that of Theorem \ref{thm:amb2012} after replacing
  Theorem \ref{thm:evinjnc} with Theorem \ref{lem:amb2013}.
\end{proof}
\subsection{The Koll\'ar torsion-freeness theorem}
Next, we prove our version of the Koll\'ar torsion-freeness theorem 
\cite[Theorem 2.1\((i)\)]{Kol86} for
generalized normal crossing pairs.
\begin{theorem}[The Koll\'ar torsion-freeness theorem for generalized normal
  crossing pairs]
  \label{thm:amb2015gnc}
  Let \(f\colon (X,B) \to Y\) be a proper morphism of locally Noetherian schemes
  of equal characteristic zero such that \((X,B)\) is a generalized normal
  crossing pair and \(Y\) is locally quasi-excellent and has a
  dualizing complex \(\omega_Y^\bullet\).
  Let \(\sL\) be an invertible \(\cO_X\)-module such that
  \[
    \sL^{\otimes r} \simeq \omega^{[r]}_{(X,B)} \otimes_{\cO_X} \sH
  \]
  for an integer \(r \ge 1\) such that
  \(rB\) has integer coefficients
  and for an invertible \(\cO_X\)-module
  \(\sH\) such that \(f^*f_*\sH \to \sH\) is surjective.
  Then, every associated point of \(R^i f_*\sL\) is the \(f\)-image of the
  generic point of an irreducible component of \(X\) or a log canonical center
  of \((X,B)\).
\end{theorem}
\begin{proof}
  Working one connected component of \(X\) at a time, we may assume that \(X\)
  is connected by \cite[Proposition 3.1.7\((iii)\)]{EGAIV2}.
  We prove the contrapositive: If \(y \in Y\) is not the \(f\)-image of the
  generic point of any irreducible component of \(X\) or any log canonical
  center of \((X,B)\), then \(y \notin \Ass_{Y}(R^if_*\sL)\).
  Replacing \(Y\) with \(\Spec(\cO_{Y,y})\) we may assume that \(Y = \Spec(R)\)
  for a quasi-excellent local ring \((R,\fm)\).
  Moreover, we may replace \(R\) by \(R_\red\) to assume that \(R\) is reduced
  by \cite[Lemme 3.1.10.1]{EGAIV2}.
  \par If \(R\) is a field, then either \(R^if_*\sL = 0\), in which case
  \(\Ass_Y(R^if_*\sL) = \emptyset\) by \cite[Corollaire 3.1.5]{EGAIV2} and there
  is nothing to show, or \(R^if_*\sL \ne 0\) and
  \(\fm = (0)\) is the image of every irreducible component of \(X\).
  We may therefore assume that \(R\) is not a field and that \(\fm\) is
  not a minimal prime of \(R\).
  \par Let \(\{\fp_1,\fp_2,\ldots,\fp_n\}\) be the set of prime ideals in \(R\)
  consisting of the minimal primes of \(R\), the \(f\)-images of the generic points
  of the irreducible components of \(X\), and the \(f\)-images of the log
  canonical centers of \((X,B)\).
  Since \(\fm \ne \fp_i\) for any \(i\), there exists \(h \in \fm - \bigcup_i
  \fp_i\) by prime avoidance.
  Then, \(h\) is a nonzerodivisor on \(R\), and
  the pullback of \(h\) to \(X\) then defines an effective Cartier divisor
  \(D\) on \(X\) that does not contain any log canonical center of \((X,B)\).
  \par Now consider the commutative diagram
  \[
    \begin{tikzcd}
      R \rar{h \cdot -}\dar[equals] & R\dar{1 \mapsto h^{-1}}\\
      R \rar & R(h)
    \end{tikzcd}
  \]
  where \(R(h) = \Gamma(Y,\cO_Y(V(h)))\).
  Tensoring by \(R^if_*\sL\), we obtain the commutative diagram
  \[
    \begin{tikzcd}
      R^if_*\sL \rar{h\cdot-}\dar[equals] & R^if_*\sL\dar\\
      R^if_*\sL \rar[hook] & R^if_*\bigl(\sL(D)\bigr)
    \end{tikzcd}
  \]
  by the projection formula, where the bottom horizontal map is injective by
  Theorem \ref{thm:amb2014gnc}.
  By the commutativity of the diagram, the top horizontal map is injective, and
  hence \(h\) is a nonzerodivisor on \(R^if_*\sL\).
  Thus, \(\fm \notin \Ass_Y(R^if_*\sL)\) as claimed.
\end{proof}
\subsection{The Ohsawa--Koll\'ar vanishing theorem}
We prove our version of the Ohsawa--Koll\'ar vanishing theorem for
normal crossing pairs.
\begin{theorem}[The Ohsawa--Koll\'ar vanishing theorem for normal crossing pairs
  on regular schemes]\label{thm:amb2016}
  Let
  \[
    (X,B) \overset{f}{\longrightarrow} Y \overset{g}{\longrightarrow} Z
  \]
  be morphisms of locally Noetherian schemes
  of equal characteristic zero such that \(f\) is proper, \(g\) is locally projective,
  \(X\) is regular, \(B\) is an \(\RR\)-divisor with normal crossing support and
  coefficients in \([0,1]\), and \(Z\)
  is locally quasi-excellent and has a
  dualizing complex \(\omega_Z^\bullet\).
  Let \(L\) be a Cartier divisor on \(X\) such that \(L \sim_\RR K_X+B + f^*A\)
  for a \(g\)-ample Cartier divisor \(A\) on \(Y\).
  Then, we have
  \[
    R^pg_*R^qf_*\bigl(\cO_X(L)\bigr) = 0
  \]
  for all \(q\) and for all \(p \ne 0\).
\end{theorem}
\begin{proof}
  We proceed by induction on \(\dim(X)\).
  The \(\dim(X) = 0\) case is trivial, and hence it suffices to show the
  inductive case when \(\dim(X) > 0\).
  As in the proof of Theorem \ref{thm:evinjsncred}, we may assume that \(X\) is
  integral, \(f\) and \(g\) are surjective, and \(Z = \Spec(R)\)
  for an excellent integral local \(\QQ\)-algebra
  \((R,\fm)\) admitting a dualizing complex \(\omega_R^\bullet\).
  By perturbing the coefficients of \(B\)
  \cite[Lemma 5.2.13]{Fuj17}, we may replace \(\RR\) coefficients with
  \(\QQ\) coefficients.\smallskip
  \par We denote by \(\sL\) and \(\sA\) the invertible sheaves associated to
  \(L\) and \(A\), respectively.
  Since \(\sA\) is \(g\)-ample, we know that \(f^*\sA\) is \((g \circ
  f)\)-semi-ample by \cite[Lemma 2.11\((i)\)]{CT20}.
  Let \(m\) be a sufficiently large integer such that
  \begin{equation}\label{eq:amb2016serrevan}
    R^p g_*R^q f_*(\sL \otimes_{\cO_X} f^*\sA^{\otimes m}) = 0
  \end{equation}
  for all \(p \ne 0\), which exists by Serre vanishing, and such that
  \(f^*\sA^{\otimes m}\) is \((g \circ f)\)-generated.
  By our Bertini theorem for relatively generated invertible sheaves
  over integral Noetherian local \(\QQ\)-algebras
  \cite[Theorem 10.1 and Remark 10.2]{LM}, there exists an effective
  Cartier divisor \(S\) on \(Y\) such that setting \(T \coloneqq f^*S\),
  we have \(\cO_X(T) \simeq f^*\sA^{\otimes m}\), \(T\) is regular, and
  \(B+T\) has normal crossing support.
  \par Consider the short exact sequence
  \[
    0 \longrightarrow \sL \longrightarrow \sL(T) \longrightarrow
    \bigl(\sL(T)\bigr)\bigr\rvert_T \longrightarrow0.
  \]
  In the long exact sequence associated to pushing forward by \(f\),
  the connecting map
  \[
    R^qf_*\Bigl(\bigl(\sL(T)\bigr)\bigr\rvert_T\Bigr)
    \longrightarrow R^{q+1}f_*\sL
  \]
  is the zero
  map by Theorem \ref{thm:amb2015gnc} since the image is supported by \(T\), which
  contains no strata of \((X,B)\) and \(L \sim_\QQ K_X+B\) locally over \(Y\).
  This long exact sequence therefore breaks up into the short exact sequences
  \begin{align}
    0 \longrightarrow R^qf_*\sL
    \longrightarrow R^qf_*\bigl(\sL(T)\bigr) &\longrightarrow
    R^qf_*\Bigl(\bigl(\sL(T)\bigr)\bigr\rvert_T\Bigr)
    \longrightarrow0
    \label{eq:amb2016ses}
  \shortintertext{for every \(q\).\endgraf
  We now consider the long exact sequence}
    \cdots \longrightarrow R^pg_*R^qf_*\sL \longrightarrow
    R^pg_*R^qf_*\bigl(\sL(T)\bigr) &\longrightarrow
    R^pg_*R^qf_*\Bigl(\bigl(\sL(T)\bigr)\bigr\rvert_T\Bigr)
    \longrightarrow\cdots
    \label{eq:amb2016les}
  \end{align}
  associated to pushing forward the
  short exact sequence \eqref{eq:amb2016ses} by \(g\).
  For \(p \ge 2\), we first note that the \(\QQ\)-linear equivalence \(L+T
  \sim_\QQ K_X+B+T+f^*A\) restricts to the \(\QQ\)-linear equivalence
  \[
    (L+T)\rvert_T \sim_\QQ K_T+B_T+\bigl(f\rvert_T\bigr)^*\bigl(A\rvert_S\bigr).
  \]
  Since the long exact sequence \eqref{eq:amb2016les} implies
  \begin{equation}
    R^pg_*R^qf_*\bigl(\sL(T)\bigr) \simeq
    R^pg_*R^qf_*(\sL \otimes_{\cO_X} f^*\sA^{\otimes m}) = 0
    \label{eq:amb2016applyserrevan}
  \end{equation}
  for all \(p \ne 0\) by the projection formula and \eqref{eq:amb2016serrevan},
  the inductive hypothesis and 
  the long exact sequence \eqref{eq:amb2016les} imply
  \[
    0 = R^{p-1}g_*R^qf_*\Bigl(\bigl(\sL(T)\bigr)\bigr\rvert_T\Bigr)
    \overset{\sim}{\longrightarrow}
    R^{p}g_*R^qf_*\sL.
  \]
  \par It remains to show \(R^{1}g_*R^qf_*\sL = 0\).
  Consider the morphism connecting the two Leray spectral sequences
  \begin{align*}
    E_2^{p,q} = R^pg_*R^qf_*\sL
    &\Rightarrow R^{p+q}(g \circ f)_*\sL,\\
    E_2^{p,q} = R^pg_*R^qf_*\bigl(\sL(T)\bigr) &\Rightarrow R^{p+q}(g \circ
    f)_*\bigl(\sL(T)\bigr).
  \end{align*}
  Since the terms of the two spectral sequences vanish for all \(p \notin
  \{0,1\}\), we obtain the commutative diagram
  \[
    \begin{tikzcd}
      R^{1+q}(g \circ f)_*\sL \rar[hook]{\beta} & R^{1+q}(g \circ
      f)_*\bigl(\sL(T)\bigr)\\
      R^1g_*R^1f_*\sL \rar\uar[hook] &
      R^1g_*R^qf_*\bigl(\sL(T)\bigr)\uar[hook]
    \end{tikzcd}
  \]
  with injective vertical arrows.
  The map \(\beta\) is injective by Theorem \ref{thm:amb2014gnc} since \(T\)
  contains no strata of \((X,B)\).
  By the commutativity of the diagram, we see that
  \[
    R^{1}g_*R^qf_*\sL \hooklongrightarrow
    R^1g_*R^qf_*\bigl(\sL(T)\bigr) = 0
  \]
  where the vanishing on the right follows from
  \eqref{eq:amb2016applyserrevan}.
\end{proof}
Finally, we extend our Ohsawa--Koll\'ar vanishing theorem \ref{thm:amb2016}
to the generalized normal crossing case.
\begin{theorem}[The Ohsawa--Koll\'ar vanishing theorem for generalized normal
  crossing pairs]
  \label{thm:amb2016gnc}
  Let
  \[
    (X,B) \overset{f}{\longrightarrow} Y \overset{g}{\longrightarrow} Z
  \]
  be morphisms of locally Noetherian schemes
  of equal characteristic zero such that \(f\) is proper, \(g\) is locally projective,
  \((X,B)\) is a generalized normal crossing pair, and \(Z\)
  is locally quasi-excellent and has a
  dualizing complex \(\omega_Z^\bullet\).
  Let \(\sL\) be an invertible \(\cO_X\)-module such that
  \[
    \sL^{\otimes r} \simeq \omega^{[r]}_{(X,B)} \otimes_{\cO_X} f^*\sA
  \]
  for an integer \(r \ge 1\) such that
  \(rB\) has integer coefficients and \(\sA\) is a \(g\)-ample invertible
  \(\cO_Y\)-module.
  Then, we have
  \[
    R^pg_*R^qf_*\sL = 0
  \]
  for all \(q\) and for all \(p \ne 0\).
\end{theorem}
\begin{proof}
  Denote by \(\pi\colon \bar{X} \to X\) the normalization morphism.
  As in the proof of Theorem \ref{thm:amb2012}, 
  we can construct a simplicial scheme
  \[
    X_\bullet = \Biggl(
      \begin{tikzcd}[cramped,column sep=1.475em]
        \cdots 
        \arrow[shift left=15pt]{r}
        \arrow[shift left=10pt,leftarrow]{r}
        \arrow[shift left=5pt]{r}
        \arrow[leftarrow]{r}
        \arrow[shift right=5pt]{r}
        \arrow[shift right=10pt,leftarrow]{r}
        \arrow[shift right=15pt]{r}
        & \bar{X} \times_X \bar{X} \times_X \bar{X}
        \arrow[shift left=10pt]{r}
        \arrow[shift left=5pt,leftarrow]{r}
        \arrow{r}
        \arrow[shift right=5pt,leftarrow]{r}
        \arrow[shift right=10pt]{r}
        & \bar{X} \times_X \bar{X} \arrow[shift left=5pt]{r}
        \arrow[leftarrow]{r}
        \arrow[shift right=5pt]{r}
        & \bar{X}
      \end{tikzcd}
    \Biggr)
  \]
  where \(X_n = (\bar{X}/X)^{n+1}\) for every \(n \ge 0\).
  As before,
  by \cite[Lemma 10(1)]{Amb20} and Corollaries \ref{cor:delignedblocal} and
  \ref{cor:simplicial},
  since \((X,B)\) is a generalized normal crossings log pair,
  the natural maps
  \[
    \cO_X \longrightarrow \RR \varepsilon_*\cO_{X_\bullet} \qquad \text{and}
    \qquad
    \cO_U \longrightarrow \RR \varepsilon_*\cO_{U_\bullet}
  \]
  are quasi-isomorphisms where \(\varepsilon\colon X_\bullet \to X\) is the
  augmentation morphism, \(\Sigma_n \coloneqq \varepsilon_n^{-1}(\Sigma)\),
  and \(U_n \coloneqq X_n - \Sigma_n\).
  It therefore suffices to show that
  \[
    R^pg_*R^q(f \circ \varepsilon)_*\sL_\bullet = 0
  \]
  for all \(q\) and for all \(p \ne 0\), where \(\sL_\bullet \coloneqq
  \varepsilon^*\sL\).
  As in the proof of \cite[Theorem 12]{Amb20}, the generalized normal crossing
  pair \((X,B)\) induces a pair \((X_n,B_n)\) such that \(X_n\) is regular and
  \(B_n\) has normal crossing support and coefficients in \([0,1]\) for every
  \(n\).
  We can then apply Theorem \ref{thm:amb2016}.
\end{proof}
\subsection{Extensions to other categories of spaces}
\label{sect:injectivityothercats}
We now extend our injectivity and torsion-freeness theorems to the other
categories of spaces.
For simplicity, we assume the morphism \(f\) is locally projective.
\begin{customthm}{A}\label{thm:maininj}
  Let \(f\colon (X,B) \to Y\) be a locally projective morphism of spaces of one
  of the types
  \((\ref{setup:introalgebraicspaces})\textnormal{--}(\ref{setup:introadicspaces})\)
  such that \((X,B)\) is a generalized normal crossing pair.
  In case \((\ref{setup:introalgebraicspaces})\) (resp.\
  \((\ref{setup:introformalqschemes})\)), suppose that \(Y\) is locally
  quasi-excellent of equal characteristic zero and has a dualizing complex
  (resp.\ \(c\)-dualizing complex) \(\omega_Y^\bullet\).
  In cases \((\ref{setup:introberkovichspaces})\),
  \((\ref{setup:introrigidanalyticspaces})\), and
  \((\ref{setup:introadicspaces})\), suppose that \(k\) is of characteristic
  zero.
  Let \(\sL\) be an invertible \(\cO_X\)-module.
  \begin{enumerate}[label=\((\roman*)\),ref=\roman*]
    \item\label{thm:maininjesvi}
      \textnormal{(The Esnault--Viehweg injectivity theorem)}
      Suppose that 
      \[
        \sL^{\otimes r} \simeq \omega^{[r]}_{(X,B)}
      \]
      for an integer \(r \ge 1\) such that \(rB\) has
      integer coefficients.
      Let \(D\) be an effective Cartier divisor with support in \(\Supp(B)\).
      Then, the natural map
      \[
        R^if_*\sL \longrightarrow R^if_*\bigl(\sL(D)\bigr)
      \]
      is injective for every \(i\).
    \item\label{thm:maininjtankol}
      \textnormal{(The Tankeev--Koll\'ar injectivity theorem)}
      Suppose that
      \[
        \sL^{\otimes r} \simeq \omega^{[r]}_{(X,B)} \otimes_{\cO_X} \sH
      \]
      for an integer \(r \ge 1\) such that
      \(rB\) has integer coefficients and for an invertible \(\cO_X\)-module \(\sH\)
      such that \(f^*f_*\sH \to \sH\) is surjective.
      Let \(s \in \Gamma(X,\sH)\) be a global section that is invertible at the
      generic point of every log canonical center of \((X,B)\) and let \(D\) be the
      effective Cartier divisor defined by \(s\).
      Then, the natural maps
      \[
        R^i f_*\sL \longrightarrow R^if_*\bigl(\sL(D)\bigr)
      \]
      are injective for every \(i\).
    \item\label{thm:maininjtf} \textnormal{(The Koll\'ar torsion-freeness theorem)}
      Suppose that
      \[
        \sL^{\otimes r} \simeq \omega^{[r]}_{(X,B)}
      \]
      for an integer \(r \ge 1\) such that
      \(rB\) has integer coefficients
      and for an invertible \(\cO_X\)-module
      \(\sH\) such that \(f^*f_*\sH \to \sH\) is surjective.
      In case \((\ref{setup:introalgebraicspaces})\), suppose that \(Y\) is an
      algebraic space.
      Then, every associated subspace of \(R^i f_*\sL\) is the \(f\)-image of an
      irreducible component of \(X\) or of a log canonical center of \((X,B)\).
    \item\label{thm:maininjohskol}
      \textnormal{(The Ohsawa--Koll\'ar vanishing theorem)}
      Let \(g\colon Y \to Z\) be another locally projective morphism, where
      \(Z\) satisfies the same hypotheses as \(Y\) above.
      Suppose that
      \[
        \sL^{\otimes r} \simeq \omega^{[r]}_{(X,B)} \otimes_{\cO_X} f^*\sA
      \]
      for an integer \(r \ge 1\) such that
      \(rB\) has integer coefficients and \(\sA\) is a \(g\)-ample invertible
      \(\cO_Y\)-module.
      Then, we have
      \[
        R^pg_*R^qf_*\sL = 0
      \]
      for all \(q\) and for all \(p \ne 0\).
  \end{enumerate}
\end{customthm}
\begin{proof}
  Since all four statements are local, we may replace \(Y\) by an affinoid
  subdomain to assume that \(Y\) is affinoid and \(f\) is projective.
  In case \((\ref{setup:introalgebraicspaces})\), we obtain a projective
  morphism of schemes after smooth base change, and the dualizing complex
  \(\omega_Y^\bullet\) is compatible with this base change \cite[Definition
  2.16]{AB10}.
  In the other cases, we apply the GAGA theorems from
  \citeleft\citen{EGAIII1}\citemid Proposition 5.1.2\citepunct
  \citen{AT19}\citemid Theorem C.1.1\citepunct \citen{Poi10}\citemid
  Th\'eor\`eme A.1\((i)\)\citepunct \citen{Kop74}\citemid Folgerung
  6.6\citepunct \citen{Hub07}\citemid Corollary 6.4\citeright, which say that
  the relevant cohomology modules are compatible with GAGA, and \cite[\S\S23--25]{LM},
  which says that Grothendieck duality and dualizing complexes are compatible
  with GAGA.
  Note that the sheaves \(\omega_{(X,B)}^{[r]}\) and the generalized normal
  crossings conditions are compatible with GAGA
  because these notions are defined in terms
  of completions of local rings.
  In case \((\ref{setup:introalgebraicspaces})\), instead of constructing
  \(\omega_{(X,B)}^{[r]}\) on \(X\), the hypothesis of each statement is 
  that \(\sL\) is an invertible
  \(\cO_X\)-module such that after replacing \(Y\) by a smooth covering by a
  scheme, we have an isomorphism between \(\sL^{\otimes r}\) and (the twist of)
  \(\omega^{[r]}_{(X,B)}\).
  Now that we have reduced to the scheme case, we may apply Theorems
  \ref{thm:amb2012}, \ref{thm:amb2014gnc}, \ref{thm:amb2015gnc}, and
  \ref{thm:amb2016gnc}.
\end{proof}

\subsection{Derived splinters and the vanishing property for maps on Tor}
We now give an application of our version of the Koll\'ar torsion-freeness
theorem (Theorem \ref{thm:amb2015gnc}) to commutative algebra.
Specifically, we extend Ma's characterization \cite[Theorem 5.5]{Ma18}
of derived splinters of equal characteristic zero as the class
of \(\QQ\)-algebras satisfying Hochster and Huneke's vanishing conjecture for
maps of Tor \cite[Theorem 4.1 and (4.4)]{HH95} to the complete local case.
Ma's result applies to local rings essentially of finite type
over fields of characteristic zero.
\par As mentioned in \S\ref{sect:intro}, 
our proof recovers Hochster and Huneke's vanishing theorem for maps of Tor
for all regular domains of equal characteristic zero \cite[Theorem 4.1]{HH95}.
In contrast with \cite{HH95}, this proof does not use reduction modulo \(p\).
See \cite[Remark 5.6(2)]{Ma18}.
\begin{theorem}[{cf.\ \cite[Theorem 5.5]{Ma18}}]\label{thm:ma18}
  Let \(S\) be a Noetherian complete local domain of equal characteristic zero.
  The following are equivalent:
  \begin{enumerate}[label=\((\arabic*)\),ref=\arabic*]
    \item\label{thm:ma18i}
      \(S\) satisfies the \textsl{vanishing conditions for maps of Tor},
      i.e., for all maps \(A \to R \to S\) of \(\QQ\)-algebras
      such that \(A\) is a regular domain and \(A \to R\) is a module-finite
      torsion-free extension, the natural map
      \[
        \Tor_i^A(M,R) \longrightarrow \Tor_i^A(M,S)
      \]
      vanishes for every \(A\)-module \(M\) and every \(i \ge 1\).
    \item\label{thm:ma18ii} \(S\) is a derived splinter.
    \item\label{thm:ma18iii} For every regular local ring \(A\) with \(S = A/P\) and every
      module-finite torsion-free extension \(A \to B\) with \(Q \in \Spec(B)\)
      lying over \(P\), the map \(P \to Q\) splits as a map of \(A\)-modules.
  \end{enumerate}
\end{theorem}
\begin{proof}
  First, \((\ref{thm:ma18i}) \Leftrightarrow (\ref{thm:ma18iii})\) holds by
  \cite[Theorem 4.3]{Ma18}.
  Moreover, derived splinters are the same as pseudo-rational rings in this
  context by \cite[Theorem 9.5]{Mur}, and hence
  \((\ref{thm:ma18i}) \Rightarrow (\ref{thm:ma18ii})\) follows from
  \cite[Proposition 3.4]{Ma18}.
  \par It remains to show \((\ref{thm:ma18ii}) \Rightarrow
  (\ref{thm:ma18i})\).
  By \cite[(4.5)]{HH95}, we may assume that \(A\) is local, \(R\) is a domain,
  \(A \to S\) is surjective, and \(M\) is a finitely generated \(A\)-module.
  Since \(A\) is regular local, the module \(M\) has finite projective
  dimension.
  To show that \((\ref{thm:ma18i})\) holds, it now suffices to show that
  \cite[Corollary 5.4]{Ma18} holds after replacing the local domain \((A,\fm)\)
  essentially of finite type over a field of characteristic zero
  in the statement of \cite[Corollary
  5.4]{Ma18} with a Noetherian complete local domain \((A,\fm)\) of equal
  characteristic zero.
  The only place where this assumption is used in the proof of \cite[Theorem
  5.2]{Ma18}, which (in equal characteristic zero) uses resolutions of
  singularities \cite{Hir64}
  and Koll\'ar's vanishing theorem \cite[Theorem
  2.1\((ii)\)]{Kol86}.
  Resolutions of singularities exist in this context by \cite[Chapter I, \S3,
  Main theorem I\((n)\)]{Hir64}.
  Koll\'ar's vanishing theorem holds by our version of the Koll\'ar
  torsion-freeness theorem: If \(f\colon X \to Y\)
  is a proper surjective morphism of integral regular excellent schemes of equal
  characteristic zero and \(Y\) has a dualizing complex \(\omega_Y^\bullet\), then
  \(R^if_*\omega_X = 0\) for all \(i\) greater than the generic fiber dimension
  since \(R^if_*\omega_X\) must be torsion-free by Theorem \ref{thm:amb2015gnc}.
\end{proof}

\section{A weight filtration on
cohomology\texorpdfstring{\except{toc}{\nopagebreak\\}}{} of schemes and compactifiable
rigid analytic spaces}\label{sect:weights}
In this section, we construct a weight filtration on (pro-)\'etale cohomology of
schemes and compactifiable rigid analytic spaces.
The case of varieties is due to Deligne \cite{Del74} and the case of
compactifiable complex analytic spaces is due to Cirici and Guill\'en
\cite{CG14}.
The key point is that as in \cite{CG14},
the data necessary to construct the weight filtration for
complex varieties \(X\) in \cite{Del74} is all contained in a hyperresolution for
\(X\), and hence we can apply our extension criterion
(Theorems \ref{thm:gna215} and \ref{thm:cg39}) to construct an analogous weight
filtration for schemes and for rigid analytic spaces.\medskip
\par To construct the weight filtration, Deligne used Hodge theory
\cite[(8.1.19)]{Del74}.
Later, Gillet and Soul\'e constructed the weight filtration using resolutions of
singularities and algebraic \(K\)-theory \cite[Theorem 3]{GS95}.
However, the extension criterion of Guill\'en and Navarro Aznar
\cite[Th\'eor\`eme 2.1.5]{GNA02} gives a new approach to constructing the weight
filtration that only uses resolutions of singularities and either the
Chow--Hironaka lemma (in \cite{GNA02}) or the weak factorization theorem (in
this paper).
This observation is due to Totaro, who constructed the weight filtration on
singular cohomology with compact support on complex or real analytic spaces
\cite[Theorems 2.2 and 2.3]{Tot02} using \cite[Th\'eor\`eme 2.1.5]{GNA02}.
Using similar ideas,
McCrory and Parusi\'{n}ski constructed the weight filtration on Borel--Moore homology
\cite[Theorem 1.1]{McP11} and on compactly supported singular homology
\cite[Theorem 8.1]{McP14} on real
algebraic varieties, both with \(\ZZ/2\ZZ\) coefficients.
Cirici and Guill\'en showed that a similar construction (with the version of the
Guill\'en--Navarro Aznar extension criterion considered in \S\ref{sect:cg14})
can be used to construct the weight filtration on compactifiable complex
analytic spaces \cite[Theorem 6.1]{CG14}.\medskip
\par We follow the approach in \cite[\S6]{CG14}.
We work with the (pro-)\'etale topology and (pro-)\'etale cohomology.
See \S\ref{sect:etaletop} for references discussing the (pro-)\'etale topology.
\subsection{Conventions}
We denote by \(\Sp\) a subcategory of a small category of reduced spaces of one of the types
\((\ref{setup:introformalqschemes})\) or \((\ref{setup:introadicspaces})\)
that is essentially stable under fiber products, immersions, and proper
morphisms.
We will restrict to the subcases when \(\Sp\) consists of ordinary schemes
and when \(k\) is a \(p\)-adic field, respectively.
We will assume \(\Sp\) is chosen according to the strategy outlined in
Remark \ref{rem:smallspaces}.
We denote by \(\Sp_\reg\) the subcategory of \(\Sp\) consisting of regular spaces.
See Definition \ref{def:sp2} for definitions of \(\Sp^2\), \(\Sp_{\reg}^2\),
\(\Sp_\comp^2\), and \(\Sp^2_{\reg,\comp}\).
We denote by \(I\) a small category.
While we will remind the reader of these assumptions in statements of results in
this section, we will not restate these assumptions in definitions or remarks.
\par We will also continue to use the terminology \textsl{\(I\)-space} from
Definition \ref{def:ispaces} to refer to objects in \(\Diag_I(\Sp)\).
\subsection{(Pro-)\'etale cohomology of a pair
\texorpdfstring{\except{toc}{(\emph{X},\emph{U})}\for{toc}{\((X,U)\)}}{(X,U)}}
Following the strategy of \cite[\S6]{CG14}, we can use Theorem \ref{thm:cg39}
to construct a weight
filtration on (pro-)\'etale cohomology of schemes or adic spaces over \(p\)-adic
fields.
The case of germs of complex analytic spaces (case
\((\ref{setup:introcomplexanalyticgerms})\)) also holds by the proof in
\cite{CG14}.
\par For the statement below, for an object \((X,U)\) of \(\Sp_{\reg,\comp}^2\),
we denote by \(j\colon U \hookrightarrow X\) the inclusion morphism.
We first discuss the \'etale case.
Consider the \'etale sheaf \(\bm{\mu}_{n,U}\) on \(U\) where \(\mu_n\) is the
multiplicative group of \(n\)-th roots of unity for an integer \(n > 0\).
We can then compute \(\RR j_*\bm{\mu}_{n,U}\) using the Godement resolution as
defined in \cite[Expos\'e XVII, D\'efinition 4.2.2]{SGA43}, where we note that
the \'etale sites for schemes
\cite[\href{https://stacks.math.columbia.edu/tag/040S}{Tag
040S}]{stacks-project} and for adic spaces \cite[Proposition
2.5.5]{Hub96} have enough points.
Next, consider the filtered complex of sheaves
\[
  \bigl( \RR j_*\bm{\mu}_{n,U},\tau_{\le\bullet} \bigr)
\]
where
\(\tau_{\le\bullet}\) is the canonical filtration.
For each integer \(n > 0\), taking the right derived functor of global sections yields a
functor
\[
  \begin{tikzcd}[cramped,row sep=0,column sep=1.475em]
    \cW_n\colon&[-2.125em] \Sp_{\reg,\comp}^2 \rar & \D^+_1(\FF \Ab)\\
    & (X,U) \rar[mapsto] & \RR\Gamma\bigl(X_\et,\bigl( \RR
      j_*\bm{\mu}_{n,U},\tau_{\le \bullet}
    \bigr)\bigr)
  \end{tikzcd}
\]
that is \(\Phi\)-rectified since it factors through \(C^+(\FF A\mhyphen\Mod)\)
\cite[p.\ 34]{GNA02}.
Here,
\[
  \D^+_1(\FF \Ab) \coloneqq \C^+(\FF \Ab)[\cE_1^{-1}]
\]
denotes the category of filtered complexes of
Abelian groups localized with respect to the class of filtered morphisms that
induce isomorphisms on the \(E_1\) page of the spectral sequences associated to
the filtrations
\cite[Definition 2.2]{CG14}.
\par In the pro-\'etale topology, we can work directly with the constant sheaf
\(\underline{\ZZ_l}_{\kern1ptU}\) since the pro-\'etale site has enough points
\citeleft\citen{stacks-project}\citemid
\href{https://stacks.math.columbia.edu/tag/0992}{Tag 0992}\citepunct
\citen{Sch13}\citemid Proposition 3.13\citeright.
\begin{theorem}
  \label{thm:weightfiltration}
  Suppose that \(\Sp\) is a subcategory of a small category of
  finite-dimensional reduced spaces of type
  \((\ref{setup:introformalqschemes})\) or \((\ref{setup:introadicspaces})\)
  that is essentially stable under fiber products, immersions, and proper
  morphisms.
  In case \((\ref{setup:introformalqschemes})\), suppose that all objects in \(\Sp\) are
  ordinary schemes that are Noetherian, quasi-excellent, and of equal
  characteristic zero.
  In case
  \((\ref{setup:introadicspaces})\), suppose that
  \(k\) is a \(p\)-adic field of characteristic zero.
  Then, for every integer \(n > 0\) such that \(p \nmid n\),
  there exists a \(\Phi\)-rectified functor
  \[
    \cW_n'\colon \Sp_\infty \longrightarrow \D^+_1(\FF \Ab)
  \]
  satisfying the following properties:
  \begin{enumerate}[label=\((\roman*)\),ref=\roman*]
    \item\label{thm:weightfiltrationi}
      If \(U_\infty\) is an object of \(\Sp_\infty\), then
      \(\mathbf{h}^i(\cW'_n(U_\infty)) \simeq H^i_\et(U,\bm{\mu}_{n,U})\) for
      every \(i\).
    \item If \((X,U)\) is an object of \(\Sp_{\reg,\comp}^2\) then
      \(\cW_n'(U_\infty) \simeq \cW_n(X,U)\).
    \item\label{thm:weightfiltrationiii}
      For every \(p,q\in \ZZ\) and every acyclic square
      \[
        \begin{tikzcd}
          \tilde{Y}_\infty \rar \dar & \tilde{X}_\infty \dar\\
          Y_\infty \rar & X_\infty
        \end{tikzcd}
      \]
      in \(\Sp_\infty\), there is a long exact sequence
      \[
        \cdots \longrightarrow E_2^{*,q}\bigl(\cW'_n(X_\infty)\bigr)
        \longrightarrow E_2^{*,q}\bigl(\cW'_n(\tilde{X}_\infty)\bigr) \oplus
        E_2^{*,q}\bigl(\cW'_n(Y_\infty)\bigr) \longrightarrow
        E_2^{*,q}\bigl(\cW'_n(\tilde{Y}_\infty)\bigr) \longrightarrow \cdots.
      \]
  \end{enumerate}
  The same results hold for pro-\'etale cohomology with \(\ZZ_l\) or \(\QQ_l\)
  coefficients when \(l \ne p\).
\end{theorem}
\begin{proof}
  By \cite[Theorem 2.8]{CG14} (see Example
  \ref{ex:cohdescent}\((\ref{ex:erqisos})\)), the triple \((\C^+(\FF
  \Ab),\cE_1,\bfs^1)\) is a category of cohomological descent.
  By Theorem \ref{thm:cg39}, to show
  \((\ref{thm:weightfiltrationi})\textnormal{--}(\ref{thm:weightfiltrationiii})\),
  it suffices to show the functor \(\cW_n\) satisfies
  conditions \((\ref{thm:cg39f1})^2\) and \((\ref{thm:cg39f2})^2\) in the
  statement of Theorem \ref{thm:cg39}.
  Condition \((\ref{thm:cg39f1})^2\) holds by definition of the global sections
  functor, and hence it remains to show \((\ref{thm:cg39f2})^2\).
  \par We want to show that for every acyclic square \((X_\bullet,U_\bullet)\)
  augmented over \((X,U)\) in \(\Sp^2_{\reg,\comp}\), the simple
  \(\bfs^1(\cW_n(X_\bullet,U_\bullet))\) is \(E_1\)-acyclic.
  By \cite[Propositon 2.7]{CG14}, for every cubical codiagram of filtered
  complexes \(K^\bullet\), we have a quasi-isomorphism
  \[
    E_1^{*,q}\bigl(\bfs(K^\bullet)\bigr) \overset{\sim}{\longleftrightarrow} \bfs
    E_1^{*,q}(K^\bullet).
  \]
  Thus, it suffices to show that for every fixed integer \(q\), the sequence
  \[
    \mathclap{\cdots \longrightarrow E_2^{*,q}\bigl(\cW_n(X,U)\bigr)
    \longrightarrow E_2^{*,q}\bigl(\cW_n(\tilde{X},\tilde{U})\bigr) \oplus
    E_2^{*,q}\bigl(\cW_n(Y,U \cap Y)\bigr) \longrightarrow
    E_2^{*,q}\bigl(\cW_n(\tilde{Y},\tilde{U} \cap \tilde{Y})\bigr) \longrightarrow
    \cdots}
  \]
  is exact.
  By definition of \(\cW_n\), the spectral sequence \(E_2^{*,q}(\cW_n(X,U))\) is the
  Leray spectral sequence for the inclusion \(j \colon U \hookrightarrow
  X\) (with a shift) and therefore coincides with the Gysin complex
  \(G^q(X,U)^*\) (see \citeleft\citen{GNA02}\citemid p.\ 81\citepunct
  \citen{CG14}\citemid 4.3\citeright) for every \(q\).
  \par We can therefore apply \cite[Proposition 4.5]{CG14} and its proof with the
  following changes.
  In the proof of \cite[Proposition 4.5(1)]{CG14}, it suffices to replace
  Poincar\'e--Verdier duality and the six functor formalism used in the proof of
  \cite[Proposition 4.1]{CG14} with the corresponding results from \cite{SGA43}
  (which also applies to
  formal schemes by \cite[Expos\'e I, Corollaire 8.4]{SGA1new})
  in case \((\ref{setup:introformalqschemes})\) and from
  \cite{Hub96} in case
  \((\ref{setup:introadicspaces})\) in the \'etale topology.
  In the pro-\'etale topology, we instead use \cite[\S6]{BS15} and
  \cite[\S3]{Sch13}, respectively.
  In the proof of \cite[Proposition 4.5(2)]{CG14}, we make the following changes
  in the base case \((r,s) = (1,0)\):
  \begin{itemize}
    \item
      In case \((\ref{setup:introformalqschemes})\),
      we apply \cite[Expos\'e VII, Th\'eor\`emes 2.2.1 and
      8.1]{SGA5} over \(U\) and then apply \(H^i(X,\RR j_*(\,\cdot\,))\)
      instead of using \cite[Proposition 4.2]{CG14}.
      Note that in the pro-\'etale case, we cannot reduce to the \'etale case
      directly.
      Instead, we follow the first steps of the proof of \cite[Expos\'e VII,
      Th\'eor\`eme 8.1]{SGA5} in the pro-\'etale topology using
      the fact that the pro-\'etale site has enough points 
      \citeleft\citen{stacks-project}\citemid
      \href{https://stacks.math.columbia.edu/tag/0992}{Tag 0992}\citeright\ 
      and proper base
      change \cite[Lemma 6.7.5]{BS15} to reduce to the computation for
      projective space in \cite[Expos\'e VII, Lemme 8.2]{SGA5}.
      This computation can now be reduced to the \'etale case using
      \cite[Proposition 5.6.2]{BS15}.
    \item In case 
      \((\ref{setup:introadicspaces})\), we need a rigid analytic analogue of
      \cite[Expos\'e VII, Th\'eor\`emes 2.2.1 and 8.1]{SGA5} over \(U\), which
      can then replace \cite[Proposition 4.2]{CG14} after applying \(H^i(X,\RR
      j_*(\,\cdot\,))\).
      As in the previous case,
      we follow the first steps of the proof of \cite[Expos\'e VII,
      Th\'eor\`eme 8.1]{SGA5} using the fact that the \'etale site and
      pro-\'etale site have enough points \citeleft\citen{Hub96}\citemid
      Proposition 2.5.5\citepunct \citen{Sch13}\citemid Proposition
      3.13\citeright\ and proper base change \citeleft\citen{Hub96}\citemid
      Theorem 4.4.1\citepunct \citen{Sch}\citemid Theorem 19.2\citeright.
      This reduces to a computation of (pro-)\'etale cohomology of projective
      space, which can be done algebraically by \citeleft\citen{Hub96}\citemid
      Theorem 3.8.1\citepunct \citen{Sch}\citemid Proposition
      27.5\citeright.\qedhere
  \end{itemize}
\end{proof}

\addtocontents{toc}{\protect\medskip}
\bookmarksetup{startatroot}
\section*{Acknowledgments}
Theorem \ref{thm:maininj} arose from a question of Osamu Fujino.
I would like to thank him for encouraging me to write this paper.
I am grateful to Osamu Fujino and Haoyang Guo for helpful comments on a
preliminary draft of this paper.

\end{document}